\documentclass[onefignum,onetabnum]{siamart250211}

\usepackage{subfigure,threeparttable}
\usepackage [latin1]{inputenc}
\usepackage{booktabs,multirow}
\usepackage{lipsum,enumerate,enumitem}
\usepackage{amsfonts,mathrsfs,stmaryrd}
\usepackage{graphicx}
\usepackage{epstopdf}
\usepackage{algorithmic}
\ifpdf
\DeclareGraphicsExtensions{.eps,.pdf,.png,.jpg}
\else
\DeclareGraphicsExtensions{.eps}
\fi

\newsiamremark{remark}{Remark}
\newsiamremark{hypothesis}{Hypothesis}
\crefname{hypothesis}{Hypothesis}{Hypotheses}
\newsiamthm{claim}{Claim}

\def\rmk{{\mathrm{k}}}
\def\ck{{\scriptstyle{K}}}

\newcommand{\myvec}[1]{\boldsymbol{#1}}
\newcommand{\zd}{\,\mathrm{d}}
\newcommand{\abs}[1]{\left|#1\right|}
\newcommand{\absb}[1]{\big|#1\big|}
\newcommand{\absB}[1]{\Big|#1\Big|}
\newcommand{\abst}[1]{|#1|}
\newcommand{\bra}[1]{\left(#1\right)}
\newcommand{\brab}[1]{\big(#1\big)}
\newcommand{\braB}[1]{\Big(#1\Big)}
\newcommand{\brat}[1]{(#1)}
\newcommand{\kbra}[1]{\left[#1\right]}
\newcommand{\kbrab}[1]{\big[#1\big]}
\newcommand{\kbraB}[1]{\Big[#1\Big]}
\newcommand{\kbrat}[1]{[#1]}
\newcommand{\myinner}[1]{\left\langle#1\right\rangle}
\newcommand{\myinnert}[1]{\langle#1\rangle}
\newcommand{\myinnerb}[1]{\big\langle#1\big\rangle}

\newcommand{\mynorm}[1]{\left\|#1\right\|}
\newcommand{\mynormb}[1]{\big\|#1\big\|}

\usepackage{cite}

\headers{Stability of implicit-explicit multistep methods}{H.-L. Liao, C. Quan, T. Tang and T. Zhou}

\title{A semi-generating function approach to the stability of implicit-explicit multistep methods for nonlinear parabolic equations\thanks{Submitted to the editors \today.
		\funding{This work is supported by the National Natural Science Foundation of China under grants 12471383, 12271241, 11731006, 12288201 and K20911001, Basic Research Program of Jiangsu Province under grant BK20252027,
			Ministry of Education Key Laboratory of NSLSCS under grant 202501, Guangdong Basic and Applied Basic Research Foundation under grant 2023B1515020030,  Shenzhen Science and Technology Innovation Program under grant JCYJ20230807092402004, and Hetao Shenzhen-Hong Kong Science and Technology  Innovation Cooperation Zone Project under grant HZQSWS-KCCYB-2024016.}}}

\author{ Hong-lin Liao\thanks{Corresponding author. ORCID 0000-0003-0777-6832. School of Mathematics, Nanjing University of Aeronautics and Astronautics, Nanjing 211106, China;
		Key Laboratory of Mathematical Modeling and High Performance Computing of Air Vehicles (NUAA), MIIT, Nanjing 211106, China. (\email{liaohl@nuaa.edu.cn}, \email{liaohl@csrc.ac.cn})}
	\and 
	Chaoyu Quan\thanks{School of Science and Engineering, The Chinese University of Hong Kong (Shenzhen), 518172, P.R. China; Shenzhen International Center for Industrial and Applied Mathematics, Shenzhen Research Institute of Big Data, Shenzhen, 518172, China. (\email{quanchaoyu@cuhk.edu.cn})}
	\and
	Tao Tang\thanks{School of Mathematics and Statistics, Guangzhou Nanfang College, Guangzhou 510970, China.  (\email{ttang@nfu.edu.cn})}  
	\and
	Tao Zhou\thanks{Institute of Computational Mathematics and Scientific/Engineering Computing, Academy of Mathematics and Systems Science, Chinese Academy of Sciences, Beijing, 100190, China. Email: (\email{tzhou@lsec.cc.ac.cn})}
}

\usepackage{amsopn}

\begin{document}
	
	\maketitle
	
	\begin{abstract}
		The rigorous stability analysis of high-order implicit-explicit linear multistep (IELM) methods for nonlinear parabolic equations by using discrete energy arguments is a long standing open issue due to their non-A-stability property. A novel semi-generating function approach combined with a global discrete energy analysis 
		is suggested for the stability and convergence  of general IELM methods in solving nonlinear parabolic equations. Inspired from  the Grenander-Szeg\H{o} theorem for Toeplitz matrices,  the semi-generating function approach is used to handle the three groups of discrete coefficients via three complex polynomials on the unit circle.  
		A unified theoretical framework is then presented to establish the unconditional stability of IELM methods if the minimum eigenvalue of composite convolution kernels for the implicit part is properly large and the spectral norm bound of composite convolution kernels for the explicit part is properly small.  An indicator, called  implicit-explicit controllability intensity, is then introduced  to evaluate the degree of controllability of the implicit part over the explicit part. Some of the existing IELM methods, up to fifth-order time accuracy, are revisited and compared by computing the associated implicit-explicit controllability intensities such that one can choose an IELM method or proper parameter to maintain the unconditional stability for a specific nonlinear parabolic model. We also propose a new parameterized class of IELM methods, up to the ninth-order time accuracy, which satisfy the a priori settings of our theory and have a large value of the implicit-explicit controllability intensity by choosing a proper parameter so that they would be well suited for a wide class of nonlinear parabolic problems. 
	\end{abstract}
	
	
	\section{Introduction}\label{sec: introduction}
	\setcounter{equation}{0}
	
	Let $V$ and $H$ be two real Hilbert spaces such that $V \subset H=H' \subset V'$, with $V$ densely and continuously embedded in $H$ and $V'$ being the dual space of $V$.
	We will investigate the stability  of  implicit-explicit linear multistep (in short, IELM) methods for the nonlinear parabolic equation \cite{Akrivis:2013, AkrivisCrouzeixMakridakis:1999NM,LiWangZhou:2020BDF,HuangShen:2024}
	\begin{align}\label{nonlinearModel}
		u_t(t)+\varpi\mathcal{L}u(t)=\mathcal{F}(u(t)),\quad 0<t<T,
	\end{align}
	subject to the initial data $u(0)=u^0\in H$, where $\varpi>0$ is a prescribed constant, $\mathcal{L}: V\rightarrow V'$ is a positive definite, self-adjoint, bounded linear operator and  $\mathcal{F}: V\rightarrow V'$ can be a nonlinear operator. For differential operators, $\mathcal{F}:V\to V'$ allows to have the same order as or lower order than $\mathcal {L}$, cf. \cite[Section 1]{AkrivisCrouzeixMakridakis:1999NM}. We denote the inner product in $H$ and the antiduality pairing between $V'$ and $V$ by $\myinner{\cdot, \cdot}$. The induced norm in $H$ is denoted by $\|\cdot\|_H$ with $\mynorm{v}_H:=\myinner{v,v}^{1/2}$, and the norm $\|\cdot\|_V$ in $V$ can be defined by $\|v\|_V:=\myinner{\mathcal{L}v,v}^{1/2}$. The space $V'$ can be considered the completion of $H$ with respect to the dual norm
	\begin{align}\label{starnorm}
		\mynormb{v}_{\star}:=\mathop{\rm{sup}}\limits_{w \in V \backslash \{0\}} \frac{\abs{\myinnert{v, w}}}{\|w\|_V}
		=\mathop{\rm{sup}}\limits_{\|w\|_V=1}\abs{\myinnert{v, w}}\quad \text{ $\forall v \in V'$}.
	\end{align}
	Always, we assume that the nonlinear functional $\mathcal{F}(u(t))$ satisfies the following local Lipschitz condition in a tube $\mathcal{B}_{u(t)}:=\big\{v \in V: \|v-u(t)\|_V \le 1\big\}$ centered at the exact solution $u(t)$, and, for simplicity, defined here in terms of the norm of $V$,
	\begin{align}\label{operator N: local Lipschitz}
		\mynorm{\mathcal{F}(v)-\mathcal{F}(w)}_{\star} \le \mu_0\mynorm{v-w}_V+\mu_1\mynorm{v-w}_H\quad\text{ $\forall v, w \in \mathcal{B}_{u(t)}$},
	\end{align}
	with a nonnegative constant $ \mu_0\in(0,\varpi)$ and an arbitrary constant $\mu_1$. Obviously, the condition \eqref{operator N: local Lipschitz} holds as an equality with $\mu_1 =0$ for the special case $\mathcal F:=\mu_0\mathcal L$.
	
	To improve the computational efficiency of time approximation for the  problem  \eqref{nonlinearModel},  the operator $\mathcal{L}$ is always discretized implicitly in time  while $\mathcal{F}$ is discretized explicitly. In the sense of the local Lipschitz condition \eqref{operator N: local Lipschitz}, the present framework would be applicable to parabolic equations 
	\eqref{nonlinearModel} with a non-selfadjoint linear operator $\mathcal{L}=\mathcal{L}_s+\mathcal{L}_a$, where $\mathcal{L}_s:=(\mathcal{L}+\mathcal{L}^*)/2$ and  $\mathcal{L}_a:=(\mathcal{L}-\mathcal{L}^*)/2$ are  the self-adjoint and anti-self-adjoint parts, respectively, if the anti-self-adjoint part $\mathcal{L}_a$ satisfies
	\begin{align*}
		\mynorm{\mathcal{L}_a(v-w)}_{\star} \le \tilde{\mu}_0\mynorm{v-w}_V+\tilde{\mu}_1\mynorm{v-w}_H\quad\text{for $\forall v, w \in \mathcal{B}_{u(t)}$ and $\forall t \in [0,T]$},
	\end{align*}
	with a small constant $ \tilde{\mu}_0\in(0,\varpi)$ and an arbitrary constant $\tilde{\mu}_1$. In such a case, 
	the parabolic problem  \eqref{nonlinearModel} can be understood by replacing  $\mathcal{L}u$ and $\mathcal{F}(u)$ with $\mathcal{L}_su$ and $-\varpi\mathcal{L}_au+\mathcal{F}(u)$, respectively, that is, the anti-self-adjoint part $-\varpi\mathcal{L}_au$ is discretized explicitly. Throughout this paper, we only consider the parabolic problem with the time-independent linear operator $\mathcal{L}$; while some further developments of our theory would be required to handle the IELM methods for nonlinear parabolic problems  with a time-dependent linear operator $\mathcal{L}=\mathcal{L}(t)$, cf. \cite{Akrivis:2015,AkrivisKatsoprinakis:2016,AkrivisLubich:2015} on the stability of high-order schemes based on the wide-spread backward differentiation formulas (BDF) \cite{Gear:1971}.

	Consider the time mesh $0=t_0<t_1<\cdots<t_N=T$ with the time-step size $\tau =t_j-t_{j-1}$ for $j\ge1$. Let $u^j$ be the numerical approximation of $U^j:=u(t_j)$ at the discrete time level $t_j$ for $0\le j\le N$ and denote $\partial_{\tau}u^j=(u^j-u^{j-1})/\tau$ for $j\ge1$.
	To integrate the nonlinear parabolic problem \eqref{nonlinearModel} from $t_{n-1}$ ($n\ge1$) to the point $t_n$, we consider the following $\rmk$-step ($\rmk\ge1$)  implicit-explicit multistep (IELM) method involving the numerical solutions $u^{n-\rmk}$, $u^{n-\rmk+1}$, $\dots$, $u^n$:  
	\begin{align}\label{scheme: general imex multistep}	
		\sum_{j=0}^{\rmk-1}a_{j}^{(\rmk)}\partial_{\tau}u^{n-j}
		+\varpi\sum_{j=0}^{\rmk}b_{j}^{(\rmk)}\mathcal{L}u^{n-j}
		=\sum_{j=0}^{\rmk-1}c_{j}^{(\rmk)}\mathcal{F}(u^{n-j-1})+\mathfrak{C}_n^{(\rmk)}\brat{u^{0}}
	\end{align}
	for $n\ge1$, where $a_{j}^{(\rmk)}$, $b_{j}^{(\rmk)}$ and $c_{j}^{(\rmk)}$ are the discrete coefficients with $a_{0}^{(\rmk)},b_{0}^{(\rmk)}, c_{0}^{(\rmk)}>0$. The correction terms $\mathfrak{C}_n^{(\rmk)}\brat{u^{0}}$ are defined at the starting $(\rmk-1)$ steps to maintain the time accuracy with $\mathfrak{C}_n^{(\rmk)}\brat{u^{0}}:=0$ for $n\ge\rmk$. To highlight the main idea of this article and simplify our presentation, we assume throughout this paper that the correction terms $\mathfrak{C}_n^{(\rmk)}\brat{u^{0}}$ for $1\le n\le\rmk-1$ are available (cf. \cite{JinLiZhou:2017,LiWangZhou:2020BDF}) so that the IELM scheme \eqref{scheme: general imex multistep} is $\rmk$-th order consistent at the first $\rmk-1$ steps.
	The IELM scheme \eqref{scheme: general imex multistep} can be characterized by the following three polynomials
	\begin{align*}
		\varrho_a(\zeta):=\sum_{j=0}^{\rmk-1}a_{j}^{(\rmk)}\zeta^{\rmk-1-j},\quad
		\varrho_b(\zeta):=\sum_{j=0}^{\rmk}b_{j}^{(\rmk)}\zeta^{\rmk-j},\quad
		\varrho_c(\zeta):=\sum_{j=0}^{\rmk-1}c_{j}^{(\rmk)}\zeta^{\rmk-1-j}.
	\end{align*} 
	We always assume that $\varrho_a$ is a Schur polynomial (all its roots lie strictly inside the unit disk) and the real part $\Re\kbrat{\frac{\varrho_b(\zeta)}{\zeta\varrho_a(\zeta)}}>0$ for $|\zeta|>1$, see Remark~\ref{remark: bound quadratic form}.

	The form \eqref{scheme: general imex multistep}  contains many of the existing IELM schemes based on the BDF schemes, including the weighted backward differentiation formulas (WBDF) suggested by Li \& Xie \cite{LiXie:1991WBDF}, the modified implicit-explicit  backward differentiation formulas (MBDF)  constructed by Akrivis \& Karakatsani \cite{AkrivisKarakatsani:2003} and the implicit-explicit generalized backward differentiation formulas (GBDF) proposed recently by Huang \& Shen \cite{HuangShen:2024,HuangShen:2023,HuangShen:2025mcom}. 
	In general, these variants of implicit-explicit BDF schemes  were proposed originally to enlarge the absolute stability regions of the classical BDF methods  so that they can achieve unconditional stability and admit large time-steps for time integration when the nonlinear term is discretized explicitly to avoid  Newton-type inner iterations at each time level. A parameterized class of new ImEx (in short, NIMEX) schemes was suggested in \cite{RosalesSeiboldShirokoffZhou:2017,SeiboldShirokoffZhou:2019}, in which the authors developed a separate unconditional stability theory that can be leveraged to achieve unconditional stability even in cases where the explicit part is as stiff as the implicit part. Nonetheless, due to their non-A-stability  property (the absolute stability region of the implicit part alone does not completely include the left half plane $\mathbb{C}^-$) of third- and higher-order methods, it has been a long standing open question on the rigorous stability (always refers to the perturbed stability, that is, the numerical solution is continuously dependent on the initial and exterior perturbations in proper norms) and error analysis of high-order IELM methods for nonlinear parabolic equations by using discrete energy arguments, cf. \cite[Chapter V.8]{HairerWanner:2002}. As is well known, compared with the spectral and Fourier techniques \cite{Akrivis:2013,AkrivisCrouzeixMakridakis:1998MCOM,AkrivisCrouzeixMakridakis:1999NM}, 
	the discrete energy techniques (especially when certain spatial approximation is taken into account) would be elementary and applicable to linear and nonlinear partial differential equations, including reaction-diffusion equations, convection-diffusion equations, the Navier-Stokes equations
	and associated nonlinear coupled systems.
	
		
		A wide-spread discrete energy argument makes full use of the Nevanlinna-Odeh multiplier  technique \cite{NevanlinnaOdeh:1981}. Lubich, Mansour and Venkataraman \cite{LubichMansourVenkataraman:2013} combined this technique with Dahlquist's G-stability theory \cite{Dahlquist:1978} in the discrete energy analysis of high-order BDF schemes up to fifth-order for parabolic equations on evolving surfaces. The application and further developments of Nevanlinna-Odeh multipliers in the numerical analysis of fully implicit and implicit-explicit BDF methods for linear and nonlinear parabolic problems can be found in \cite{Akrivis:2015,AkrivisChenYuZhou:2021,AkrivisChenYu:2024,AkrivisKatsoprinakis:2016,AkrivisLubich:2015,ContriKovacsMassing:2025} and references therein. In general, different multipliers would be required for the discrete energy analysis of different IELM methods,  cf. \cite{AkrivisChenYuZhou:2021,AkrivisChenYu:2024} on the extensions of Nevanlinna-Odeh-type multiplier to the sixth-order BDF and seventh-order WBDF methods. Another example is the 
		GBDF-$\rmk$ $(2\le \rmk\le4)$ method \cite{HuangShen:2023,HuangShen:2024} with a free parameter $\beta$,
		\begin{align}\label{scheme: GBDF imex multistep}	
			\sum_{j=0}^{\rmk-1}a_{\mathrm{G},j}^{(\rmk)}\partial_{\tau}u^{n-j}
			+\varpi\sum_{j=0}^{\rmk-1}b_{\mathrm{G},j}^{(\rmk)}\mathcal{L}u^{n-j}
			=\sum_{j=0}^{\rmk-1}c_{\mathrm{G},j}^{(\rmk)}\mathcal{F}(u^{n-j-1})
			\quad\text{for $n\ge\rmk$,}
		\end{align}
		where the discrete coefficients $a_{\mathrm{G},j}^{(\rmk)}$, $b_{\mathrm{G},j}^{(\rmk)}$ and $c_{\mathrm{G},j}^{(\rmk)}$  for $0\le j\le \rmk-1$ can be determined independently
		by linear systems of Vandermonde-type. 
		\cite[Theorem 2]{HuangShen:2024} states that the GBDF2 and GBDF3 methods for  $\beta>1$, and the GBDF4 method for  $\beta\ge2$ are stable for linear parabolic problems.  The energy analysis takes advantage of 
		Dahlquist's G-stability theory \cite{Dahlquist:1978} and a novel decomposition of implicit part, see \cite[(3.5)-(3.6)]{HuangShen:2024},
		\begin{align}\label{implicitKernelDecomposition: HuangShen:2024}
			\sum_{j=0}^{\rmk-1}b_{\mathrm{G},j}^{(\rmk,\beta)} v^{n-j}= \eta_{\rmk}(\beta)\sum_{j=0}^{\rmk-1}c_{\mathrm{G},j}^{(\rmk,\beta)}v^{n-j}
			+\sum_{j=0}^{\rmk-1}d_{j}^{(\rmk,\beta)}v^{n-j},
		\end{align}
		where the exquisite decomposition factors $\eta_{2}:=(\beta-1)/\beta$, $\eta_{3}:=(\beta-1)/(\beta+1)$ and $\eta_{4}:=(\beta-1)/(\beta+3)$. By using  
		the multiplier $\sum_{j=0}^{\rmk-1}c_{\mathrm{G},j}^{(\rmk,\beta)}v^{n-j}$,  \cite[Theorem 3]{HuangShen:2024} establishes
		the corresponding convergence for the nonlinear parabolic problem \eqref{nonlinearModel} under the following stability condition (in our notations)
		\begin{align}\label{stability condition: HuangShen:2024}
			\eta_{\rmk}(\beta)>\mu_0/\varpi\quad\text{for $2\le \rmk\le 4$.}
		\end{align}
		Since the decomposition factors $\eta_{\rmk}(\beta)$ always vanish at $\beta=1$, the analysis in \cite{HuangShen:2024} would not be applicable to the BDF-$\rmk$ schemes corresponding to the case $\beta=1$.

		Very recently, based on the consistent splitting approximation \cite{GuermondShen:2003,JohnstonLiu:2004,LiuLiuPego:2007} of the incompressible Navier-Stokes equations, Huang and Shen \cite{HuangShen:2025mcom} establish the stability and convergence of the GBDF-$\rmk$ schemes (with three fixed parameters $\beta_{\rmk}=3,6,9$ corresponding to the order index $\rmk=2,3,4$, respectively)  in the $\ell^{\infty}(H^1)\cap \ell^{2}(H^2)$ norm  by using 
		a convolution-type multiplier $\sum_{j=0}^{\rmk-1}c_{\mathrm{G},j}^{(\rmk,\beta_{\rmk})}v^{n-j}$. 
		They take advantage of Dahlquist's G-stability theory \cite{Dahlquist:1978} and construct a refined decomposition  of the implicit part for certain constants $\xi_{\rmk}>\tfrac{\sqrt{2}}2$, see \cite[(3.16)]{HuangShen:2025mcom}, 
		\begin{align}\label{implicitKernelDecomposition: HuangShen2025mcom}
			\sum_{j=0}^{\rmk-1}b_{\mathrm{G},j}^{(\rmk,\beta_{\rmk})} v^{n-j}=\xi_{\rmk}\sum_{j=0}^{\rmk-1}c_{\mathrm{G},j}^{(\rmk,\beta_{\rmk})}v^{n-j}
			+\sum_{j=0}^{\rmk-1}d_{j}^{(\rmk,\beta_{\rmk})}v^{n-j}
			+\sum_{j=0}^{\rmk-1}f_{j}^{(\rmk,\beta_{\rmk})}v^{n-j},
		\end{align}
		where the coefficients $f_{j}^{(\rmk,\beta_{\rmk})}$ are picked delicately, see \cite[(3.17a)-(3.17c)]{HuangShen:2025mcom}. 
		The implicit splittings \eqref{implicitKernelDecomposition: HuangShen:2024} and \eqref{implicitKernelDecomposition: HuangShen2025mcom} play important roles in the discrete energy analysis in \cite{HuangShen:2024,HuangShen:2025mcom}; however, the refined implicit part decomposition \eqref{implicitKernelDecomposition: HuangShen2025mcom} works only for the fixed parameters $\beta_{\rmk}=3,6,9$. It remains unclear whether the stability and convergence of other parameterized IELM methods (including the WBDF, MBDF, and NIMEX methods mentioned above) can be established by using the discrete energy analysis with implicit part decompositions similar to \eqref{implicitKernelDecomposition: HuangShen2025mcom}.

		This paper proposes a novel semi-generating function approach combined with a global discrete energy analysis for the stability and convergence analysis of IELM methods in solving nonlinear parabolic equations. The main features of our approach are that it is theoretically concise (cf. the proof of Theorem \ref{thm: multistep stability}) and would be applicable for a wide class of parameterized IELM methods without involving the construction of  any Nevanlinna-Odeh-type multipliers or implicit part  decompositions like \eqref{implicitKernelDecomposition: HuangShen2025mcom}.
		We always reformulate the IELM methods \eqref{scheme: general imex multistep} as follows,
		\begin{align}\label{scheme: general imex multistep convolution}	
			\sum_{j=1}^na_{n-j}^{(\rmk)}\partial_{\tau}u^{j}
			+\varpi\sum_{j=1}^{n}b_{n-j}^{(\rmk)}\mathcal{L}u^{j}
			=\sum_{j=1}^{n}c_{n-j}^{(\rmk)}\mathcal{F}(u^{j-1})+\mathfrak{C}_n^{(\rmk)}\brat{u^{0}}
		\end{align}
		for $1\le n\le N$, where the values of the discrete coefficients $a_{j}^{(\rmk)}$, $b_{j}^{(\rmk)}$ and $c_{j}^{(\rmk)}$ are extended to the index $0\le j\le n-1$ but 
		assume that the discrete coefficients $b_{j}^{(\rmk)}$ vanish when $j\ge\rmk+1$, while the discrete coefficients $a_{j}^{(\rmk)}$ and $c_{j}^{(\rmk)}$ vanish when $j\ge\rmk$.

		Our framework will use the discrete energy
		analysis with the discrete orthogonal convolution (DOC) kernels \cite{LiaoKang:2022,LiaoTangZhou:2021bdf345,LiaoTangZhou:2024,LiaoTangZhou:2020bdf3,LiaoZhang:2021}.
		For a finite sequence $\vec{a}^{(\rmk)}=\big\{a_{0}^{(\rmk)},a_{1}^{(\rmk)},\cdots,a_{\rmk-1}^{(\rmk)}\big\}$, we will define 
		the DOC kernels $\vec{a}^{(-1,\rmk)}=\big\{a_{0}^{(-1,\rmk)},a_{1}^{(-1,\rmk)},\cdots\big\}$ as follows \cite{LiaoZhang:2021}
		\begin{align}\label{def: DOC-Kernels}
			a_{0}^{(-1,\rmk)}:=\frac{1}{a_{0}^{(\rmk)}}
			\quad \mathrm{and} \quad
			a_{j}^{(-1,\rmk)}:=-\frac{1}{a_{0}^{(\rmk)}}
			\sum_{i=1}^{j}a_{j-i}^{(-1,\rmk)}a_{i}^{(\rmk)}\quad \text{for $j\ge1$.}
		\end{align}
		For any $n\ge1$, one has the discrete orthogonal convolution identity \cite{LiaoTangZhou:2021bdf345,LiaoTangZhou:2024}
		\begin{align}\label{eq: orthogonal identity}
			\sum_{\ell=j}^{n}a_{n-\ell}^{(-1,\rmk)}a^{(\rmk)}_{\ell-j}\equiv\delta_{nj}=\sum_{\ell=j}^{n}a_{n-\ell}^{(\rmk)}a^{(-1,\rmk)}_{\ell-j}
			\quad\text{for any $1\leq j\le n$,}
		\end{align}
		where $\delta_{nj}$ is the Kronecker delta symbol. 
		Thus, by exchanging the summation order, 
		\begin{align*}
			\sum_{j=1}^{n}a_{n-j}^{(-1,\rmk)}
			\sum_{\ell=1}^{j}a_{j-\ell}^{(\rmk)}\partial_{\tau}u^{\ell}
			=&\,\sum_{\ell=1}^{n}\partial_{\tau}u^{\ell}\sum_{j=\ell}^{n}a_{n-j}^{(-1,\rmk)}a_{j-\ell}^{(\rmk)}
			=\partial_{\tau}u^n\quad\text{for $n\ge1$.}
		\end{align*}
		Multiplying equation \eqref{scheme: general imex multistep convolution}	 with 
		the DOC kernels $a_{m-n}^{(-1,\rmk)}$, summing over $n$ from $n=1$ to $m$ and replacing $m$ by $n$, we get
		\begin{align}\label{Dis: DOC action multistep formula Dk}
			\sum_{j=1}^{n}a_{n-j}^{(-1,\rmk)}
			\sum_{\ell=1}^ja_{j-\ell}^{(\rmk)}\partial_{\tau}u^{\ell}
			+\varpi\sum_{j=1}^{n}a_{n-j}^{(-1,\rmk)}\sum_{\ell=1}^{j}b_{j-\ell}^{(\rmk)}\mathcal{L}u^{\ell}	
			\hspace{3cm}\\
			=\sum_{j=1}^{n}a_{n-j}^{(-1,\rmk)}\sum_{\ell=1}^{j}c_{j-\ell}^{(\rmk)}\mathcal{F}(u^{\ell-1})	+\sum_{j=1}^{n}a_{n-j}^{(-1,\rmk)}\mathfrak{C}_j^{(\rmk)}\brat{u^{0}}
			\quad\text{for $n\ge1$.}\notag
		\end{align}
		By exchanging the summation order, one can apply \eqref{eq: orthogonal identity} to find
		\begin{align}\label{scheme: general imex multistep-differential}   
			\partial_{\tau}u^{n}
			+&\,\varpi\sum_{\ell=1}^{n}\hat{b}_{n-\ell}^{(\rmk)}\mathcal{L}u^{\ell}
			=\sum_{\ell=1}^{n}\hat{c}_{n-\ell}^{(\rmk)}\mathcal{F}(u^{\ell-1})
			+\sum_{\ell=1}^{n}a_{n-\ell}^{(-1,\rmk)}\mathfrak{C}_\ell^{(\rmk)}\brat{u^{0}}
		\end{align}
		for $n\ge1$, where the composited kernels $\hat{b}_{n-\ell}^{(\rmk)}$ and $\hat{c}_{n-\ell}^{(\rmk)}$ are defined by
		\begin{align}\label{scheme: multistep composited kernels}    
			\hat{b}_{j}^{(\rmk)}:=\sum_{i=0}^{j}a_{j-i}^{(-1,\rmk)}b_{i}^{(\rmk)}\quad\text{and}\quad 
			\hat{c}_{j}^{(\rmk)}:=\sum_{i=0}^{j}a_{j-i}^{(-1,\rmk)}c_{i}^{(\rmk)}\quad\text{for $j\ge0$}.
		\end{align}
		
		The equivalent form \eqref{scheme: general imex multistep-differential}, which is different from the original formulation  \eqref{scheme: general imex multistep convolution}, is our starting point of discrete energy  analysis.  This discrete convolution form contains the global information of solutions from $t_1$ to $t_n$ so that our analysis can be called the global discrete
		energy method, which would be especially suitable for the numerical analysis of IELM methods due to the nonlocal property in time levels.

		For the coefficients $a_{n-j}^{(\rmk)}$ and the associated DOC kernels $a_{n-j}^{(-1,\rmk)}$ defined by \eqref{def: DOC-Kernels}, we introduce the following $n\times n$ lower triangular Toeplitz matrices
		\begin{align*}
			&A_{L,\rmk}:=\begin{pmatrix}
				a_0^{(\rmk)} & && &\\
				\vdots                     & \ddots & &    &\\
				a_{\rmk-1}^{(\rmk)}&\cdots &a_0^{(\rmk)}&&\\
				&\ddots&\cdots&\ddots&\\
				&&a_{\rmk-1}^{(\rmk)}&\cdots&a_0^{(\rmk)}
			\end{pmatrix},\;
			A_{L,\rmk}^{(-1)}:=
			\begin{pmatrix}
				a_0^{(-1,\rmk)} & &\\
				a_1^{(-1,\rmk)} &a_0^{(-1,\rmk)} & \\
				\vdots            &    \ddots          &     \\
				a_{n-1}^{(-1,\rmk)}&\cdots &a_0^{(-1,\rmk)}
			\end{pmatrix}.
		\end{align*}
		The discrete orthogonal convolution identity \eqref{eq: orthogonal identity} says that  $A_{L,\rmk}^{(-1)}=A_{L,\rmk}^{-1}.$ In a similar way, one can write out the lower triangular Toeplitz matrices $B_{L,\rmk}$ and $C_{L,\rmk}$ from the discrete coefficients $b_{n-j}^{(\rmk)}$ and $c_{n-j}^{(\rmk)}$; while $B_{L,\rmk}^{(-1)}=B_{L,\rmk}^{-1}$ and $C_{L,\rmk}^{(-1)}=C_{L,\rmk}^{-1}$ are the lower triangular Toeplitz matrices for the corresponding DOC kernels $b_{n-j}^{(-1,\rmk)}$ and $c_{n-j}^{(-1,\rmk)}$ defined as in \eqref{def: DOC-Kernels}, respectively. Moreover,
		for the composited kernels $\hat{b}_{n-\ell}^{(\rmk)}$ and $\hat{c}_{n-\ell}^{(\rmk)}$ defined in \eqref{scheme: multistep composited kernels}, it is easy to see that the corresponding lower triangular Toeplitz matrices are
		(the products of lower triangular Toeplitz matrices are lower triangular Toeplitz matrices
		and any two lower triangular Toeplitz matrices of the same size commute, see \cite[Section 0.9.7]{HornJohnson:2013})
		\begin{align}  \label{matrix: hatB_L hatC_L}
			\widehat{B}_{L,\rmk}:=&\,A_{L,\rmk}^{(-1)}B_{L,\rmk}=A_{L,\rmk}^{-1}B_{L,\rmk},\quad
			\widehat{C}_{L,\rmk}:=A_{L,\rmk}^{(-1)}C_{L,\rmk}=A_{L,\rmk}^{-1}C_{L,\rmk}.
		\end{align}
		
		In the discrete energy analysis with respect to the norm $\mynorm{\cdot}_H$ (by testing with $2\tau u^{n}$) to the convolution form \eqref{scheme: general imex multistep-differential}, 
		the treatment of the implicit part requires to determine  the minimum eigenvalue $\lambda_{\mathrm{I}}^{(\rmk)}$ of the symmetric Toeplitz matrix $\mathcal{S}(\widehat{B}_{L,\rmk})$, where  $\mathcal{S}(D):=(D+D^T)/2$ for any matrix $D$, 
		while certain spectral norm bounds $\sigma_{\mathrm{F}}^{(\rmk)}$ and $\sigma_{\mathrm{E}}^{(\rmk)}$  of  lower triangular matrices $A_{L,\rmk}^{-1}$ and
		$\widehat{C}_{L,\rmk}$ should be evaluated in handling the explicit and exterior parts. These issues will be addressed in Section 2 by a novel semi-generating function method.
		
		Section \ref{sec: stability of IELM} performs the stability analysis  by a complete mathematical induction to the boundedness of solutions with respect to the norm $\mynorm{\cdot}_V$. It is shown that the IELM method \eqref{scheme: general imex multistep convolution} is unconditionally stable if the ratio of the minimum eigenvalue $\lambda_{\mathrm{I}}^{(\rmk)}$ of  $\mathcal{S}(\widehat{B}_{L,\rmk})$ over certain spectral norm bound  $\sigma_{\mathrm{E}}^{(\rmk)}$ of $\widehat{C}_{L,\rmk}$ is larger than $\mu_0/\varpi$. Motivated by the stability and convergence analysis, we introduce the implicit-explicit controllability intensity  $\mathfrak{I}_{\mathrm{IE}}^{(\rmk)}:=\lambda_{\mathrm{I}}^{(\rmk)}/\sigma_{\mathrm{E}}^{(\rmk)}$ to evaluate the degree of controllability of the implicit part of IELM methods over the associated explicit part.

		Section \ref{sec: existing IELM} revisits and compares some IELM methods for possible applications to the nonlinear parabolic problem \eqref{nonlinearModel} by evaluating the values of $\lambda_{\mathrm{I}}^{(\rmk)}$, $\sigma_{\mathrm{E}}^{(\rmk)}$, $\sigma_{\mathrm{F}}^{(\rmk)}$ and  $\mathfrak{I}_{\mathrm{IE}}^{(\rmk)}$. We will investigate four parameterized classes of  IELM methods, including $\alpha$-parameterized WBDF, $\beta$-parameterized  GBDF, $\delta$-parameterized  NIMEX methods and a simplified version (called SIELM schemes) of NIMEX methods. Some concluding remarks are included in the last section.

		\section{Semi-generating function method and technical lemmas}
		\label{sec: semi-generating function method}
		\setcounter{equation}{0}

		To investigate the stability roles of the implicit and explicit parts in any IELM method for nonlinear parabolic problems, we will introduce the semi-generating functions of lower triangular Toeplitz matrices
		and connect them with our discrete energy techniques. Lemma \ref{lem: Toeplitz-Caratheodory}, which is an extension of a classical result essentially due to Toeplitz and Carath\'{e}odory \cite{GrenanderSzego:2001}, uses the real part of the associated semi-generating function to bound the eigenvalues of the symmetric parts of lower triangular Toeplitz matrices, Lemma \ref{lem: composited generating function} describes the semi-generating function of composited sequences (corresponding to the  product of Toeplitz matrices) and Lemma \ref{lemma: spectral norm bound} uses the modulus of the semi-generating function to bound the spectral norms of lower triangular matrices.  The  proofs of Lemmas \ref{lem: Toeplitz-Caratheodory}-\ref{lemma: spectral norm bound} are gathered in Appendix \ref{sec: appendx Lemmas 2.1-2.3}.
		
		\begin{lemma}[Toeplitz and Caratheordory \cite{GrenanderSzego:2001}]\label{lem: Toeplitz-Caratheodory}
			Define a semi-generating function $a(\theta):=\sum_{k=0}^{\infty}a_k e^{\imath k\theta }{\in L^2([0,2\pi))}$ for a real sequence $\{a_{0},a_{1},a_{2},\cdots\}\in\ell^2(\mathbb{N})$.
			For any index $n\ge1$, 
			consider the following real quadratic form
			\begin{align*}
				Q_n:=\sum_{k=1}^nw_k\sum_{j=1}^ka_{k-j}w_j\quad\text{for any sequence $\{w_{1}, w_{2}, \cdots, w_{n}\}$,}
			\end{align*}
			corresponding to the real symmetric Toeplitz matrix $\mathcal{S}(P_{L, n})=(P_{L, n}+P_{L, n}^T)/2$ with
			the following associated  lower triangular Toeplitz matrix
			\begin{align}\label{lower triangular Toeplitz matrix PL}
				P_{L, n}:=\left(
				\begin{array}{ccccc}
					a_0 & && \\
					a_1 &a_0 & &\\
					\vdots            &    \ddots          & \ddots & \\
					a_{n-1}&\cdots &a_1&a_0\\
				\end{array}
				\right)_{n\times n}.
			\end{align}
			\begin{itemize}[itemindent=-0.5cm]
				\item[(i)] Then the real quadratic form $Q_n$ is positive definite if and only if the real part of $a(\theta)$ is positive, that is, $\Re\kbra{a(\theta)}>0$ for $\theta\in[0,2\pi)$;
				\item[(ii)]  and the eigenvalues $\lambda_j(Q_n)$ of the real quadratic form $Q_n$ can be bounded by
				\begin{align*}
					\min_{\theta\in[0,2\pi)}\Re\kbrab{a(\theta)}\le\lambda_j(Q_n)\le \max_{\theta\in[0,2\pi)}\Re\kbrab{a(\theta)}\quad\text{for any $n\ge j+1\ge1$.}
				\end{align*}
				\item[(iii)] Moreover, the eigenvalues $\lambda_j(Q_n)$ for $j=0,1,\cdots,n-1$ 
				are equally distributed as $\Re\kbrab{a(\frac{2\pi j}{n})}$ in the sense that
				$\lim_{n\rightarrow\infty}\frac1{n}\sum_{j=0}^{n-1}\braB{\lambda_j(Q_n)-\Re\kbrab{a\brab{\frac{2\pi j}{n}}}}=0.$
			\end{itemize}
		\end{lemma}
		
		\begin{lemma}\label{lem: composited generating function}
			Let the $L^2$ norm bounded functions $a(\theta):=\sum_{j=0}^{\infty}a_je^{\imath j\theta}{\in L^2([0,2\pi))}$ and $b(\theta):=\sum_{j=0}^{\infty}b_je^{\imath j\theta}{\in L^2([0,2\pi))}$ be the semi-generating functions for real sequences $\{a_{0},a_{1},\cdots,a_{k},\cdots\}$  and $\{b_{0},b_{1},\cdots,b_{k},\cdots\}$, respectively. 
			\begin{itemize}[itemindent=-0.5cm]
				\item[(i)] 	For the composited sequence $\{\hat{b}_{0},\hat{b}_{1},\cdots,\hat{b}_{k},\cdots\}$ defined by $\hat{b}_{j}:=\sum_{k=0}^{j}a_{j-k}b_{k},$
				the semi-generating function $\hat{b}(\theta)=\sum_{j=0}^{\infty}\hat{b}_je^{\imath j\theta}$ satisfies $\hat{b}(\theta)=a(\theta)b(\theta)$.
				\item[(ii)] Assume that $\{\xi_{0},\xi_{1},\cdots,\xi_{k},\cdots\}$ are the DOC kernels of  $\{a_{0},a_{1},\cdots,a_{k},\cdots\}$, defined by  $\xi_0:=\frac1{a_0}$ and $\xi_j:=-\frac1{a_0}\sum_{k=1}^{j}\xi_{j-k}a_{k}$ for $j\ge1$.
				Then
				the associated semi-generating function $\xi(\theta)=\sum_{j=0}^{\infty}\xi_je^{\imath j\theta}$ satisfies $\xi(\theta)=1/a(\theta)$.		 
			\end{itemize}
		\end{lemma}

		\begin{lemma}\label{lemma: spectral norm bound}
			For the lower triangular Toeplitz matrix $P_{L,n}$ in \eqref{lower triangular Toeplitz matrix PL}
			and the associated semi-generating function $a(\theta)=\sum_{k=0}^{\infty}a_k e^{\imath k\theta }{\in L^2([0,2\pi))}$, 
			the spectral norm of $P_{L,n}$ is not larger than $\max_{\theta\in[0,2\pi)}\abs{a(\theta)}$ for any order index $n\ge1$.
		\end{lemma}

		Now we return to the coefficients $a_{j}^{(\rmk)}$, $b_{j}^{(\rmk)}$ and $c_{j}^{(\rmk)}$  of the IELM method \eqref{scheme: general imex multistep}. 
		According to Lemma \ref{lem: Toeplitz-Caratheodory}, the associated semi-generating functions are defined by
		\begin{align}\label{matrix: A_L B_L C_L generating function}
			a^{(\rmk)}(\theta):=\sum_{j=0}^{\rmk-1}a_{j}^{(\rmk)}e^{\imath j\theta},\quad
			b^{(\rmk)}(\theta):=\sum_{j=0}^{\rmk}b_{j}^{(\rmk)}e^{\imath j\theta}\quad\text{and}\quad
			c^{(\rmk)}(\theta):=\sum_{j=0}^{\rmk-1}c_{j}^{(\rmk)}e^{\imath j\theta}.
		\end{align}
		Lemma \ref{lem: composited generating function} (ii) gives the semi-generating function 
		for the DOC kernels $a_{n-j}^{(-1,\rmk)}$,
		\begin{align}\label{matrix: A_L DOC generating function}
			a^{(-1,\rmk)}(\theta):=\sum_{j=0}^{\infty}a_{j}^{(-1,\rmk)}e^{\imath j\theta}
			=\frac{1}{a^{(\rmk)}(\theta)}.
		\end{align}
		For the composited discrete kernels $\hat{b}_{n-\ell}^{(\rmk)}$ and $\hat{c}_{n-\ell}^{(\rmk)}$ defined in \eqref{scheme: multistep composited kernels},
		Lemma \ref{lem: composited generating function} (i)  gives the associated semi-generating functions
		\begin{align}  
			\hat{b}^{(\rmk)}(\theta):=&\,\sum_{j=0}^{\infty}\hat{b}_{j}^{(\rmk)}e^{\imath j\theta}
			=a^{(-1,\rmk)}(\theta)b^{(\rmk)}(\theta)=\frac{b^{(\rmk)}(\theta)}{a^{(\rmk)}(\theta)},\label{matrix: hatB_L generating function}\\
			\hat{c}^{(\rmk)}(\theta):=&\,\sum_{j=0}^{\infty}\hat{c}_{j}^{(\rmk)}e^{\imath j\theta}
			=a^{(-1,\rmk)}(\theta)c^{(\rmk)}(\theta)=\frac{c^{(\rmk)}(\theta)}{a^{(\rmk)}(\theta)}.\label{matrix: hatC_L generating function}
		\end{align}


		Then, we have the following result, which builds some close relationships between the semi-generating functions of lower triangular Toeplitz matrices and our discrete energy techniques for the stability of IELM methods. 
		\begin{lemma}\label{lemma: bound quadratic form}	
			Let $\varrho_a$ be a Schur polynomial.  For the semi-generating functions $a^{(\rmk)}(\theta)$, $b^{(\rmk)}(\theta)$ and $c^{(\rmk)}(\theta)$ defined in \eqref{matrix: A_L B_L C_L generating function}, 
			assume that the following maxima and minimum
				are positive and finite:
			\begin{align}\label{def: lambda sigma}
				\sigma_{\mathrm{F}}^{(\rmk)}=\max\limits_{\theta\in[0,2\pi)}\abs{\frac{1}{a^{(\rmk)}(\theta)}},\;
				\sigma_{\mathrm{E}}^{(\rmk)}=\max\limits_{\theta\in[0,2\pi)}\abs{\frac{c^{(\rmk)}(\theta)}{a^{(\rmk)}(\theta)}},	\;
				\lambda_{\mathrm{I}}^{(\rmk)}=\min_{\theta\in[0,2\pi)}
				\Re\kbra{\frac{b^{(\rmk)}(\theta)}{a^{(\rmk)}(\theta)}}.
			\end{align}
			Then, for $n\ge1$, the spectral
			norms of the lower triangular Toeplitz matrices  $A_{L,\rmk}^{-1}$ and $A_{L,\rmk}^{-1}C_{L,\rmk}$ are bounded by the positive constants $\sigma_{\mathrm{F}}^{(\rmk)}$ and $\sigma_{\mathrm{E}}^{(\rmk)}$, respectively, and all eigenvalues of the symmetric matrix $\mathcal{S}(A_{L,\rmk}^{-1}B_{L,\rmk})$ are not less than $\lambda_{\mathrm{I}}^{(\rmk)}$. 
			For any sequences $\{v^{i},u^i: i \geq 1\}$, it holds that
			\begin{align*}
				(i)\;\;&\,\sum_{i=1}^{n} \sum_{j=1}^{i}\hat{b}_{i-j}^{(\rmk)}v^{j} v^{i} \geq \lambda_{\mathrm{I}}^{(\rmk)}\sum_{i=1}^{n}\absb{v^{i}}^2\quad\text{for $n\ge1$,}	\\
				(ii)\;\;&\,\sum_{i=1}^{n} \sum_{j=1}^{i}a_{i-j}^{(-1,\rmk)}v^{j} u^{i} \leq
				\sigma_{\mathrm{F}}^{(\rmk)}\sqrt{\sum_{i=1}^n\absb{v^{i}}^2}\sqrt{\sum_{i=1}^n\absb{u^{i}}^2} \quad\text{for $n\ge1$,}\\
				(iii)\;\;&\,\sum_{i=1}^{n} \sum_{j=1}^{i}\hat{c}_{i-j}^{(\rmk)}v^{j}u^{i} \leq
				\sigma_{\mathrm{E}}^{(\rmk)}\sqrt{\sum_{i=1}^n\absb{v^{i}}^2}\sqrt{\sum_{i=1}^n\absb{u^{i}}^2} \quad\text{for $n\ge1$.}	
			\end{align*}
		\end{lemma}
		\begin{proof}Lemma \ref{lem: composited generating function} and the minimum eigenvalue estimate in Lemma \ref{lem: Toeplitz-Caratheodory} (ii) say that all eigenvalues of the symmetric matrix $\mathcal{S}(A_{L,\rmk}^{-1}B_{L,\rmk})$ are not less than $\lambda_{\mathrm{I}}^{(\rmk)}$. Then, the well-known Cauchy's interlacing theorem \cite[Theorem 4.3.17]{HornJohnson:2013} yields the first result (i)  immediately. Lemmas \ref{lem: composited generating function} and \ref{lemma: spectral norm bound} imply that the spectral
			norms of the lower triangular Toeplitz matrices  $A_{L,\rmk}^{-1}$ and $A_{L,\rmk}^{-1}C_{L,\rmk}$ are bounded by the two constants $\sigma_{\mathrm{F}}^{(\rmk)}$ and $\sigma_{\mathrm{E}}^{(\rmk)}$, respectively. Thus the claimed results (ii)-(iii) can be verified by the Cauchy-Schwarz inequality. The proof is complete.
		\end{proof}

		
		\begin{remark}\label{remark: bound quadratic form}
			The definitions in \eqref{matrix: A_L B_L C_L generating function} of the semi-generating functions $a^{(\rmk)}( \theta)$, $b^{(\rmk)}( \theta)$ and $c^{(\rmk)}( \theta)$ are stemmed from the Grenander-Szeg\H{o} theorem and the standard generating function of a symmetric Toeplitz matrix; while they are closely related to the characteristic polynomials $\varrho_{a}^{(\rmk)},$ $\varrho_{b}^{(\rmk)}$ and $\varrho_{c}^{(\rmk)}$ of the IELM method \eqref{scheme: general imex multistep}.
			One has
			\begin{align*}
				a^{(\rmk)}( \theta)=e^{\imath(\rmk-1)\theta}\varrho_{a}^{(\rmk)}(e^{-\imath\theta}),\;\;
				b^{(\rmk)}( \theta)=e^{\imath\rmk\theta}\varrho_{b}^{(\rmk)}(e^{-\imath\theta}),\;\;
				c^{(\rmk)}( \theta)=e^{\imath(\rmk-1)\theta}\varrho_{c}^{(\rmk)}(e^{-\imath\theta}).
			\end{align*}
			Since $\varrho_a$ is Schur, 
				$a^{(\rmk)}(\theta)\ne0$ for all $\theta\in[0,2\pi]$, and $ c^{(\rmk)}(\theta)\not\equiv0$.
				Thus, $1/a^{(\rmk)}(\theta)$, $c^{(\rmk)}(\theta)/a^{(\rmk)}(\theta)$ and
				$b^{(\rmk)}(\theta)/a^{(\rmk)}(\theta)$ are continuous
				on $[0,2\pi]$, and their modulus attain finite maxima. They confirm the boundedness of $\sigma_{\mathrm{F}}^{(\rmk)}$ and $\sigma_{\mathrm{E}}^{(\rmk)}$ and $\lambda_{\mathrm{I}}^{(\rmk)}$ in \eqref{def: lambda sigma}. 
				
			
			At the same time, Lemma \ref{lemma: bound quadratic form} and the discrete energy analysis throughout this paper are limited to the positivity assumption $\lambda_{\mathrm{I}}^{(\rmk)}>0$, which is stemmed from the positive property  $\Re\kbrab{\frac{\varrho_{b}^{(\rmk)}(\zeta)}{\zeta\varrho_{a}^{(\rmk)}(\zeta)}}>0$ for $\abs{\zeta}>1$. 
			In the mathematical manner, our theoretical assumptions (the Schur property of $\varrho_a$ and
			$\lambda_{\mathrm{I}}^{(\rmk)}>0$) are severer than the strong $A(0)$-stability requirement of the spectral and Fourier techniques \cite{AkrivisCrouzeixMakridakis:1998MCOM,AkrivisCrouzeixMakridakis:1999NM}. We include an example to demonstrate it, but leave the detail proof to interested readers. Typically, the well-known BDF-$\rmk$ methods with $b^{(\rmk)}(\theta)=1$	fulfill the assumptions of Lemma~\ref{lemma: bound quadratic form}
			for $2\le\rmk\le5$. Nonetheless, one has
			$\lambda_{\mathrm{I}}^{(6)}=\min_{\theta\in[0,2\pi)}\Re\kbra{1/a^{(6)}(\theta)}<0$,
			cf. \cite[Lemma 2.4 and Remark 2.5]{LiaoKang:2022}, so that 
			the case $\rmk=6$ does not meet our assumptions. By contrast, the BDF6 method is strongly $A(0)$-stable and fulfills  the theoretical setting in \cite{AkrivisCrouzeixMakridakis:1998MCOM,AkrivisCrouzeixMakridakis:1999NM}.
			
	To handle the case of $\lambda_{\mathrm{I}}^{(\rmk)}<0$, certain discrete multiplier like the Nevanlinna-Odeh multipliers  \cite{NevanlinnaOdeh:1981} would be helpful for the discrete energy analysis, while the topic is out of our current scope. 
		Practically, the case $\lambda_{\mathrm{I}}^{(\rmk)}<0$ means that the composited discrete kernels  $\hat{b}_{i-j}^{(\rmk)}$ in \eqref{scheme: multistep composited kernels} introduce certain anti-dissipation effect (do not inherit the dissipativity of $\mathcal{L}u^{j}$ in itself) and the corresponding IELM method would be weakly dissipative 	although it can maintain the numerical stability (possibly, requiring small time-step size)
		for linear and nonlinear parabolic problems
		\cite{AkrivisChenYuZhou:2021,AkrivisKatsoprinakis:2016,
			AkrivisLubich:2015,ContriKovacsMassing:2025}.
			\end{remark}

		We want to emphasize that the boundedness of $\sigma_{\mathrm{F}}^{(\rmk)}$,   $\sigma_{\mathrm{E}}^{(\rmk)}$ and $\lambda_{\mathrm{I}}^{(\rmk)}$ defined in Lemma \ref{lemma: bound quadratic form} will be essential in the discrete energy analysis for the unconditional stability of IELM methods \eqref{scheme: general imex multistep}.
		From the perspective of the transformed IELM scheme \eqref{scheme: general imex multistep-differential} and the associated discrete energy  analysis  (by testing with $2\tau u^{n}$),  the spectral norm bound $\sigma_{\mathrm{F}}^{(\rmk)}$ in \eqref{def: lambda sigma} represents the ability  of the DOC kernels  $a_{i-j}^{(-1,\rmk)}$ to control the numerical growth of exterior forces (truncation errors), cf. the estimate \eqref{general imex multistep Error-product3} in the proof of Theorem \ref{thm: multistep stability}. The smaller the value of $\sigma_{\mathrm{F}}^{(\rmk)}$, the less the DOC kernels amplify the perturbation error. The spectral norm bound $\sigma_{\mathrm{E}}^{(\rmk)}$ in \eqref{def: lambda sigma} represents the ability of the composited kernels  $\hat{c}_{i-j}^{(\rmk)}$ to restrain the possible instability due to the explicit approximation of  nonlinear term $\mathcal{F}(u)$. The smaller the value of $\sigma_{\mathrm{E}}^{(\rmk)}$, the better the composited discrete kernels  $\hat{c}_{i-j}^{(\rmk)}$ control the numerical instability. Meanwhile, the minimum eigenvalue $\lambda_{\mathrm{I}}^{(\rmk)}$ in \eqref{def: lambda sigma} represents the ability of the composited discrete kernels  $\hat{b}_{i-j}^{(\rmk)}$ in the implicit part to maintain the dissipation property of $\mathcal{L}u^{j}$. 
		The larger the value of $\lambda_{\mathrm{I}}^{(\rmk)}$, the better the composited discrete kernels  $\hat{b}_{i-j}^{(\rmk)}$ inherit the dissipativity of $\mathcal{L}u^{j}$. 
		
		The following result gives the lower bounds of perturbation amplification factor $\sigma_{\mathrm{F}}^{(\rmk)}$ and nonlinear amplification factor $\sigma_{\mathrm{E}}^{(\rmk)}$, and the upper bound of dissipation preserving factor $\lambda_{\mathrm{I}}^{(\rmk)}$ for any consistent IELM method.
		
		\begin{lemma}\label{lemma: IELM stability intensity upper bound}
			Assume that the $\rmk$-th step IELM method \eqref{scheme: general imex multistep}	 with the discrete coefficients $a_{j}^{(\rmk)}$, $b_{j}^{(\rmk)}$ and $c_{j}^{(\rmk)}$ is consistent. It holds that 
			$	\sigma_{\mathrm{F}}^{(\rmk)}\ge1$,  $\sigma_{\mathrm{E}}^{(\rmk)}\ge1$ and $\lambda_{\mathrm{I}}^{(\rmk)}\le 1.$
			In particular, the implicit-explicit Euler scheme achieves the optimal values, that is,  $\sigma_{\mathrm{F}}^{(1)}=1$, $\sigma_{\mathrm{E}}^{(1)}=1$ and $\lambda_{\mathrm{I}}^{(1)}=1$.
		\end{lemma}
		\begin{proof}For  the semi-generating functions $a^{(\rmk)}(\theta)$, $b^{(\rmk)}(\theta)$ and $c^{(\rmk)}(\theta)$ defined in \eqref{matrix: A_L B_L C_L generating function}, the consistency gives $\sum_{j=0}^{\rmk-1}a_{j}^{(\rmk)}=1,$ $\sum_{j=0}^{\rmk}b_{j}^{(\rmk)}=1$ and $\sum_{j=0}^{\rmk-1}c_{j}^{(\rmk)}=1.$
			That is, $a^{(\rmk)}(0)=1$, $b^{(\rmk)}(0)=1$ and $c^{(\rmk)}(0)=1$.
			Then the definitions in \eqref{def: lambda sigma} imply that 	$$\sigma_{\mathrm{F}}^{(\rmk)}\ge\frac1{\abst{a^{(\rmk)}(0)}}=1,\quad 
			\sigma_{\mathrm{E}}^{(\rmk)}\ge\frac{\abst{c^{(\rmk)}(0)}}{\abst{a^{(\rmk)}(0)}}=1\quad \text{and}\quad
			\lambda_{\mathrm{I}}^{(\rmk)}\le \frac{b^{(\rmk)}(0)}{a^{(\rmk)}(0)}=1.$$ 
			For the implicit-explicit Euler scheme with $b_{0}^{(1)}=1$, we have $a^{(1)}(\theta)=1$, $b^{(1)}(\theta)=1$ and $c^{(1)}(\theta)=1$. The definitions in \eqref{def: lambda sigma} yield $\sigma_{\mathrm{F}}^{(1)}=1$, $\sigma_{\mathrm{E}}^{(1)}=1$ and $\lambda_{\mathrm{I}}^{(1)}=1$.
		\end{proof}

		\section{Consistency and stability of IELM methods}\label{sec: stability of IELM}
		\setcounter{equation}{0}
		
		We assume that both  methods $(\varrho_a,\varrho_b)$ and $(\varrho_a,\varrho_c)$ are consistent of order $q$, cf. \cite[Section 2]{AscherRuuthWetton:1995}, 
		\begin{align}\label{scheme: multistep-linear order conditions} 
			\sum_{j=0}^{\rmk-1}a_{j}^{(\rmk)}=1,\quad\sum_{j=0}^{\rmk-1}a_{j}^{(\rmk)}\partial_{\tau}t_{n-j}^{\ell}
			=\ell\sum_{j=0}^{\rmk}b_{j}^{(\rmk)}t_{n-j}^{\ell-1}
			=\ell\sum_{j=0}^{\rmk-1}c_{j}^{(\rmk)}t_{n-j-1}^{\ell-1}
		\end{align}
		for $1\le \ell\le q$, where the first condition ensures that the term $ \sum_{j=0}^{\rmk-1}a_{j}^{(\rmk)}\partial_{\tau}u^{n-j}$ is consistent with the first-order derivative.  As pointed out by \cite[Theorem 2.1]{AscherRuuthWetton:1995}, a $\rmk$-step IELM scheme cannot have order of accuracy greater than $\rmk$. Hereafter, we always set  $q=\rmk$.  Following the derivations in \cite[Section 3]{AkrivisCrouzeixMakridakis:1998MCOM} or \cite[Section 2]{AkrivisCrouzeixMakridakis:1999NM}, one can use the order conditions in \eqref{scheme: multistep-linear order conditions} to obtain the truncation error of the $\rmk$-step IELM scheme \eqref{scheme: general imex multistep} 	for  $n\ge\rmk$, 
		\begin{align}	\label{scheme: general imex Truncation Error}
			&\,R_n^{(\rmk)}
			:=\sum_{j=0}^{\rmk-1}a_{j}^{(\rmk)}\partial_{\tau}u(t_{n-j})
			+\varpi\sum_{j=0}^{\rmk}b_{j}^{(\rmk)}\mathcal{L}u(t_{n-j})
			-\sum_{j=0}^{\rmk-1}c_{j}^{(\rmk)}\mathcal{F}\kbra{u(t_{n-j-1})}\\
			=&\,\frac{1}{(\rmk+1)!}
			\braB{\sum_{j=0}^{\rmk-1}a_{j}^{(\rmk)}\partial_{\tau}t_{n-j}^{\rmk+1}
				-\sum_{j=0}^{\rmk}(\rmk+1)b_{j}^{(\rmk)}t_{n-j}^{\rmk}}\frac{\zd^{\rmk+1}u(t_n)}{\zd t^{\rmk+1}}\notag\\
			&\,+\frac{1}{\rmk!}
			\braB{\sum_{j=0}^{\rmk}b_{j}^{(\rmk)}t_{n-j}^{\rmk}-\sum_{j=0}^{\rmk-1}c_{j}^{(\rmk)}t_{n-j-1}^{\rmk}}
			\frac{\zd^{\rmk}\mathcal{F}(u(t_n))}{\zd t^{\rmk}}+O(\tau^{\rmk+1})\notag\\
			=&\,\tau^{\rmk}\frac{\zd^{\rmk+1}u(t_n)}{\zd t^{\rmk+1}}\cdot \lim_{\zeta\rightarrow1}\frac{\frac{\zeta-1}{\ln\zeta}\varrho_{a}^{(\rmk)}(\zeta)-\varrho_{b}^{(\rmk)}(\zeta)}{(\zeta-1)^{\rmk}}+b_{0}^{(\rmk)}\tau^{\rmk}\frac{\zd^{\rmk}\mathcal{F}(u(t_n))}{\zd t^{\rmk}}+O(\tau^{\rmk+1}),\notag
		\end{align}
		where in the last equality, \cite[Chapter III.2, Theorem 2.4, pp. 368]{HairerNorsettWanner:1992} and the following relationship $\frac{\varrho_{b}^{(\rmk)}(\zeta)-\varrho_{c}^{(\rmk)}(\zeta)}{(\zeta-1)^{\rmk}}=b_{0}^{(\rmk)}$ from \cite[Remark 3.1]{AkrivisCrouzeixMakridakis:1998MCOM} are applied.
		
		\subsection{Stability of IELM methods}	
		Here and hereafter, any subscripted $\ck$, such as $\ck_{u}$,  $\ck_\Omega$, $\ck_1$ and so on,
		denotes a fixed constant.
		The appearing constants may be dependent on the given data and the solution but are always independent of the time-step size $\tau$. 
		
		\begin{theorem}\label{thm: multistep stability} 
			Under the local Lipschitz condition \eqref{operator N: local Lipschitz} on the operator $\mathcal{F}$,  assume that the solution $u$ of the nonlinear parabolic equation \eqref{nonlinearModel} is  regular such that $\mynormb{\partial_t^{(\rmk+1)}u(t)}_{\star}$ and $\mynormb{\partial_t^{(\rmk)}\mathcal{F}(u(t))}_{\star}$ are bounded.
			Assume further that the IELM method \eqref{scheme: general imex multistep} is $\rmk$-th order consistent and satisfies the assumptions of Lemma \ref{lemma: bound quadratic form}.	If 
			\begin{align} \label{stability condition}
				\frac{\lambda_{\mathrm{I}}^{(\rmk)}}{\sigma_{\mathrm{E}}^{(\rmk)}}>\frac{\mu_0}{\varpi},
			\end{align}
			and the step size $\tau$ (relies on the values of ${\lambda_{\mathrm{I}}^{(\rmk)}}/{\sigma_{\mathrm{E}}^{(\rmk)}}$ and $\lambda_{\mathrm{I}}^{(\rmk)}/\sigma_{\mathrm{F}}^{(\rmk)}$) is sufficiently small, the IELM method \eqref{scheme: general imex multistep} is stable and convergent with the order of $O(\tau^{\rmk})$,
			\begin{align}\label{convergence: imex multistep}
				\mynormb{u(t_n)-u^n}_H^2
				+\epsilon_1\lambda_{\mathrm{I}}^{(\rmk)}\varpi\sum_{j=1}^{n}\tau\mynormb{u(t_j)-u^j}_V^2
				\le&\,\ck_1^2\tau^{2\rmk}\quad\text{for $1\le n\le N$.}
			\end{align}	
		\end{theorem}
		\begin{proof}The exact solution $U^j=u(t_j)$ satisfies the numerical method \eqref{scheme: general imex multistep} up to the consistency error $R_n^{(\rmk)}$ defined by \eqref{scheme: general imex Truncation Error},
			\begin{align*}
				\sum_{j=1}^na_{n-j}^{(\rmk)}\partial_{\tau}U^{j}
				+\varpi\sum_{j=1}^{n}b_{n-j}^{(\rmk)}\mathcal{L}U^{j}
				=\sum_{j=1}^{n}c_{n-j}^{(\rmk)}\mathcal{F}(U^{j-1})+\mathfrak{C}_n^{(\rmk)}\brat{U^{0}}
				+R_n^{(\rmk)}
			\end{align*}
			for $1\le n\le N$. Under the solution regularity and the assumption on the correction terms $\mathfrak{C}_n^{(\rmk)}\brat{u^{0}}$, there exists a positive constant $\ck_u$ such that  
			\begin{align}\label{IELM truncation error}	
				\mynormb{R_n^{(\rmk)}}_{\star}\le \ck_u\tau^{\rmk}\quad\text{for $1\le n\le N$.}
			\end{align}
			The solution errors $\tilde{u}^j=U^{j}-u^j$ satisfy the following error equation
			\begin{align}\label{scheme: general imex multistep Error}
				\sum_{j=1}^na_{n-j}^{(\rmk)}\partial_{\tau}\tilde{u}^{j}
				+\varpi\sum_{j=1}^{n}b_{n-j}^{(\rmk)}\mathcal{L}\tilde{u}^{j}
				=\sum_{j=1}^{n}c_{n-j}^{(\rmk)}\kbra{\mathcal{F}(U^{j-1})-\mathcal{F}(u^{j-1})}
				+R_n^{(\rmk)}
			\end{align}
			for $1\le n\le N$.
			We consider the complete mathematical induction for the bound
			\begin{align}\label{error: multistep H2 norm bound} 
				\mynormb{\tilde{u}^{\ell}}_V\le 1
				\quad\text{for $1\le \ell\le N$}.
			\end{align}
			It holds for $\ell=0$  since $\tilde{u}^{0}=0$. We will derive the error bound for the case $\ell=m$ from the following induction hypothesis
			\begin{align}\label{error: multistep H2 norm induction hypothesis} 
				\mynormb{\tilde{u}^{\ell}}_V\le 1
				\quad\text{for $1\le \ell\le m-1$}.
			\end{align}
			This hypothesis and the local Lipschitz condition \eqref{operator N: local Lipschitz} on $\mathcal{F}$ imply that
			\begin{align}\label{error: multistep hypothesis deduced uniform bound} 
				\mynormb{\mathcal{F}(U^{\ell})-\mathcal{F}(u^{\ell})}_{\star}\le \mu_0 \mynormb{\tilde{u}^{\ell}}_V+\mu_1\mynormb{\tilde{u}^{\ell}}_H\quad\text{for $1\le \ell\le m-1$.}
			\end{align}

			Following the derivation of \eqref{scheme: general imex multistep-differential}, we can obtain from \eqref{scheme: general imex multistep Error} that
			\begin{align}\label{general imex multistep Error-differential}   
				\partial_{\tau}\tilde{u}^{n}
				+&\,\varpi\sum_{\ell=1}^{n}\hat{b}_{n-\ell}^{(\rmk)}\mathcal{L}\tilde{u}^{\ell}
				=\sum_{\ell=1}^{n}\hat{c}_{n-\ell}^{(\rmk)}\kbra{\mathcal{F}(U^{\ell-1})-\mathcal{F}(u^{\ell-1})}
				+\sum_{\ell=1}^{n}a_{n-\ell}^{(-1,\rmk)}R_\ell^{(\rmk)}
			\end{align}
			for $1\le n\le N$, where the composited kernels $\hat{b}_{n-\ell}^{(\rmk)}$ and $\hat{c}_{n-\ell}^{(\rmk)}$ are defined by \eqref{scheme: multistep composited kernels}.
			By testing the equation \eqref{general imex multistep Error-differential} with $2\tau\tilde{u}^{n}$,
			and summing over $n$ from $n=1$ to $m$, we have
			\begin{align}\label{general imex multistep Error-product}   
				&\,\mynormb{\tilde{u}^{m}}_H^2-\mynormb{\tilde{u}^{0}}_H^2
				+\tau^2\sum_{j=1}^m\mynormb{\partial_{\tau}\tilde{u}^{j}}_H^2
				+2\varpi\tau\sum_{j=1}^{m}\sum_{\ell=1}^{j}\hat{b}_{j-\ell}^{(\rmk)}\myinnerb{\mathcal{L}\tilde{u}^{\ell},\tilde{u}^{j}}\\
				&\,\hspace{0.5cm}=\underbrace{2\tau\sum_{j=1}^{m}\sum_{\ell=1}^{j}\hat{c}_{j-\ell}^{(\rmk)}\myinnerb{\mathcal{F}(U^{\ell-1})-\mathcal{F}(u^{\ell-1}),\tilde{u}^{j}}}_{\textrm{RHS}_1}
				+2\tau\sum_{j=1}^{m}\sum_{\ell=1}^{j}a_{j-\ell}^{(-1,\rmk)}\myinnerb{R_\ell^{(\rmk)},\tilde{u}^{j}}.\notag
			\end{align}
			Lemma \ref{lemma: bound quadratic form} (i) gives 
			\begin{align*}    
				&\,2\varpi\tau\sum_{j=1}^{m}\sum_{\ell=1}^{j}\hat{b}_{j-\ell}^{(\rmk)}\myinnerb{\mathcal{L}\tilde{u}^{\ell},\tilde{u}^{j}}\ge 2\lambda_{\mathrm{I}}^{(\rmk)}\varpi\sum_{j=1}^{m}\tau\mynormb{\tilde{u}^{j}}_V^2.
			\end{align*}
			Applying Lemma \ref{lemma: bound quadratic form} (iii) and the estimate \eqref{error: multistep hypothesis deduced uniform bound}, the first term ($\textrm{RHS}_1$) on the right hand side of 
			\eqref{general imex multistep Error-product}  can be bounded by
			\begin{align*}    
				\textrm{RHS}_1\le&\, 
				2\sigma_{\mathrm{E}}^{(\rmk)}\sqrt{\sum_{i=1}^{m}\tau\mynormb{\mathcal{F}(U^{i-1})-\mathcal{F}(u^{i-1})}_{\star}^2}
				\sqrt{\sum_{i=1}^{m}\tau\mynormb{\tilde{u}^{i}}_V^2}\\
				\le&\, 
				2\sigma_{\mathrm{E}}^{(\rmk)}\mu_0\sqrt{\sum_{i=1}^{m-1}\tau\mynormb{\tilde{u}^{i}}_V^2}
				\sqrt{\sum_{i=1}^{m}\tau\mynormb{\tilde{u}^{i}}_V^2}
				+2\sigma_{\mathrm{E}}^{(\rmk)}\mu_1\sqrt{\sum_{i=1}^{m-1}\tau\mynormb{\tilde{u}^{i}}_H^2}
				\sqrt{\sum_{i=1}^{m}\tau\mynormb{\tilde{u}^{i}}_V^2},
			\end{align*}
			where the triangle inequality was used in the last step. 
			Thus, by using the Young inequality and Lemma \ref{lemma: bound quadratic form} (ii), the right hand side (RHS) of \eqref{general imex multistep Error-product} is bounded by
			\begin{align*}    
				\textrm{RHS}\le &\,\lambda_{\mathrm{I}}^{(\rmk)}\varpi\sum_{i=1}^{m}\tau\mynormb{\tilde{u}^{i}}_V^2
				+\frac{(\sigma_{\mathrm{E}}^{(\rmk)}\mu_0)^2}{\lambda_{\mathrm{I}}^{(\rmk)}\varpi}
				\sum_{i=1}^{m-1}\tau\mynormb{\tilde{u}^{i}}_V^2+\frac{\epsilon_1}{2}\lambda_{\mathrm{I}}^{(\rmk)}\varpi\sum_{i=1}^{m}\tau\mynormb{\tilde{u}^{i}}_V^2\\
				&\,+\frac{2(\sigma_{\mathrm{E}}^{(\rmk)}\mu_1)^2}{\epsilon_1\lambda_{\mathrm{I}}^{(\rmk)}\varpi}
				\sum_{i=1}^{m-1}\tau\mynormb{\tilde{u}^{i}}_H^2+\frac{2(\sigma_{\mathrm{F}}^{(\rmk)})^2}{\epsilon_1\lambda_{\mathrm{I}}^{(\rmk)}\varpi}\sum_{i=1}^{m}\tau \mynormb{R_i^{(\rmk)}}_*^2
				+\frac{\epsilon_1}{2}\lambda_{\mathrm{I}}^{(\rmk)}\varpi\sum_{i=1}^{m}\tau\mynormb{\tilde{u}^{i}}_V^2,
			\end{align*}
			where  $\epsilon_1>0$ is a parameter to be determined.  
			Then it follows from \eqref{general imex multistep Error-product}   that
			\begin{align}\label{general imex multistep Error-product3}
				\mynormb{\tilde{u}^{m}}_H^2+&\,(1-\epsilon_1)\lambda_{\mathrm{I}}^{(\rmk)}
				\varpi\tau\mynormb{\tilde{u}^{m}}_V^2
				+\kbra{1-\epsilon_1-\frac{(\sigma_{\mathrm{E}}^{(\rmk)}\mu_0)^2}
					{(\lambda_{\mathrm{I}}^{(\rmk)}\varpi)^2}}\lambda_{\mathrm{I}}^{(\rmk)}\varpi\sum_{j=1}^{m-1}\tau\mynormb{\tilde{u}^{j}}_V^2\\
				\le&\,\frac{2(\sigma_{\mathrm{E}}^{(\rmk)}\mu_1)^2}{\epsilon_1\lambda_{\mathrm{I}}^{(\rmk)}\varpi}
				\sum_{i=1}^{m-1}\tau\mynormb{\tilde{u}^{i}}_H^2+\frac{2(\sigma_{\mathrm{F}}^{(\rmk)})^2}{\epsilon_1\lambda_{\mathrm{I}}^{(\rmk)}\varpi}\sum_{i=1}^{m}\tau \mynormb{R_i^{(\rmk)}}_*^2.\notag
			\end{align}
			Under the stability condition \eqref{stability condition}, we choose  $\epsilon_1:=\frac12-\frac{(\sigma_{\mathrm{E}}^{(\rmk)}\mu_0)^2}
			{2(\lambda_{\mathrm{I}}^{(\rmk)}\varpi)^2}>0$ such that
			\begin{align*} 
				\mynormb{\tilde{u}^{m}}_H^2+\epsilon_1\lambda_{\mathrm{I}}^{(\rmk)}\varpi\sum_{j=1}^{m}\tau\mynormb{\tilde{u}^{j}}_V^2
				\le&\,\frac{2(\sigma_{\mathrm{E}}^{(\rmk)}\mu_1)^2}{\epsilon_1\lambda_{\mathrm{I}}^{(\rmk)}\varpi}
				\sum_{i=1}^{m-1}\tau\mynormb{\tilde{u}^{i}}_H^2+\frac{2(\sigma_{\mathrm{F}}^{(\rmk)})^2}{\epsilon_1\lambda_{\mathrm{I}}^{(\rmk)}\varpi}\sum_{i=1}^{m}\tau \mynormb{R_i^{(\rmk)}}_*^2.
			\end{align*}
			The standard discrete Gr\"{o}nwall inequality, such as \cite[Lemma 3.1]{LiaoZhang:2021},  gives
			\begin{align*}  
				\mynormb{\tilde{u}^{m}}_H^2+\epsilon_1\lambda_{\mathrm{I}}^{(\rmk)}\varpi\sum_{j=1}^{m}\tau\mynormb{\tilde{u}^{j}}_V^2
				\le&\,\frac{2(\sigma_{\mathrm{F}}^{(\rmk)})^2}{\epsilon_1\lambda_{\mathrm{I}}^{(\rmk)}\varpi}
				\exp\braB{\frac{2(\sigma_{\mathrm{E}}^{(\rmk)}\mu_1)^2}{\epsilon_1\lambda_{\mathrm{I}}^{(\rmk)}\varpi}t_{m-1}}
				\sum_{i=1}^{m}\tau \mynormb{R_i^{(\rmk)}}_*^2.
			\end{align*}
			Recalling the consistency estimates in \eqref{IELM truncation error}, we get
			\begin{align}  \label{general imex multistep Error-product4}
				\mynormb{\tilde{u}^{m}}_H^2+\epsilon_1\lambda_{\mathrm{I}}^{(\rmk)}\varpi\sum_{j=1}^{m}\tau\mynormb{\tilde{u}^{j}}_V^2
				\le&\,\frac{2(\sigma_{\mathrm{F}}^{(\rmk)})^2}{\epsilon_1\lambda_{\mathrm{I}}^{(\rmk)}\varpi}
				\exp\braB{\frac{2(\sigma_{\mathrm{E}}^{(\rmk)}\mu_1)^2}{\epsilon_1\lambda_{\mathrm{I}}^{(\rmk)}\varpi} t_{m-1}}\ck_u^2t_{m}\tau^{2\rmk},
			\end{align}
			and then, $\mynormb{\tilde{u}^{m}}_V
			\le\ck_1\tau^{\rmk-\frac12}$, where the constant $\ck_1:=\frac{\sigma_{\mathrm{F}}^{(\rmk)}\ck_u\sqrt{2T}}{\epsilon_1\lambda_{\mathrm{I}}^{(\rmk)}\varpi}
			\exp\braB{\frac{(\sigma_{\mathrm{E}}^{(\rmk)}\mu_1)^2}
				{\epsilon_1\lambda_{\mathrm{I}}^{(\rmk)}\varpi}T}$.
			By choosing a small time-step size $\tau\le \ck_1^{\frac{2}{1-2\rmk}}$ or
			\begin{align}\label{error: small time-step} 
				\tau\le&\, \kbra{\frac{\epsilon_1\lambda_{\mathrm{I}}^{(\rmk)}\varpi}{\sigma_{\mathrm{F}}^{(\rmk)}\ck_u\sqrt{2T}}
					\exp\braB{-\frac{(\sigma_{\mathrm{E}}^{(\rmk)}\mu_1)^2}
						{\epsilon_1\lambda_{\mathrm{I}}^{(\rmk)}\varpi}T}}^{\frac{2}{2\rmk-1}},
			\end{align}
			one gets $\mynormb{\tilde{u}^{m}}_V\le1$.
			It says that the error bound \eqref{error: multistep H2 norm bound} holds for $\ell=m$ and
			completes the mathematical induction. Thus the IELM method \eqref{scheme: general imex multistep} or \eqref{scheme: general imex multistep convolution} is unconditionally stable. The above inequality  \eqref{general imex multistep Error-product4} leads to the claimed error estimate \eqref{convergence: imex multistep}	and completes the proof.
		\end{proof}
		
		One can see from the above proof that the stability requirement \eqref{stability condition} vanishes when we consider only the linear parabolic problem or the semilinear parabolic problem with $\mu_0=0$. In such cases,  $\lambda_{\mathrm{I}}^{(\rmk)}>0$ is sufficient for the unconditional stability of such problems. In the physical meaning, the stability condition \eqref{stability condition} says that the implicit part should maintain the original dissipativity as much as possible and the explicit part should suppress the possible nonlinear instability as much as possible so that the numerical dissipation can well balance the nonlinear instability to achieve the unconditional stability of IELM methods \eqref{scheme: general imex multistep}.

		For the given parameters $\varpi$ and $\mu_0$ from the problem \eqref{nonlinearModel}, the stability condition \eqref{stability condition} presents a key requirement for potential users to choose a high-order IELM method especially when the ratio $\mu_0/\varpi$ is properly large. As an example, the well-known consistent splitting algorithm of the incompressible Navier-Stokes equations imposes a special value $\mu_0/\varpi=\sqrt{2}/2$, see \cite{KangLiaoLiu:2027}, which is determined by the Stokes pressure estimate of the unconstrained Navier-Stokes system.
		
		In many practical applications, cf. \cite[Section 2.2]{AkrivisLubich:2015} and \cite[Section 2]{LiWangZhou:2020BDF}, the value of $\mu_0$ might be chosen appropriately small to satisfy the required stability condition \eqref{stability condition}; nonetheless, it always introduces  a larger value of $\mu_1$ and eventually imposes a smaller value of the maximum step-size since the step-size restriction \eqref{error: small time-step}  relies on the value of $1/\mu_1$. This is obviously not what users expect because large time steps are always preferred  when one adopts high-order time approximations to accelerate the numerical simulations. 
		In such a case, a high-order IELM method having properly large values of $\lambda_{\mathrm{I}}^{(\rmk)}/\sigma_{\mathrm{E}}^{(\rmk)}$ and $\lambda_{\mathrm{I}}^{(\rmk)}/\sigma_{\mathrm{F}}^{(\rmk)}$ would be also desirable since it admits a properly large time-step size according to the step-size restriction \eqref{error: small time-step}.
		
		In general,  we always have $\lambda_{\mathrm{I}}^{(\rmk)}/\sigma_{\mathrm{E}}^{(\rmk)}\le1$ for any $\rmk$-step IELM method according to Lemma \ref{lemma: IELM stability intensity upper bound}, while the implicit-explicit  Euler scheme is unconditionally stable and convergent for the problem \eqref{nonlinearModel} due to the optimal value  $\lambda_{\mathrm{I}}^{(1)}/\sigma_{\mathrm{E}}^{(1)}=1$. 
			
			\begin{remark}
		In the mathematical sense, the restriction \eqref{stability condition} is only a sufficient (possibly not sharp) condition to the unconditional  stability of IELM methods for the abstract parabolic problem \eqref{nonlinearModel} due to the application of the discrete energy method. Actually, the requirement \eqref{stability condition} would be severer (see Remark \ref{remark: BDF stability conditions}) than the existing stability condition derived by the spectral and Fourier techniques, see \cite[(1.5)]{AkrivisCrouzeixMakridakis:1999NM},
			\begin{align}\label{stability condition: AkrivisKarakatsani:2003}
				\frac1{\mathcal{K}_{\rmk}}>\frac{\mu_0}{\varpi}\quad\text{with}\quad 
				\mathcal{K}_{\rmk}:=\sup\limits_{x>0}\max\limits_{\abs{\zeta}=1}
				\abs{\frac{x\varrho_{c}^{(\rmk)}(\zeta)}{(\zeta-1)\varrho_{a}^{(\rmk)}(\zeta)
						+x\varrho_{b}^{(\rmk)}(\zeta)}}\,.
			\end{align}
		Under the assumptions of
				Lemma~\ref{lemma: bound quadratic form}, one has
				\begin{align*}
					\mathcal{K}_{\rmk}=\sup\limits_{x>0}\max\limits_{\theta \in [0,2\pi]}
					{\frac{\abs{\frac{c^{(\rmk)}(\theta)}{a^{(\rmk)}(\theta)}}}{\absb{x^{-1}(1-e^{\imath\theta})+\frac{b^{(\rmk)}(\theta)}{a^{(\rmk)}(\theta)}}}}.
				\end{align*}
				Since
				\begin{align*}
					\absB{x^{-1}(1-e^{\imath\theta})+\frac{b^{(\rmk)}(\theta)}{a^{(\rmk)}(\theta)}}\ge \Re\Big[{x^{-1}(1-e^{\imath\theta})+\frac{b^{(\rmk)}(\theta)}{a^{(\rmk)}(\theta)}}\Big]\ge{\lambda_{\mathrm{I}}^{(\rmk)}}>0,
				\end{align*}
				it follows that $\mathcal{K}_{\rmk}\le {\sigma_{\mathrm{E}}^{(\rmk)}}/{\lambda_{\mathrm{I}}^{(\rmk)}}$. It says that the condition~\eqref{stability condition: AkrivisKarakatsani:2003}
				is theoretically sharper than our stability requirement \eqref{stability condition}.
		\end{remark}

		\subsection{Implicit-explicit controllability intensity}

		The above stability analysis inspires us to introduce an indicator, named the implicit-explicit controllability intensity $\mathfrak{I}_{\mathrm{IE}}^{(\rmk)}$, defined by the ratio of the minimum eigenvalue (dissipation preserving factor) $\lambda_{\mathrm{I}}^{(\rmk)}$ from the implicit part over the spectral norm bound (nonlinear amplification factor) $\sigma_{\mathrm{E}}^{(\rmk)}$ from the explicit part,
		\begin{align}\label{def: stability intensity}
			\mathfrak{I}_{\mathrm{IE}}^{(\rmk)}:=\frac{\lambda_{\mathrm{I}}^{(\rmk)}}{\sigma_{\mathrm{E}}^{(\rmk)}}
			=\frac{\min\limits_{\theta\in[0,2\pi)}
				\Re\kbra{\frac{b^{(\rmk)}(\theta)}{a^{(\rmk)}(\theta)}}}{\max\limits_{\theta\in[0,2\pi)}\abs{\frac{c^{(\rmk)}(\theta)}{a^{(\rmk)}(\theta)}}}\,,
		\end{align}
		where the deduced formula follows from the definitions in \eqref{def: lambda sigma}. As discussed in Section 2, it would represent the degree of controllability of the implicit part over the explicit part of a given IELM method. Actually, the controllability intensity $\mathfrak{I}_{\mathrm{IE}}$ is determined by the IELM method itself, while the underlying physical model determines the required controllability intensity threshold, such as $\mu_0/\varpi$ from the nonlinear parabolic model \eqref{nonlinearModel}.

		In the next section, we will revisit and compare some IELM methods for solving the nonlinear parabolic problem \eqref{nonlinearModel} by computing their controllability intensity $\mathfrak{I}_{\mathrm{IE}}^{(\rmk)}$. As the end of this section, we mention the following corollary of Lemma \ref{lemma: IELM stability intensity upper bound}.
		
		\begin{corollary}\label{corollary: IELM stability intensity upper bound}
			For any  consistent $\rmk$-step IELM method \eqref{scheme: general imex multistep}, the implicit-explicit controllability intensity $\mathfrak{I}_{\mathrm{IE}}$ can not be larger than 1, that is, $\mathfrak{I}_{\mathrm{IE}}^{(\rmk)}\le 1$, while the optimal value 1 can be achieved by the implicit-explicit Euler scheme with $\mathfrak{I}_{\mathrm{IE}}^{(1)}=1$.
		\end{corollary}

		\section{Controllability intensities of some IELM methods}\label{sec: existing IELM}
		\setcounter{equation}{0}

		In this section, we will evaluate the effectiveness of four different parameterized classes of  IELM schemes, including $\alpha$-parameterized WBDF  \cite{LiXie:1991WBDF},  $\beta$-parameterized GBDF \cite{HuangShen:2024}, $\delta$-parameterized NIMEX \cite{RosalesSeiboldShirokoffZhou:2017,SeiboldShirokoffZhou:2019} and a new class of $\gamma$-parameterized SIELM schemes, for the nonlinear parabolic problem \eqref{nonlinearModel} by calculating the values of theoretical indicators $\sigma_{\mathrm{F}}^{(\rmk)}$,   $\sigma_{\mathrm{E}}^{(\rmk)}$, $\lambda_{\mathrm{I}}^{(\rmk)}$ and  $\mathfrak{I}_{\mathrm{IE}}^{(\rmk)}$. It should be noted that the calculations and comparisons here mainly demonstrate the theoretical effectiveness of the semi-generating function method and the global discrete energy method in the above two sections, and would not represent their actual effects (such as the numerical precision and  admissible maximum time-step size) in the numerical simulations of a specific application.

		\subsection{WBDF methods}
		The WBDF-$\rmk$ $(2\le \rmk\le7)$ formulas \cite{LiXie:1991WBDF} are constructed by using 
		the backward differentiation formula with the $\alpha$-weighted approximation for the implicit part.
		As shown in \cite[Theorem 4]{LiXie:1991WBDF}, the WBDF2 method is A-stable if $\alpha\ge\frac12$; while \cite[Theorem 5]{LiXie:1991WBDF} states that the WBDF-$\rmk$ ($\rmk=3,4,5$) methods are $A(\theta)$-stable if $\alpha>\frac12$,  the WBDF-$\rmk$ ($\rmk=6,7$) methods are $A(\theta)$-stable if $\alpha\ge\frac{13}{5}$, and furthermore, the absolute stability regions of 
		WBDF-$\rmk$ methods enlarge as $\alpha$ increases.
		One has the  implicit-explicit WBDF-$\rmk$ methods for the model \eqref{nonlinearModel},
		\begin{align}\label{scheme: WBDF imex multistep}	
			\sum_{j=0}^{\rmk-1}a_{\mathrm{W},j}^{(\rmk)}\partial_{\tau}u^{n-j}
			+\varpi\sum_{j=0}^{\rmk}b_{\mathrm{W},j}^{(\rmk)}\mathcal{L}u^{n-j}
			=\sum_{j=0}^{\rmk-1}c_{\mathrm{W},j}^{(\rmk)}\mathcal{F}(u^{n-j-1})\quad\text{for $n\ge\rmk$,}
		\end{align}
		where we set $b_{\mathrm{W},0}^{(\rmk)}=\alpha$, $b_{\mathrm{W},1}^{(\rmk)}=1-\alpha$ and $b_{\mathrm{W},j}^{(\rmk)}=0$ for $2\le j\le \rmk$. 
		The corresponding three characteristic polynomials read
		$$\varrho_{a,\mathrm{W}}^{(\rmk)}(\zeta):=\sum_{j=1}^{\rmk}\frac{f_{\mathrm{W}}^{(j)}(1)}{j!}(\zeta-1)^{j-1}\quad \text{with}\quad f_{\mathrm{W}}(z)=(\alpha z-\alpha+1)z^{\rmk-1}\ln z,$$ 
		$\varrho_{b,\mathrm{W}}^{(\rmk)}(\zeta):=\zeta^{\rmk-1}(\alpha\zeta-\alpha+1)$ and $\varrho_{c,\mathrm{W}}^{(\rmk)}(\zeta):=\zeta^{\rmk-1}(\alpha\zeta-\alpha+1)-\alpha(\zeta-1)^{\rmk}$.
		Without special declarations, here we consider the stability property of WBDF-$\rmk$ ($2\le\rmk\le5$) methods for $\alpha\ge1$. Note that, the WBDF6 scheme satisfies the positivity condition $\lambda_{\mathrm{I},\mathrm{W}}^{(6,\alpha)}>0$ in Lemma \ref{lemma: bound quadratic form} for $\alpha>4.69$, but its associated implicit-explicit controllability intensity is very small, $\mathfrak{I}_{\mathrm{IE},\mathrm{W}}^{(6,\alpha)}<1/90$, limiting its  applicability to the nonlinear model \eqref{nonlinearModel}. The WBDF7 scheme fails to satisfy this positivity condition for any $\alpha\ge1$.

		\begin{table}[htb!]
			\centering
			\begin{threeparttable}
				\centering
				\caption{Stability property of WBDF-$\rmk$ methods}
				\vspace*{0.3pt}
				\def\temptablewidth{0.9\textwidth}
				\label{table: controllability intensity WBDF}
				\begin{tabular*}{\temptablewidth}{@{\extracolsep{\fill}}ccccc}
					\toprule
					$\rmk$&$\sigma_{\mathrm{F},\mathrm{W}}^{(\rmk,\alpha)}$ &$\sigma_{\mathrm{E},\mathrm{W}}^{(\rmk,\alpha)}$&$\lambda_{\mathrm{I},\mathrm{W}}^{(\rmk,\alpha)}$ 
					&$\mathfrak{I}_{\mathrm{IE},\mathrm{W}}^{(\rmk,\alpha)}$\\  \midrule		
					$2$  &1 & $\frac{2\alpha+1}{2\alpha}$& $\frac{2\alpha-1}{2\alpha}$
					&$\tfrac{2\alpha-1}{2\alpha +1}$ \\[4pt]
					$3$  &1 & $\frac{3(6\alpha+1)}{2(6\alpha-1)}$& $\frac{3(2\alpha-1)}{2(6\alpha-1)}$ 
					&$\tfrac{2\alpha-1}{6\alpha +1}$ \\[4pt]
					$4$ & $1$&$\frac{3(14\alpha+1)}{4(5\alpha-1)}$
					& $	\frac{120\alpha-61}{80(5\alpha-1) }$
					&$	\frac{120\alpha-61}{60(14\alpha +1)}$\\[4pt]
					$5$  &$\tfrac{47\alpha-2}{40\alpha}$	&$\frac{45 (10 \alpha +1)}{32 (5 \alpha -1)}$
					&$\frac{5 (6 \alpha -5)}{32 (5 \alpha -1)}$
					&$\frac{6 \alpha -5}{90 \alpha +9}$
				\end{tabular*}
				{\rule{\temptablewidth}{0.5pt}}
			\end{threeparttable}
		\end{table}	
		
		Table \ref{table: controllability intensity WBDF} collects the upper bounds of  $\sigma_{\mathrm{F},\mathrm{W}}^{(\rmk,\alpha)}$ and $\sigma_{\mathrm{E},\mathrm{W}}^{(\rmk,\alpha)}$, the lower bounds of $\lambda_{\mathrm{I},\mathrm{W}}^{(\rmk,\alpha)}$ and $\mathfrak{I}_{\mathrm{IE},\mathrm{W}}^{(\rmk,\alpha)}$ for the WBDF-$\rmk$ $(2\le \rmk\le 5)$ methods, see Propositions \ref{prop: WBDF2}-\ref{prop: WBDF5} in section \ref{sec: SM-WBDF} of the supplementary material. As seen, the implicit-explicit controllability intensity $\mathfrak{I}_{\mathrm{IE},\mathrm{W}}^{(\rmk,\alpha)}$ is always increasing with the free parameter $\alpha$ so that the WBDF schemes enhance the applicability to the nonlinear parabolic problem \eqref{nonlinearModel} compared with the BDF methods with $\alpha=1$. Since one can find some appropriate parameter $\alpha>\frac{\varpi+\mu_0}{2(\varpi-\mu_0)}$ to satisfy the  condition \eqref{stability condition} for any $0<\mu_0<\varpi$, we say that  the WBDF2 scheme has good adaptability to  \eqref{nonlinearModel}. 
		
		In contrast, the WBDF-$\rmk$ $(3\le \rmk\le 5)$ schemes seem only applicable to the cases where $\mu_0/\varpi$ is small and their applicability to the nonlinear model \eqref{nonlinearModel} rapidly decreases as the temporal order $\rmk$ increases. It is theoretically and practically desirable to develop high-order IELM methods having a large value of $\mathfrak{I}_{\mathrm{IE}}^{(\rmk)}$.
		Also, the value of  $\lambda_{\mathrm{I},\mathrm{W}}^{(\rmk,\alpha)}/\sigma_{\mathrm{F},\mathrm{W}}^{(\rmk,\alpha)}$, which is closely related to the admissible maximum time-step size, rapidly decreases as  $\rmk$ increases. It is well consistent with the usual numerical experiences: higher order BDF-type methods always require smaller time-step size to maintain the numerical stability for nonlinear parabolic problems, see the numerical tests in \cite[Example 3]{HuangShen:2024}, \cite[Example 1]{HuangShen:2025mcom}
		and \cite[Example 1]{KangLiaoLiu:2027}. 
		
		\begin{remark}\label{remark: BDF stability conditions}
			Taking  $\alpha = 1$ in the last column of Table \ref{table: controllability intensity WBDF}, one has the following stability conditions of the implicit-explicit BDF-$\rmk$ methods for $3\le\rmk\le5$, respectively,
			$$\frac{\mu_0}\varpi<\frac{1}{7}\approx0.142857,\quad \frac{\mu_0}\varpi<\frac{59}{900}\approx0.065556,\quad
			\frac{\mu_0}\varpi<\frac{1}{99}\approx0.010101.$$
			For the condition \eqref{stability condition: AkrivisKarakatsani:2003} derived by the spectral and Fourier techniques, one has $\mathcal{K}_{\rmk}=2^\rmk-1$ for the implicit-explicit BDF-$\rmk$ methods, see \cite[Remark 2.4]{AkrivisCrouzeixMakridakis:1999NM}, and the stability condition \eqref{stability condition: AkrivisKarakatsani:2003} for $\rmk=3, 4$ and 5 reads, respectively,
			$$\frac{\mu_0}\varpi<\frac{1}{7}\approx0.142857,\quad \frac{\mu_0}\varpi<\frac{1}{15}\approx0.066667,\quad
			\frac{\mu_0}\varpi<\frac{1}{31}\approx0.032258.$$
			Obviously, our stability condition \eqref{stability condition} is not sharper than \eqref{stability condition: AkrivisKarakatsani:2003} especially for the case $\rmk=5$.
			Applying the discrete energy technique with optimal Nevanlinna-Odeh multipliers, Akrivis and Katsoprinakis \cite{AkrivisKatsoprinakis:2016} obtain the condition 
			$\frac{\mu_0}\varpi<\frac{1}{(2^\rmk-1)(1+\hat{\eta}_\rmk)}$
			with $\hat{\eta}_3 = 1 / 13$, $\hat{\eta}_4 = 0.2878$, and $\hat{\eta}_5 = 0.8097337459$, for a broader class of parabolic equations with time-dependent operators.  One can check that our stability condition \eqref{stability condition} for the BDF-$\rmk$ methods is sharper than this stability condition for $\rmk=3$ and $4$, but is more restrictive than the requirement $\frac{\mu_0}\varpi < 0.01782476$ in \cite{AkrivisKatsoprinakis:2016} for the case $\rmk=5$. Up to now, we are not able to improve the stability condition \eqref{stability condition}.
		\end{remark}

		\subsection{GBDF and NIMEX methods}
		
		Here we consider the stability property of GBDF-$\rmk$ ($2\le \rmk\le 5$) methods \cite{HuangShen:2024} for $\beta\ge1$.
		In fact, the GBDF6 method constructed by following \cite{HuangShen:2024} does not satisfy $\lambda_{\mathrm{I},\mathrm{G}}^{(6)}>0$ for $\beta>2$ since the first characteristic polynomial $\varrho_{a,\mathrm{G}}^{(6)}$ is not a Schur polynomial if $\beta>2$.

		\begin{table}[htb!]
			\centering
			\begin{threeparttable}
				\centering
				\caption{Stability property of GBDF-$\rmk$ methods}
				\vspace*{0.3pt}
				\def\temptablewidth{1\textwidth}
				\label{table: controllability intensity GBDF}
				\begin{tabular*}{\temptablewidth}{@{\extracolsep{\fill}}ccccc}
					\toprule
					$\rmk$	&$\sigma_{\mathrm{F},\mathrm{G}}^{(\rmk,\beta)}$&$\sigma_{\mathrm{E},\mathrm{G}}^{(\rmk,\beta)}$ &$\lambda_{\mathrm{I},\mathrm{G}}^{(\rmk,\beta)}$ 
					&$\mathfrak{I}_{\mathrm{IE},\mathrm{G}}^{(\rmk,\beta)}$\\  \midrule		
					$2$  &1 	& $\frac{2\beta+1}{2\beta}$& $\frac{2\beta-1}{2\beta}$
					&$\tfrac{2\beta-1}{2\beta +1}$\\[4pt]
					$3$  &1& $\frac{6 \beta ^2+12 \beta +3}{6 \beta ^2+6 \beta -2}$& $\frac{6 \beta ^2-4}{6 \beta ^2+6 \beta -2}$
					& $\tfrac{6\beta ^2-4}{6\beta ^2+12\beta +3}$  \\[4pt]
					$4$ &$\frac{11\beta-1}{10\beta}$	&$\frac{4 \beta ^3+19\beta ^2+20 \beta +3}{4 (\beta ^3+3 \beta ^2+\beta -1)}$
					& $\frac{4 \beta ^3+5 \beta ^2-4 \beta -3}{4 (\beta ^3+3 \beta ^2+\beta -1)}$ 
					&  $\tfrac{4 \beta ^3+5 \beta ^2-4 \beta -3}{4 \beta ^3+19 \beta ^2+20 \beta+3}$ \\[4pt]
					$5^*$  &$\frac{20\beta-1}{10\beta}$& $\frac{5(2 \beta ^4+30 \beta ^3+32 \beta ^2+38 \beta +5)}{10 \beta ^4+60 \beta ^3+90 \beta ^2-32}$& $\frac{5(2 \beta ^4+\beta ^3+4 \beta ^2-4 \beta -2)}{10 \beta ^4+60 \beta ^3+90 \beta ^2-32}$
					& $\frac{2 \beta ^4+\beta ^3+4 \beta ^2-4 \beta -2}{2 \beta ^4+30 \beta ^3+32 \beta ^2+38 \beta +5}$\\[4pt]
					$5^\star$  &$\frac{20\beta-1}{10\beta}$& $\frac{5 (2 \beta ^4+21 \beta ^3+29 \beta ^2+38 \beta +5)}{10 \beta ^4+60 \beta ^3+90 \beta ^2-32}$& $\frac{5 (2 \beta ^4+3 \beta ^3+25 \beta ^2-9 \beta -3)}{10 \beta ^4+60 \beta ^3+90 \beta ^2-32}$
					& $\frac{2 \beta ^4+3 \beta ^3+25 \beta ^2-9 \beta -3}{2 \beta ^4+21 \beta ^3+29 \beta ^2+38 \beta +5}$
				\end{tabular*}
				{\rule{\temptablewidth}{0.5pt}}
				\footnotesize
				\tnote{The two cases $\rmk=5^*$ and $5^\star$ require $1\le \beta<18$ and $\beta\ge18$, respectively.}
			\end{threeparttable}
		\end{table}	
		
		Table \ref{table: controllability intensity GBDF} collects the upper bounds of  $\sigma_{\mathrm{F},\mathrm{G}}^{(\rmk,\beta)}$ and $\sigma_{\mathrm{E},\mathrm{G}}^{(\rmk,\beta)}$, the lower bounds of $\lambda_{\mathrm{I},\mathrm{G}}^{(\rmk,\beta)}$ and $\mathfrak{I}_{\mathrm{IE},\mathrm{G}}^{(\rmk,\beta)}$ for the GBDF-$\rmk$ methods, see Propositions \ref{prop: GBDF3}, \ref{prop: GBDF4} and \ref{prop: GBDF5} together with Remarks \ref{remark: lambda GBDF4} and \ref{remark: lambda GBDF5} in section \ref{sec: SM-GBDF} of the supplementary material for  more details. As seen, the  controllability intensity $\mathfrak{I}_{\mathrm{IE},\mathrm{G}}^{(\rmk,\beta)}$ is always increasing with the free parameter $\beta$ so that the GBDF schemes enhance the applicability to the nonlinear parabolic problem \eqref{nonlinearModel} compared with the standard BDF methods with $\beta=1$. Moreover, one may find some appropriate parameter $\beta$ to satisfy the stability condition \eqref{stability condition} for any $0<\mu_0<\varpi$. In this sense, the GBDF-$\rmk$ methods always have better adaptability to the nonlinear parabolic equation \eqref{nonlinearModel} than the WBDF-$\rmk$ schemes in Section 4.1. Note that, the value of  $\lambda_{\mathrm{I},\mathrm{G}}^{(\rmk,\beta)}/\sigma_{\mathrm{F},\mathrm{G}}^{(\rmk,\beta)}$, which is closely related to the admissible maximum time-step size, slowly increases as  $\rmk$ increases. One can choose properly larger $\beta$ for a high-order GBDF method so that the admissible maximum time-step size is comparable to that of the GBDF2 scheme, see the experiments in \cite[Example 3]{HuangShen:2024}, \cite[Example 1]{HuangShen:2025mcom} and \cite[Example 1]{KangLiaoLiu:2027}.

		At the same time,
		the GBDF-$\rmk$ schemes have their own defects: the discrete coefficients involving $\rmk-1$ degree polynomials with respect to $\beta$ are rather complex (the functions $\frac1{\abst{a_{\mathrm{G}}^{(\rmk)}(\theta)}}$, $\frac{\abst{c_{\mathrm{G}}^{(\rmk)}(\theta)}}{\abst{a_{\mathrm{G}}^{(\rmk)}(\theta)}}$ and
		$\Re\kbrab{\tfrac{b_{\mathrm{G}}^{(\rmk)}(\theta)}{a_{\mathrm{G}}^{(\rmk)}(\theta)}}$ are always not unimodal for $\theta\in[0,2\pi)$ and the calculations of their extreme values become rather complex), and the improvement of the controllability intensity $\mathfrak{I}_{\mathrm{IE},\mathrm{G}}^{(\rmk)}$ as the parameter $\beta$ increases at the expense of rapidly increased truncation error, referred to the proof of \cite[Theorem 4.1]{HuangShen:2025mcom} or the leading truncation errors in \eqref{truncation: GBDF3}, \eqref{truncation: GBDF4} and \eqref{truncation: GBDF5}. 
		
		As a by-product, according to Theorem \ref{thm: multistep stability}, the theoretical results in this subsection verify that the GBDF-$\rmk$ $(2\le \rmk\le 5)$ methods are stable for the parameter $\beta\ge1$ when they are applied to linear parabolic problems (the only condition is $\lambda_{\mathrm{I}}^{(\rmk,\beta)}>0$), which essentially improve the results in \cite[Theorem 2]{HuangShen:2024}. For the nonlinear parabolic problem \eqref{nonlinearModel}, we also improve \cite[Theorem 3]{HuangShen:2024} essentially by establishing the unconditional stability of the GBDF-$\rmk$ $(2\le \rmk\le 5)$ methods under the stability  requirement $\mathfrak{I}_{\mathrm{IE},\mathrm{G}}^{(\rmk,\beta)}>\mu_0/\varpi$, which is much sharper than the condition \eqref{stability condition: HuangShen:2024}.

		To end this subsection, we revisit the $\delta$-parameterized NIMEX  schemes \cite{RosalesSeiboldShirokoffZhou:2017,SeiboldShirokoffZhou:2019} with some comments for the value of implicit-explicit controllability intensity. The $\rmk$-step NIMEX  schemes for the nonlinear parabolic model \eqref{nonlinearModel} read
		\begin{align}\label{scheme: NIMEX multistep}	
			\sum_{j=0}^{\rmk-1}a_{\mathrm{N},j}^{(\rmk)}\partial_{\tau}u^{n-j}
			+\varpi\sum_{j=0}^{\rmk}b_{\mathrm{N},j}^{(\rmk)}\mathcal{L}u^{n-j}
			=\sum_{j=0}^{\rmk-1}c_{\mathrm{N},j}^{(\rmk)}\mathcal{F}(u^{n-j-1})\quad\text{for $n\ge\rmk$,}
		\end{align}
		where the coefficients $a_{\mathrm{N},j}^{(\rmk)}$, $b_{\mathrm{N},j}^{(\rmk)}$ and $c_{\mathrm{N},j}^{(\rmk)}$ are determined by the three characteristic polynomials (for the sake of consistency in the present context, the range of free parameter $\delta$ is modified from   $0<\delta\le1$ in the original papers \cite{RosalesSeiboldShirokoffZhou:2017,SeiboldShirokoffZhou:2019} to $\delta\ge1$ with the transform $\delta\leftarrow1/\delta$)
		$$\varrho_{a,\mathrm{N}}^{(\rmk)}(\zeta):=\sum_{j=1}^{\rmk}\frac{f_{\mathrm{N}}^{(j)}(1)}{j!}(\zeta-1)^{j-1}\quad \text{with}\quad f_{\mathrm{N}}(z)=(\delta z-\delta+1)^{\rmk}\ln z,$$ 
		$\varrho_{b,\mathrm{N}}^{(\rmk)}(\zeta):=(\delta\zeta-\delta+1)^{\rmk}$ and $\varrho_{c,\mathrm{N}}^{(\rmk)}(\zeta):=(\delta\zeta-\delta+1)^{\rmk}-\delta^{\rmk}(\zeta-1)^{\rmk}$,
		respectively.

		Note that, at least in our theoretical framework with discrete energy techniques, the $\delta$-parameterized NIMEX-$\rmk$ methods \eqref{scheme: NIMEX multistep} would be weaker than the GBDF-$\rmk$ schemes on the adaptability to the nonlinear parabolic model \eqref{nonlinearModel}.  Taking the case $\rmk=2$ with the free parameter $\delta\ge6/5$,
		one has $\max_{\delta\ge6/5}\lambda_{\mathrm{I},\mathrm{N}}^{(2,\delta)}=\lambda_{\mathrm{I},\mathrm{N}}^{(2,\frac{11}4)}
		=\frac{27}{32}$ and $\max_{\delta\ge6/5}\mathfrak{I}_{\mathrm{IE},\mathrm{N}}^{(2,\delta)}\approx 0.795354$ as $\delta\approx8.5176$. But,  $\mathfrak{I}_{\mathrm{IE},\mathrm{G}}^{(2,\beta)}
		>\max_{\delta\ge6/5}\mathfrak{I}_{\mathrm{IE},\mathrm{N}}^{(2,\delta)}$ if $\beta\ge \frac{9}{2}$. 
		Considering  $\rmk=3$ with the free parameter $\delta\ge2$, 
		one can find that $\frac{19}{37}\le\mathfrak{I}_{\mathrm{IE},\mathrm{N}}^{(3,\delta)}<\frac{88}{125}$ for  $\delta\ge2$. 
		However, $\mathfrak{I}_{\mathrm{IE},\mathrm{G}}^{(3,\beta)}>\max_{\delta\ge2}\mathfrak{I}_{\mathrm{IE},\mathrm{N}}^{(3,\delta)}$ if  $\beta\ge5.39469$. 
		
		Although the controllability intensities $\mathfrak{I}_{\mathrm{IE},\mathrm{N}}^{(\rmk,\delta)}$ of NIMEX-$\rmk$ methods \eqref{scheme: NIMEX multistep}	are not comparable to those of the GBDF-$\rmk$ schemes, we find that 
		the NIMEX-$\rmk$ methods \eqref{scheme: NIMEX multistep} for $2\le\rmk\le9$ fulfill the requirements of Lemma \ref{lemma: bound quadratic form} for proper parameter $\delta$. 
		That is, our theory in Sections \ref{sec: semi-generating function method}-\ref{sec: stability of IELM} is applicable for  the NIMEX-$\rmk$ methods up to the ninth-order accuracy, while  the application and the numerical analysis are left to a separate report. 
		Actually, $\varrho_{a,\mathrm{N}}^{(\rmk)}(\zeta)$  is a Schur polynomial if the parameter $\delta>\frac{1}{2}$, $\delta>\frac{1}{2}$, $\delta>\frac{1}{2}$, $\delta >0.705247$, $\delta>0.856554$, $\delta>1.0241$, $\delta>1.20559$ and $\delta>1.39984$ corresponding to the order index $\rmk=2,3,\cdots,8$ and 9, respectively. 
		This interesting property can be seen again in the next subsection, where we discuss a simplified version of NIMEX-$\rmk$ methods.
		
		\subsection{$\gamma$-parameterized SIELM methods}
		
		As an alternative to the GBDF-$\rmk$ schemes, this subsection discusses a new class of $\gamma$-parameterized SIELM methods for which the associated implicit-explicit controllability intensity $\mathfrak{I}_{\mathrm{IE}}^{(\rmk)}$ can approach the optimal value 1 as the parameter $\gamma$ is properly large, especially for $2\le\rmk\le5$. For the nonlinear parabolic model \eqref{nonlinearModel}, they can be formulated as follows
		\begin{align}\label{scheme: our IELM multistep}	
			\sum_{j=0}^{\rmk-1}a_{\mathrm{S},j}^{(\rmk)}\partial_{\tau}u^{n-j}
			+\varpi\sum_{j=0}^{\rmk}b_{\mathrm{S},j}^{(\rmk)}\mathcal{L}u^{n-j}
			=\sum_{j=0}^{\rmk-1}c_{\mathrm{S},j}^{(\rmk)}\mathcal{F}(u^{n-j-1})\quad\text{for $n\ge\rmk$,}
		\end{align}
		where  $a_{\mathrm{S},j}^{(\rmk)}$, $b_{\mathrm{S},j}^{(\rmk)}$ and $c_{\mathrm{S},j}^{(\rmk)}$ are determined by the three characteristic polynomials 
		$${\varrho}_{a,\mathrm{S}}^{(\rmk)}(\zeta):=\sum_{j=1}^{\rmk}\frac{f_{\mathrm{S}}^{(j)}(1)}{j!}(\zeta-1)^{j-1}\quad \text{with}\quad f_{\mathrm{S}}(z)=(\gamma z-\gamma+1)^{\rmk-1}z\ln z,$$ 
		$\varrho_{b,\mathrm{S}}^{(\rmk)}(\zeta):=\zeta(\gamma\zeta-\gamma+1)^{\rmk-1}$ and $\varrho_{c,\mathrm{S}}^{(\rmk)}(\zeta):=\zeta(\gamma\zeta-\gamma+1)^{\rmk-1}-\gamma^{\rmk-1}(\zeta-1)^{\rmk}$,
		respectively. Note that, Lemma \ref{lemma: bound quadratic form} is applicable to the SIELM-$\rmk$ methods \eqref{scheme: NIMEX multistep} for $2\le\rmk\le9$. One can check that  all roots of $\varrho_{a,\mathrm{S}}^{(\rmk)}(\zeta)$ satisfy $\abs{\zeta}<1$ if  $\gamma>0$, $\gamma>\frac{1}{\sqrt{6}}$, $\gamma>\frac{1}{2}$, $\gamma >0.64752$, $\gamma >0.83037$, $\gamma >1.02816$, $\gamma>1.23687$ and $\gamma>1.45372$ corresponding to the order index $\rmk=2,3,\cdots,8$ and 9, respectively.
		
		We see that the SIELM-$\rmk$ methods \eqref{scheme: our IELM multistep} are also generalized versions of WBDF2 scheme since the case of $\rmk=2$ is just the WBDF2 or GBDF2 scheme. Like the GBDF-$\rmk$ methods, the discrete coefficients $a_{\mathrm{S},j}^{(\rmk)}$, $b_{\mathrm{S},j}^{(\rmk)}$ and $c_{\mathrm{S},j}^{(\rmk)}$ always involve $\rmk-1$ degree polynomials with respect to $\gamma$. Unlike the GBDF-$\rmk$ methods which attain the maximum consistency order of five, the SIELM-$\rmk$ methods for $6\le\rmk\le9$ fulfill the requirements of Lemma \ref{lemma: bound quadratic form} for certain ranges of the parameter $\gamma$.  That is, our stability theory in Sections \ref{sec: semi-generating function method}-\ref{sec: stability of IELM} is applicable for  the SIELM-$\rmk$ methods \eqref{scheme: our IELM multistep} up to the ninth-order time accuracy.

		\begin{table}[htb!]
			\centering
			\begin{threeparttable}
				\centering
				\caption{Stability property of SIELM-$\rmk$ methods}
				\vspace*{0.3pt}
				\def\temptablewidth{1\textwidth}
				\label{table: controllability intensity SIELM}
				\begin{tabular*}{\temptablewidth}{@{\extracolsep{\fill}}clc}
					\toprule
					$\rmk$&Range of $\gamma$ 
					&$\mathfrak{I}_{\mathrm{IE},\mathrm{S}}^{(\rmk,\gamma)}$\\  \midrule		
					$2$  &$[1,+\infty)$
					&$\tfrac{2\gamma-1}{2\gamma +1}$\\[5pt]
					$3$  &  $[1,+\infty)$
					& $\frac{(2\gamma-1)^2}{4 \gamma ^2+4 \gamma -1}$  \\[5pt]
					$4$ &	 $[6/5,+\infty)$		
					&  $\frac{(2 \gamma -1)^3}{8 \gamma ^3+12 \gamma ^2-6 \gamma +1}$ \\[5pt]
					$5$ & 	 $[7/5,+\infty)$
					& $\frac{(2\gamma-1)^4}{16 \gamma ^4+32 \gamma ^3-24 \gamma ^2+8 \gamma -1}$\\[5pt]
					$6$ &$[10/5,80]$
					& $\frac{(2 \gamma -1)^5}{32 \gamma ^5+80 \gamma ^4-80 \gamma ^3+40 \gamma ^2-10 \gamma +1}$
					\\[5pt]
					$7$ &$[11/5,60]$
					& {$\frac{(2 \gamma-1 )^6}{64 \gamma ^6+192 \gamma ^5-240 \gamma ^4+160 \gamma ^3-60 \gamma ^2+12 \gamma -1}$}
					\\[5pt]
					$8$ &$[13/5,40]$
					& {  $\frac{(2 \gamma -1)^7}{128 \gamma ^7+448 \gamma ^6-672 \gamma ^5+560 \gamma ^4-280 \gamma ^3+84 \gamma ^2-14 \gamma +1}$}\\[5pt]
					$9$ &$[15/5,25]$
					& {$ \frac{(2\gamma-1)^8}{256\gamma^8 + 1024\gamma^7 - 1792\gamma^6 + 1792\gamma^5 - 1120\gamma^4 + 448\gamma^3 - 112\gamma^2 + 16\gamma - 1}$}
				\end{tabular*}
				{\rule{\temptablewidth}{0.5pt}}
				\footnotesize
			\end{threeparttable}
		\end{table}

		For the semi-generating functions $a_{\mathrm{S}}^{(\rmk)}(\theta)$, $b_{\mathrm{S}}^{(\rmk)}(\theta)$
		and $c_{\mathrm{S}}^{(\rmk)}(\theta)$, defined via \eqref{matrix: A_L B_L C_L generating function},
		of the SIELM-$\rmk$ methods \eqref{scheme: our IELM multistep}, we find that the functions $\frac1{\abst{a_{\mathrm{S}}^{(\rmk)}(\theta)}}$, $\frac{\abst{c_{\mathrm{S}}^{(\rmk)}(\theta)}}{\abst{a_{\mathrm{S}}^{(\rmk)}(\theta)}}$ and
		$\Re\kbrab{\tfrac{b_{\mathrm{S}}^{(\rmk)}(\theta)}{a_{\mathrm{S}}^{(\rmk)}(\theta)}}$ are unimodal over $\theta\in[0,2\pi)$ for properly large parameter $\gamma$ so that their extreme values can be attained at $\theta=0$ or $\theta=\pi$. For the order indexes $2\le \rmk\le 9$, we have
		\begin{align}\label{intensity: lambda sigma-SIELM}
			\sigma_{\mathrm{F},\mathrm{S}}^{(\rmk)}=\frac{1}{\absb{a_{\mathrm{S}}^{(\rmk)}(0)}}=1,\quad
			\sigma_{\mathrm{E},\mathrm{S}}^{(\rmk)}=
			\frac{\absb{c_{\mathrm{S}}^{(\rmk)}(\pi)}}{\absb{a_{\mathrm{S}}^{(\rmk)}(\pi)}},	\quad
			\lambda_{\mathrm{I},\mathrm{S}}^{(\rmk)}=\frac{b_{\mathrm{S}}^{(\rmk)}(\pi)}{a_{\mathrm{S}}^{(\rmk)}(\pi)}
		\end{align}
		if $\gamma$ satisfies $\gamma_{\rmk}\le \gamma\le \gamma_{\rmk}^*$, see  Propositions \ref{prop: SIELM3}-\ref{prop: SIELM9} in section \ref{sec: SM-SIELM} of the supplementary material for the cases $\rmk=3,4,\cdots,8$ and $9$, respectively.
		Table \ref{table: controllability intensity SIELM} collects the value of    $\mathfrak{I}_{\mathrm{IE},\mathrm{S}}^{(\rmk,\gamma)}$ with the ranges of $\gamma$.

		It is easy to check that the controllability intensity $\mathfrak{I}_{\mathrm{IE},\mathrm{S}}^{(\rmk,\gamma)}$ is always increasing as $\gamma$ increases so that the SIELM-$\rmk$ ($2\le\rmk\le9$) schemes enhance the applicability to the nonlinear parabolic problem \eqref{nonlinearModel} as $\gamma$ increases.
		Like the GBDF-$\rmk$ methods, the improvement of implicit-explicit controllability intensity $\mathfrak{I}_{\mathrm{IE},\mathrm{S}}^{(\rmk)}$ as $\gamma$ increases at the expense of rapidly increased truncation error, referred to the leading truncation errors in \eqref{truncation: SIELM3}, \eqref{truncation: SIELM4},  \eqref{truncation: SIELM5},  \eqref{truncation: SIELM6},  \eqref{truncation: SIELM7},  \eqref{truncation: SIELM8} and  \eqref{truncation: SIELM9}. 
		For the cases $2\le\rmk\le5$, one can find some $\gamma$ to satisfy the stability condition \eqref{stability condition} for any $0<\mu_0<\varpi$. 
		The parameter restriction $\gamma_{\rmk}\le \gamma\le \gamma_{\rmk}^*$ of  SIELM-$\rmk$ method for $6\le\rmk\le9$ gives the maximum value of  controllability intensity as follows,
		\begin{align*}
			&\,\max_{\gamma}\mathfrak{I}_{\mathrm{IE},\mathrm{S}}^{(6, \gamma)}=\mathfrak{I}_{\mathrm{IE},\mathrm{S}}^{(6, 80)}\approx0.940124,\quad
			\max_{\gamma}\mathfrak{I}_{\mathrm{IE},\mathrm{S}}^{(7, \gamma)}=\mathfrak{I}_{\mathrm{IE},\mathrm{S}}^{(7, 60)}\approx0.906633,\\
			&\,\max_{\gamma}\mathfrak{I}_{\mathrm{IE},\mathrm{S}}^{(8, \gamma)}=\mathfrak{I}_{\mathrm{IE},\mathrm{S}}^{(8, 40)}\approx0.844531,\quad 
			\max_{\gamma}\mathfrak{I}_{\mathrm{IE},\mathrm{S}}^{(9, \gamma)}=\mathfrak{I}_{\mathrm{IE},\mathrm{S}}^{(9, 25)}\approx0.740285.
		\end{align*}
		
		The SIELM-$\rmk$ methods ($2\le\rmk\le9$) always have better adaptability to the nonlinear parabolic equation \eqref{nonlinearModel} than the WBDF-$\rmk$  and NIMEX-$\rmk$ methods for $3\le \rmk\le5$. 
		Typically, the controllability intensity $\mathfrak{I}_{\mathrm{IE},\mathrm{S}}^{(\rmk,\gamma)}$ of the SIELM-$\rmk$ ($2\le\rmk\le9$) method can be greater than the mathematical threshold  $\mu_0/\varpi=\sqrt{2}/2$ required for the well-known consistent splitting algorithm of the incompressible Navier-Stokes equations. As verified theoretically and experimentally in \cite{KangLiaoLiu:2027}, the SIELM-$\rmk$ ($2\le\rmk\le9$) method fulfilling the stability condition \eqref{stability condition} can maintain the unconditional stability of the corresponding consistent splitting IELM algorithm.

		Note that, the value of  $\lambda_{\mathrm{I},\mathrm{S}}^{(\rmk,\gamma)}/\sigma_{\mathrm{F},\mathrm{S}}^{(\rmk,\gamma)}=\lambda_{\mathrm{I},\mathrm{S}}^{(\rmk,\gamma)}$, which is closely related to the admissible maximum step size, slowly increases as the temporal order $\rmk$ increases. For a high-order SIELM method, appropriately increasing $\gamma$ within its admissible range can make the admissible 
			maximum step size comparable to that of the SIELM2 scheme; see the numerical tests in
			\cite[Examples 1 and 2]{KangLiaoLiu:2027}.
		
		\section{Conclusions}
		
		A novel semi-generating function approach combined with a global discrete energy analysis is suggested for the stability and convergence analysis of general IELM methods in solving nonlinear parabolic equations. 
		Compared with the existing discrete energy approaches based on Dahlquist's G-stability theory \cite{Dahlquist:1978} with the Nevanlinna-Odeh-type multipliers \cite{Akrivis:2015,AkrivisKatsoprinakis:2016,AkrivisLubich:2015,LubichMansourVenkataraman:2013}  or the implicit part  decompositions \cite{HuangShen:2024,HuangShen:2025mcom}, 
		the unified framework is theoretically concise and applicable for a wide class of parameterized IELM methods. Generally, the  delicate constructions of Nevanlinna-Odeh-type multipliers or implicit part decompositions are always avoided or transformed via the three extreme values, including the perturbation amplification factor $\sigma_{\mathrm{F}}^{(\rmk)}$, the nonlinear amplification factor $\sigma_{\mathrm{E}}^{(\rmk)}$ and the dissipation preserving factor $\lambda_{\mathrm{I}}^{(\rmk)}$, of the three univariate real-valued functions $\frac1{\abst{a^{(\rmk)}(\theta)}}$, $\frac{\abst{c^{(\rmk)}(\theta)}}{\abst{a^{(\rmk)}(\theta)}}$ and
		$\Re\kbrab{\tfrac{b^{(\rmk)}(\theta)}{a^{(\rmk)}(\theta)}}$ in the bounded interval $\theta\in[0,2\pi)$. 
		
		This greatly facilitates our theoretical analysis on the stability of various IELM methods. Four parameterized IELM methods, including the existing $\alpha$-parameterized WBDF, $\beta$-parameterized  GBDF, $\delta$-parameterized  NIMEX methods and a new $\gamma$-parameterized class of SIELM methods, are evaluated in detail for possible applications to \eqref{nonlinearModel} by evaluating the bounds of $\lambda_{\mathrm{I}}^{(\rmk)}$, $\sigma_{\mathrm{E}}^{(\rmk)}$, $\sigma_{\mathrm{F}}^{(\rmk)}$ and  $\mathfrak{I}_{\mathrm{IE}}^{(\rmk)}$. 
		According to the theoretical range of  controllability intensity $\mathfrak{I}_{\mathrm{IE}}^{(\rmk)}$, it seems that the  GBDF-$\rmk$ ($2\le \rmk\le5$) and  SIELM-$\rmk$ ($2\le \rmk\le9$) methods always have better adaptability than other existing IELM schemes for the nonlinear parabolic model \eqref{nonlinearModel}.
		
		To further demonstrate the effectiveness of our theory, we will present some typical applications together with extensive numerical experiments in forthcoming reports. These studies will include high-order decoupled IELM algorithms for the incompressible Navier-Stokes equation and the long-time stability of high-order IELM methods for nonlinear dissipative dynamics. The implicit-explicit controllability intensity or its variants will be used to identify a suitable candidate among the IELM methods.
		
		\section*{Acknowledgment}
		The authors would like to thank Dr. Yuanyuan Kang for her careful reading of the manuscript, thank Dr. Bingquan Ji and Dr. Xuping Wang for their helpful discussions  on Lemma \ref{lemma: bound quadratic form}, and thank Dr. Zhuo Chen for checking and updating all the theoretical and numerical computations in Section 4. We also thank the three anonymous referees for their valuable comments and suggestions, which are very helpful in improving the quality of the paper.
		
		\begin{appendix}
			
			\section{Proofs of Lemmas \ref{lem: Toeplitz-Caratheodory}-\ref{lemma: spectral norm bound}}
			\label{sec: appendx Lemmas 2.1-2.3}
			\setcounter{equation}{0}
			
			\begin{proof}[Proof of Lemma \ref{lem: Toeplitz-Caratheodory}]
				For the	real sequence $\{a_{0},a_{1},\cdots,a_{n},\cdots\}$, define $\bar{a}_{0}=a_0$, while
				$\bar{a}_{k}=a_k/2$ and $\bar{a}_{-k}=a_k/2$ for $k\ge1$.
				Then the real symmetric Toeplitz matrix $\mathcal{S}(P_{L, n})=(\bar{a}_{i-j})$ with 
				the entries
				$\bar{a}_{ij}=\bar{a}_{i-j}$ for $i,j\ge1$ being constants along the diagonals of $\mathcal{S}(P_{L, n})$.
				According to \cite[Section 1.10]{GrenanderSzego:2001},
				let $\bar{a}_{k}$ be the Fourier coefficients of the trigonometric polynomial $\mathrm{g}$, that is,
				$\bar{a}_{k}=\frac{1}{2\pi}\int_{0}^{2\pi}\mathrm{g}(\theta)e^{-\imath k \theta}\zd \theta.$ Then 
				the standard generating function of the Toeplitz matrix $\mathcal{S}(P_{L, n})$ reads
				\begin{align*}
					\mathrm{g}(\theta):=\sum_{k=-\infty}^{\infty}\bar{a}_{k}e^{\imath k \theta}.
				\end{align*}
				The Grenander-Szeg\H{o} theorem \cite[pp. 64-65]{GrenanderSzego:2001} gives the relationship
				between the eigenvalues of $\mathcal{S}(P_{L, n})$ and the standard  generating function $\mathrm{g}$. That is, 
				the Toeplitz matrix $\mathcal{S}(P_{L, n})$ is positive definite if $\min_{\theta\in[0,2\pi)}\mathrm{g}(\theta)>0$, and  the associated eigenvalues $\lambda_j\kbrab{\mathcal{S}(P_{L, n})}$  can be bounded by
				$$\min_{\theta\in[0,2\pi)}\mathrm{g}(\theta)\le\lambda_j\kbrab{\mathcal{S}(P_{L, n})}\le \max_{\theta\in[0,2\pi)}\mathrm{g}(\theta)\quad\text{for $j=0,1,\cdots,n-1$}.$$
				Moreover, the eigenvalues $\lambda_j\kbrab{\mathcal{S}(P_{L, n})}$ for $j=0,2,\cdots,n-1$ 
				are equally distributed as $\mathrm{g}(\frac{2\pi j}{n})$ in the sense that
				$$\lim_{n\rightarrow\infty}\frac1{n}\sum_{j=0}^{n-1}\kbraB{\lambda_j(Q_n)-\mathrm{g}(\frac{2\pi j}{n})}=0.$$
				Thus the results (i)-(iii) follow immediately due to the following fact
				\begin{align*}
					\mathrm{g}(\theta)=&\,a_0+\sum_{k=1}^{\infty}\bar{a}_{-k}e^{-\imath k \theta}
					+\sum_{k=1}^{\infty}\bar{a}_{k}e^{\imath k \theta}
					=a_0+\sum_{k=1}^{\infty}a_{k}\cos k \theta=\Re\kbrab{a(\theta)}.
				\end{align*}
				It completes the proof.
			\end{proof}
			
			\begin{proof}[Proof of Lemma \ref{lem: composited generating function}] By exchanging the order of summation, the result (i) follows immediately,
				\begin{align*}
					\hat{b}(\theta)=\sum_{j=0}^{\infty}\braB{\sum_{k=0}^{j}a_{j-k}b_{k}}e^{\imath j\theta}
					=\sum_{k=0}^{\infty}b_{k}e^{\imath k\theta}\sum_{j=k}^{\infty}a_{j-k}e^{\imath (j-k)\theta}=a(\theta)b(\theta).
				\end{align*}
				Since $\sum_{k=0}^{j}\xi_{j-k}a_{k}=\delta_{j0}$ for any $j\ge0$, we have
				\begin{align*}
					1=\sum_{j=0}^{\infty}\braB{\sum_{k=0}^{j}\xi_{j-k}a_{k}}e^{\imath j\theta}
					=\sum_{k=0}^{\infty}a_{k}e^{\imath k\theta}\sum_{j=k}^{\infty}\xi_{j-k}e^{\imath (j-k)\theta}=a(\theta)\xi(\theta).
				\end{align*}
				This relation gives the result (ii) and completes the proof.
			\end{proof}
			
			\begin{proof}[Proof of Lemma \ref{lemma: spectral norm bound}] In this proof, the notation $\mynorm{\cdot}_{\ell^2}$ is used to denote the $\ell^2$ norm of vector and associated matrix (operator) norm. For the given semi-generating function $a(\theta)=\sum_{k=0}^{\infty}a_k e^{\imath k\theta }$, we define a lower triangular infinite  matrix (operator) $P_{L}$  as follows
				\begin{align*}
					P_{L}:=\left(
					\begin{array}{ccccc}
						a_0 & && &\\
						a_1 &a_0 & &&\\
						\vdots            &    \ddots          & \ddots & &\\
						a_{n-1}&\cdots &a_1&a_0&\\
						\vdots&\cdots &\vdots&\ddots&\ddots   
					\end{array}
					\right).
				\end{align*}
				Thus, the $m$-th component $$(P_{L}\myvec{x})_m=\sum_{j=0}^ma_{m-j}x_j$$ for 
				$\myvec{x}=(x_0,x_1,\cdots,x_m,\cdots)^T$. The  matrix $P_{L,n}$ is a restriction (the first $n$ components) of the operator $P_{L}$ on $\mathbb{C}^n$.
				For any vector $\myvec{z}=(z_0,z_1,\cdots,z_{n-1})^T\in \mathbb{C}^n$, define $\tilde{\myvec{z}}=(z_0,z_1,\cdots,z_{n-1},0,0,\cdots)^T\in \ell^2(\mathbb{N})$ such that
				$(P_{L,n}\myvec{z})_j=(P_{L}\tilde{\myvec{z}})_j$ for $1\le j\le n$.
				One has
				$$\mynorm{P_{L,n}\myvec{z}}_{\ell^2}=\mynorm{P_{L}\tilde{\myvec{z}}}_{\ell^2}\le 
				\mynorm{P_{L}}_{\ell^2}\mynorm{\tilde{\myvec{z}}}_{\ell^2}=\mynorm{P_{L}}_{\ell^2}\mynorm{\myvec{z}}_{\ell^2}\quad\text{for any $\myvec{z}\in \mathbb{C}^n$}.$$
				It gives the spectral norm
				\begin{align*}
					\mynorm{P_{L,n}}_{\ell^2}:=\mathop{\rm{sup}}\limits_{\myvec{z}\neq\myvec{0}} \frac{\mynorm{P_{L,n}\myvec{z}}_{\ell^2}}{\mynorm{\myvec{z}}_{\ell^2}}\le \mynorm{P_{L}}_{\ell^2}.
				\end{align*}
				
				It remains to show that $\mynorm{P_{L}}_{\ell^2}\le \max_{\theta\in[0,2\pi)}\abs{a(\theta)}$. For $\myvec{x}\in \ell^2(\mathbb{N})$ with the discrete Fourier transform $X(\theta)=\sum_{j=0}^{\infty}x_je^{-\imath j\theta}$, 
				consider $\myvec{y}=P_{L}\myvec{x}\in \ell^2(\mathbb{N})$ with the discrete Fourier transform $Y(\theta)=\sum_{k=0}^{\infty}y_ke^{-\imath k\theta}$. 
				Since $y_k=\sum_{j=0}^ka_{k-j}x_j$, one has
				\begin{align*}
					Y(\theta)=\sum_{k=0}^{\infty}\braB{\sum_{j=0}^ka_{k-j}x_j}e^{-\imath k\theta}
					=&\,\sum_{j=0}^{\infty}x_je^{-\imath j\theta}\sum_{k=j}^{\infty}a_{k-j}e^{-\imath (k-j)\theta}
					=X(\theta)\overline{a(\theta)},
				\end{align*}
				where one can take $m:=k-j\ge0$ in the second equality and $\overline{a(\theta)}$ is the conjugate of $a(\theta)$. 
				Then, the Parseval theorem gives
				\begin{align*}
					\mynorm{\myvec{y}}_{\ell^2}^2=&\,\frac{1}{2\pi}\int_0^{2\pi}\abs{Y(\theta)}^2\zd\theta
					=\frac{1}{2\pi}\int_0^{2\pi}\absb{X(\theta)\overline{a(\theta)}}^2\zd\theta
					=\frac{1}{2\pi}\int_0^{2\pi}\absb{\overline{a(\theta)}}^2\absb{X(\theta)}^2\zd\theta\\
					\le&\,\frac{1}{2\pi}\int_0^{2\pi}\absb{X(\theta)}^2\zd\theta \max_{\theta\in[0,2\pi)}\abs{a(\theta)}^2
					=\mynorm{\myvec{x}}_{\ell^2}^2\cdot\max_{\theta\in[0,2\pi)}\abs{a(\theta)}^2.
				\end{align*}
				This means that 
					\begin{align*}
							\mynorm{P_{L}\myvec{x}}_{\ell^2}\le\mynorm{\myvec{x}}_{\ell^2}\max_{\theta\in[0,2\pi)}\abs{a(\theta)}
							\quad\text{for any $\myvec{x}\in \ell^2(\mathbb{N})$,}
						\end{align*}
				Therefore, we obtain $\mynorm{P_{L}}_{\ell^2}\le \max_{\theta\in[0,2\pi)}\abs{a(\theta)}$
				and complete the proof.
			\end{proof}
			
		\end{appendix}

		\newpage
		

		
		\def\siamprelabel{SM}
\setcounter{section}{0}
\setcounter{page}{1}


\begin{center}
	{\bfseries{\MakeUppercase{Supplementary Materials: A semi-generating function approach to the stability of implicit-explicit multistep methods for nonlinear parabolic equations}}}
	
	\vspace{1em}
	{\small HONG-LIN LIAO, CHAOYU QUAN, TAO TANG {\scriptsize AND} TAO ZHOU}
	\vspace{2em}	
\end{center}


CONTENT: This supplementary material includes some detailed calculations of the perturbation amplification factor  $\sigma_{\mathrm{F}}^{(\rmk)}$,  nonlinear amplification factor $\sigma_{\mathrm{E}}^{(\rmk)}$, dissipation preserving factor $\lambda_{\mathrm{I}}^{(\rmk)}$ and the associated implicit-explicit controllability intensity $\mathfrak{I}_{\mathrm{IE}}^{(\rmk)}:=\lambda_{\mathrm{I}}^{(\rmk)}/\sigma_{\mathrm{E}}^{(\rmk)}$ for four parameterized classes of 
implicit-explicit linear multistep (IELM) methods, including the weighted backward differentiation formulas (WBDF)  in [J. Tsinghua Univ., 31: 1-11], the implicit-explicit generalized backward differentiation formulas (GBDF)  in [SIAM J. Numer. Anal., 62: 1609-1637], the NIMEX methods in [SIAM J. Numer. Anal., 55: 2336-2360] and a simplified version (called  SIELM schemes) of the NIMEX methods.  Hereafter, we also use a triad $\brab{\vec{a}^{(\rmk)},\vec{b}^{(\rmk)},\vec{c}^{(\rmk)}}$ to represent 
an IELM method  with the three vectors
$\vec{a}^{(\rmk)}=\brab{a_{0}^{(\rmk)},a_{1}^{(\rmk)},\cdots,a_{\rmk-1}^{(\rmk)}}$, 
$\vec{b}^{(\rmk)}=\brab{b_{0}^{(\rmk)},b_{1}^{(\rmk)},\cdots,b_{\rmk}^{(\rmk)}}$
and $\vec{c}^{(\rmk)}=\brab{c_{0}^{(\rmk)},c_{1}^{(\rmk)},\cdots,c_{\rmk-1}^{(\rmk)}}$.

\section{$\alpha$-parameterized WBDF methods}\label{sec: SM-WBDF}

The WBDF-$\rmk$ $(2\le \rmk\le7)$ formulas \cite{LiXie:1991WBDF-SM} with a free parameter $\alpha$ are constructed by using 
the backward differentiation formula at the off-set grid point $t_{n-1+\alpha}$ for the implicit part, that is, 
\begin{align*}
	\sum_{j=0}^{\rmk-1}a_{\mathrm{W},j}^{(\rmk)}\partial_{\tau}u^{n-j}
	=&\,\alpha \varpi\mathcal{L}u^n+(1-\alpha)\varpi\mathcal{L}u^{n-1}
	\quad\text{for $n\ge\rmk$.} 	
\end{align*}
As shown in \cite[Theorem 4]{LiXie:1991WBDF-SM}, the WBDF2 method is A-stable if $\alpha\ge\frac12$; while \cite[Theorem 5]{LiXie:1991WBDF-SM} states that the WBDF-$\rmk$ ($\rmk=3,4,5$) methods are $A(\theta)$-stable if $\alpha>\frac12$, and the WBDF-$\rmk$ ($\rmk=6,7$) methods are $A(\theta)$-stable if $\alpha\ge\frac{13}{5}$.  Furthermore, the absolute stability regions of 
WBDF-$\rmk$ methods always enlarge as the parameter $\alpha$ increases.
By using the order conditions in \eqref{scheme: multistep-linear order conditions}, one has the associated implicit-explicit WBDF-$\rmk$ methods for the nonlinear parabolic model \eqref{nonlinearModel},
\begin{align}\label{scheme: WBDF imex multistep-SM}	
	\sum_{j=0}^{\rmk-1}a_{\mathrm{W},j}^{(\rmk)}\partial_{\tau}u^{n-j}
	+\varpi\sum_{j=0}^{\rmk}b_{\mathrm{W},j}^{(\rmk)}\mathcal{L}u^{n-j}
	=\sum_{j=0}^{\rmk-1}c_{\mathrm{W},j}^{(\rmk)}\mathcal{F}(u^{n-j-1})\quad\text{for $n\ge\rmk$,}
\end{align}
where we set $b_{\mathrm{W},0}^{(\rmk)}=\alpha$, $b_{\mathrm{W},1}^{(\rmk)}=1-\alpha$ and $b_{\mathrm{W},j}^{(\rmk)}=0$ for $2\le j\le \rmk$. 
Actually, this settings of $b_{\mathrm{W},j}^{(\rmk)}$ uniquely determine the coefficients $a_{\mathrm{W},j}^{(\rmk)}$ and $c_{\mathrm{W},j}^{(\rmk)}$, and also make $a_{\mathrm{W},j}^{(\rmk)}$ and $c_{\mathrm{W},j}^{(\rmk)}$ linear polynomials with respect to $\alpha$.  
The corresponding three characteristic polynomials read
$${\varrho}_{a,\mathrm{W}}^{(\rmk)}(\zeta):=\sum_{j=1}^{\rmk}\frac{f_{\mathrm{W}}^{(j)}(1)}{j!}(\zeta-1)^{j-1}\quad \text{with}\quad f_{\mathrm{W}}(z)=(\alpha z-\alpha+1)z^{\rmk-1}\ln z,$$ 
$\varrho_{b,\mathrm{W}}^{(\rmk)}(\zeta):=\zeta^{\rmk-1}(\alpha\zeta-\alpha+1)$ and $\varrho_{c,\mathrm{W}}^{(\rmk)}(\zeta):=\zeta^{\rmk-1}(\alpha\zeta-\alpha+1)-\alpha(\zeta-1)^{\rmk}$.
Without special declarations, here we consider the stability property of WBDF-$\rmk$ ($2\le\rmk\le5$) methods for $\alpha\ge1$. Note that, the WBDF6 scheme satisfies the positivity condition $\lambda_{\mathrm{I},\mathrm{W}}^{(6,\alpha)}>0$ in Lemma \ref{lemma: bound quadratic form} for $\alpha>4.69$, but its associated implicit-explicit controllability intensity is very small, $\mathfrak{I}_{\mathrm{IE},\mathrm{W}}^{(6,\alpha)}<1/90$, limiting its  applicability to the nonlinear model \eqref{nonlinearModel}. The WBDF7 scheme fails to satisfy this positivity condition for any $\alpha\ge1$.


\subsection{WBDF2 scheme}
By taking $b_{\mathrm{W},0}^{(2)}=\alpha$, $b_{\mathrm{W},2}^{(2)}=0$ in the linear system \eqref{scheme: multistep-linear order conditions} with $q=2$, one has the WBDF2 scheme with the discrete  coefficients
\begin{align}\label{scheme: WBDF2}
	&\vec{a}_{\mathrm{W}}^{(2)}=(\tfrac{1}{2}+\alpha,\tfrac{1}{2}-\alpha),\quad
	\vec{b}_{\mathrm{W}}^{(2)}=(\alpha,1-\alpha,0),\quad
	\vec{c}_{\mathrm{W}}^{(2)}=(\alpha +1,-\alpha ).
\end{align}
The first characteristic polynomial ${\varrho}_{a,\mathrm{W}}^{(2)}(\zeta)$ is a Schur polynomial if $\alpha>0$.  The formula \eqref{scheme: general imex Truncation Error} gives the leading error 
\begin{align}\label{truncation: WBDF2}
	R_{\mathrm{W}}^{(2,\alpha)}=\frac{1-3 \alpha}{6}u_t^{(3)}(t_n)\tau^2+\alpha \mathcal{F}_t^{(2)}[u(t_n)]\tau^2.
\end{align}
We have the following result.
\begin{proposition}\label{prop: WBDF2}
	For the WBDF2 scheme \eqref{scheme: WBDF2} with the free parameter $\alpha\ge1$, it holds that
	\begin{align*}
		\sigma_{\mathrm{F},\mathrm{W}}^{(2,\alpha)}=1
		,\quad	\sigma_{\mathrm{E},\mathrm{W}}^{(2,\alpha)}=\frac{2\alpha+1}{2\alpha},\quad 
		\lambda_{\mathrm{I},\mathrm{W}}^{(2,\alpha)}=\frac{2\alpha-1}{2\alpha}\quad\text{such that }\quad
		\mathfrak{I}_{\mathrm{IE},\mathrm{W}}^{(2,\alpha)}=\frac{2\alpha-1}{2\alpha+1}.
	\end{align*}
\end{proposition}
\begin{proof}
	The  semi-generating functions are $a_{\mathrm{W}}^{(2)}(\theta)=\tfrac{1}{2}+\alpha+e^{\imath\theta}(\tfrac{1}{2}-\alpha)$,
	$b_{\mathrm{W}}^{(2)}(\theta)=\alpha+e^{\imath\theta}(1-\alpha)$ and $c_{\mathrm{W}}^{(2)}(\theta)=1+\alpha-e^{\imath\theta}\alpha$. 
	It is not difficult to check that the functions $\Re\kbra{\frac{p_1+p_2e^{\imath\theta}}{p_3+p_4e^{\imath\theta}}}$ and $\abs{\frac{p_1+p_2e^{\imath\theta}}{p_3+p_4e^{\imath\theta}}}$ with real coefficients $p_i$ $(1\le i\le 4)$ always achieve the extreme values  
	at $\cos\theta=\pm1$. Then the definitions in \eqref{def: lambda sigma} yield $\sigma_{\mathrm{F},\mathrm{W}}^{(2)}=1=\tfrac{1}{\abst{a_{\mathrm{W}}^{(2)}(0)}}$,
	$\sigma_{\mathrm{E},\mathrm{W}}^{(2,\alpha)}=\frac{2\alpha+1}{2\alpha}=\frac{\abst{c_{\mathrm{W}}^{(2)}(\pi)}}{\abst{a_{\mathrm{W}}^{(2)}(\pi)}}$ and $\lambda_{\mathrm{I},\mathrm{W}}^{(2,\alpha)}=\frac{2\alpha-1}{2\alpha}=\frac{b_{\mathrm{W}}^{(2)}(\pi)}{a_{\mathrm{W}}^{(2)}(\pi)}$.
	It completes the proof.
\end{proof}

\subsection{WBDF3 scheme} For  $\rmk=3$, we take $b_{\mathrm{W},0}^{(3)}=\alpha$, $b_{\mathrm{W},2}^{(3)}=0$ and $b_{\mathrm{W},3}^{(3)}=0$ in the linear system \eqref{scheme: multistep-linear order conditions} with $q=3$ and recover the WBDF3 scheme with
\begin{align}\label{scheme: WBDF3}
	&\vec{a}_{\mathrm{W}}^{(3)}=(\tfrac{3}{2}\alpha +\tfrac{1}{3},\tfrac{5}{6}-2 \alpha ,\tfrac{1}{2}\alpha -\tfrac{1}{6}),\quad
	\vec{b}_{\mathrm{W}}^{(3)}=(\alpha ,1-\alpha, 0, 0),\\
	&\vec{c}_{\mathrm{W}}^{(3)}=(2 \alpha +1,-3 \alpha ,\alpha).\notag
\end{align}
The first characteristic polynomial ${\varrho}_{a,\mathrm{W}}^{(3)}(\zeta)$ is a Schur polynomial if $\alpha>\frac16$. The formula \eqref{scheme: general imex Truncation Error} gives the leading error 
\begin{align}\label{truncation: WBDF3}
	R_{\mathrm{W}}^{(3,\alpha)}=\frac{1-4 \alpha}{12}u_t^{(4)}(t_n)\tau^3+\alpha \mathcal{F}_t^{(3)}[u(t_n)]\tau^3.
\end{align}
We have the following proposition.
\begin{proposition}\label{prop: WBDF3}
	For the WBDF3 scheme \eqref{scheme: WBDF3} with the free parameter $\alpha\ge1$, it holds that $\sigma_{\mathrm{F},\mathrm{W}}^{(3,\alpha)}=1,$
	\begin{align*}
		\sigma_{\mathrm{E},\mathrm{W}}^{(3,\alpha)}=\frac{3(6\alpha+1)}{2(6\alpha-1)},\quad \lambda_{\mathrm{I},\mathrm{W}}^{(3,\alpha)}=\frac{3(2\alpha-1)}{2(6\alpha-1)}	\quad	\text{such that }	\quad	
		\mathfrak{I}_{\mathrm{IE},\mathrm{W}}^{(3,\alpha)}=\frac{2\alpha-1}{6\alpha+1}.
	\end{align*}
\end{proposition}
\begin{proof}
	The associated semi-generating functions are $$a_{\mathrm{W}}^{(3)}(\theta)=\tfrac{3}{2}\alpha +\tfrac{1}{3}+(\tfrac{5}{6}-2 \alpha)e^{\imath\theta}+(\tfrac{1}{2}\alpha -\tfrac{1}{6})e^{2\imath\theta},$$
	$b_{\mathrm{W}}^{(3)}(\theta)=\alpha+(1-\alpha)e^{\imath\theta}$ and $c_{\mathrm{W}}^{(3)}(\theta)=2 \alpha +1-3 \alpha e^{\imath\theta}+\alpha e^{2\imath\theta}$.
	The definitions in \eqref{def: lambda sigma} give the following bounds $\sigma_{\mathrm{F},\mathrm{W}}^{(3,\alpha)}\ge 
	\frac{1}{\abst{a_{\mathrm{W}}^{(3)}(0)}}=1$, $\sigma_{\mathrm{E},\mathrm{W}}^{(3,\alpha)}\ge 
	\frac{\abst{c_{\mathrm{W}}^{(3)}(\pi)}}{\abst{a_{\mathrm{W}}^{(3)}(\pi)}}=\frac{18\alpha+3}{12 \alpha-2}$ and $\lambda_{\mathrm{I},\mathrm{W}}^{(3,\alpha)}\le 
	\frac{b_{\mathrm{W}}^{(3)}(\pi)}{a_{\mathrm{W}}^{(3)}(\pi)}
	=\frac{6 \alpha-3}{12 \alpha-2}$.
	
	Moreover, direct calculations give
	\begin{align*}
		&a_{\mathrm{W}}^{(3)}(\theta)a_{\mathrm{W}}^{(3)}(-\theta)= 
		\frac{39 \alpha ^2-15 \alpha +5}{6}
		+\frac{-144 \alpha ^2+48 \alpha +5}{18}\cos\theta
		+\frac{27 \alpha ^2-3 \alpha -2}{18}\cos2 \theta,
		\\
		&\Re\kbrab{b_{\mathrm{W}}^{(3)}(\theta)a_{\mathrm{W}}^{(3)}(-\theta)}
		=\frac{21 \alpha ^2-15 \alpha +5}{6}+ \frac{-24 \alpha ^2+16 \alpha +1}{6}\cos\theta 
		+\frac{3 \alpha^2 -\alpha}{3}  \cos2 \theta,\\
		&c_{\mathrm{W}}^{(3)}(\theta)c_{\mathrm{W}}^{(3)}(-\theta)
		=14 \alpha ^2+4 \alpha +1-6 \alpha  (3 \alpha +1)\cos\theta 
		+2 \alpha  (2 \alpha +1)  \cos2 \theta.
	\end{align*}
	One has $27 \alpha ^2-3 \alpha -2>0$ for $\alpha\ge1$, and
	\begin{align*}
		\mathcal{F}_{\mathrm{W3}}	:=1-a_{\mathrm{W}}^{(3)}(\theta)a_{\mathrm{W}}^{(3)}(-\theta)
		=&\,\frac{1}{9}\sin^2\frac{\theta}{2}\kbrab{2\brab{27 \alpha ^2-3 \alpha -2} \cos\theta-\brab{90 \alpha ^2-42 \alpha -1}}\\
		\le&\,-\frac{1}{3}\brab{12 \alpha ^2-12 \alpha +1}\sin^2\frac{\theta}{2}\le0,
	\end{align*}
	so that $\absb{a_{\mathrm{W}}^{(3)}(\theta)}\ge 
	1$ for $\theta\in[0,2\pi)$. Thus we have $\sigma_{\mathrm{F},\mathrm{W}}^{(3,\alpha)}=1$.
	It is not difficult to check that $180 \alpha ^4-24 \alpha ^3-79 \alpha ^2+43 \alpha +2>0$ for $\alpha\ge1$. One has
	\begin{align*}
		\mathcal{E}_{\mathrm{W3}}	:=&\,(12\alpha-2)^2c_{\mathrm{W}}^{(3)}(\theta)c_{\mathrm{W}}^{(3)}(-\theta)-(18 \alpha+3)^2a_{\mathrm{W}}^{(3)}(\theta)a_{\mathrm{W}}^{(3)}(-\theta)\\
		=&\,\cos^2\frac{\theta}{2}\Big[2\brab{180 \alpha ^4-24 \alpha ^3-79 \alpha ^2+43 \alpha +2} \cos\theta\Big.\\
		&\,\hspace{2cm}\Big.-360 \alpha ^4+48 \alpha ^3-22 \alpha ^2-242 \alpha -9\Big]\\
		\le&\,-\brab{180 \alpha ^2+156 \alpha +5}\cos^2\frac{\theta}{2}\le0,
	\end{align*}
	and then $\frac{\abst{c_{\mathrm{W}}^{(3)}(\theta)}}{\abst{a_{\mathrm{W}}^{(3)}(\theta)}}\le 
	\frac{18 \alpha +3}{12 \alpha -2}$  for $\theta\in[0,2\pi)$.  It says that $\sigma_{\mathrm{E},\mathrm{W}}^{(3,\alpha)}=\frac{18\alpha+3}{12 \alpha-2}$.

	Moreover, since $18 \alpha ^3-15 \alpha ^2-3 \alpha +2>0$ for $\alpha\ge1$, one has
	\begin{align*}
		\mathcal{I}_{\mathrm{W3}}	:=&\,(12\alpha-2)\Re\kbrab{b_{\mathrm{W}}^{(3)}(\theta)a_{\mathrm{W}}^{(3)}(-\theta)}-(6 \alpha-3)a_{\mathrm{W}}^{(3)}(\theta)a_{\mathrm{W}}^{(3)}(-\theta)\\
		=&\,\frac{1}{3}\cos^2\frac{\theta}{2}\kbrab{\brab{36 \alpha ^3-30 \alpha ^2+12 \alpha +7}-2\brab{18 \alpha ^3-15 \alpha ^2-3 \alpha +2} \cos\theta}\\
		\ge&\,\brab{6\alpha+1}\cos^2\frac{\theta}{2}\ge0,
	\end{align*}
	and then $\Re\kbraB{\frac{b_{\mathrm{W}}^{(3)}(\theta)}{a_{\mathrm{W}}^{(3)}(\theta)}}\ge \frac{6 \alpha-3}{12 \alpha-2}$ for $\theta\in[0,2\pi)$.  It gives the value of the minimum eigenvalue, $\lambda_{\mathrm{I},\mathrm{W}}^{(3,\alpha)}=\frac{6 \alpha-3}{12 \alpha-2}$, and completes the proof.
\end{proof}

\subsection{WBDF4 scheme}  
For  $\rmk=4$, we take  $b_{\mathrm{W},0}^{(4)}=\alpha$ and $b_{\mathrm{W},j}^{(4)}=0$ $(j=2,3,4)$ in the linear system \eqref{scheme: multistep-linear order conditions} with $q=4$ and recover the WBDF4 scheme with
\begin{align}\label{scheme: WBDF4}
	&\vec{a}_{\mathrm{W}}^{(4)}=(\tfrac{22 \alpha+3 }{12},\tfrac{13-36\alpha}{12},\tfrac{18 \alpha -5}{12}, \tfrac{1-4 \alpha }{12}),\;\;
	\vec{b}_{\mathrm{W}}^{(4)}=(\alpha, 1-\alpha, 0, 0,0),\\
	&\vec{c}_{\mathrm{W}}^{(4)}=(3 \alpha +1,-6 \alpha, 4 \alpha,-\alpha).\notag
\end{align}
The first characteristic polynomial ${\varrho}_{a,\mathrm{W}}^{(4)}(\zeta)$ is a Schur polynomial if $\alpha>\frac15$.  The formula \eqref{scheme: general imex Truncation Error} gives the leading error 
\begin{align}\label{truncation: WBDF4}
	R_{\mathrm{W}}^{(4,\alpha)}=\frac{1-5 \alpha}{30}u_t^{(5)}(t_n)\tau^4+\alpha \mathcal{F}_t^{(4)}[u(t_n)]\tau^4.
\end{align}
We have the following proposition.
\begin{proposition}\label{prop: WBDF4}
	For the WBDF4 scheme \eqref{scheme: WBDF4} with $\alpha\ge1$,
	\begin{align*}
		&\sigma_{\mathrm{F},\mathrm{W}}^{(4,\alpha)}=1,\quad
		\sigma_{\mathrm{E},\mathrm{W}}^{(4,\alpha)}=\frac{3(14\alpha+1)}{4(5\alpha-1)},\quad
		\frac{120\alpha-61}{80(5\alpha-1)}\le
		\lambda_{\mathrm{I},\mathrm{W}}^{(4,\alpha)}\le\frac{3(2\alpha-1)}{4(5\alpha-1) }\\
		&\text{such that }\quad
		\frac{120\alpha-61}{60(14\alpha +1)}\le\mathfrak{I}_{\mathrm{IE},\mathrm{W}}^{(4,\alpha)}\le\frac{2\alpha-1}{14\alpha +1};
	\end{align*}
	and for the parameter $\alpha\ge6/5$,
	\begin{align*}
		\lambda_{\mathrm{I},\mathrm{W}}^{(4,\alpha)}=\frac{3(2\alpha-1)}{4(5\alpha-1) }\quad\text{such that }\quad
		\mathfrak{I}_{\mathrm{IE},\mathrm{W}}^{(4,\alpha)}=\frac{2\alpha-1}{14\alpha +1}.
	\end{align*}
\end{proposition}
\begin{proof}
	The associated semi-generating functions are
	$$a_{\mathrm{W}}^{(4)}(\theta)=\tfrac{22 \alpha+3 }{12}+\tfrac{13-36\alpha}{12}e^{\imath \theta}
	+\tfrac{18 \alpha -5}{12}e^{2\imath \theta}+\tfrac{1-4 \alpha }{12}e^{3\imath \theta},$$
	$b_{\mathrm{W}}^{(4)}(\theta)=\alpha+(1-\alpha)e^{\imath \theta}$ and $c_{\mathrm{W}}^{(4)}(\theta)=3 \alpha +1-6\alpha e^{\imath \theta}+4\alpha e^{2\imath \theta}-\alpha e^{3\imath \theta}$.
	The definitions in \eqref{def: lambda sigma} give the following bounds $\sigma_{\mathrm{F},\mathrm{W}}^{(4,\alpha)}\ge \frac{1}{\abst{a_{\mathrm{W}}^{(4)}(0)}}=1,$ $\sigma_{\mathrm{E},\mathrm{W}}^{(4,\alpha)}\ge 
	\frac{\abst{c_{\mathrm{W}}^{(4)}(\pi)}}{\abst{a_{\mathrm{W}}^{(4)}(\pi)}}=\frac{42 \alpha +3}{20 \alpha -4}$ and $\lambda_{\mathrm{I},\mathrm{W}}^{(4,\alpha)}\le 
	\frac{b_{\mathrm{W}}^{(4)}(\pi)}{a_{\mathrm{W}}^{(4)}(\pi)}
	=\frac{6 \alpha-3}{20 \alpha-4 }$ for $\alpha\ge1$.
	
	\begin{figure}[htb!]
		\centering
		\subfigure[ $\widetilde{\mathcal{F}}_{\mathrm{W4}}(\alpha,\theta)$]
		{\includegraphics[width=1.67in]{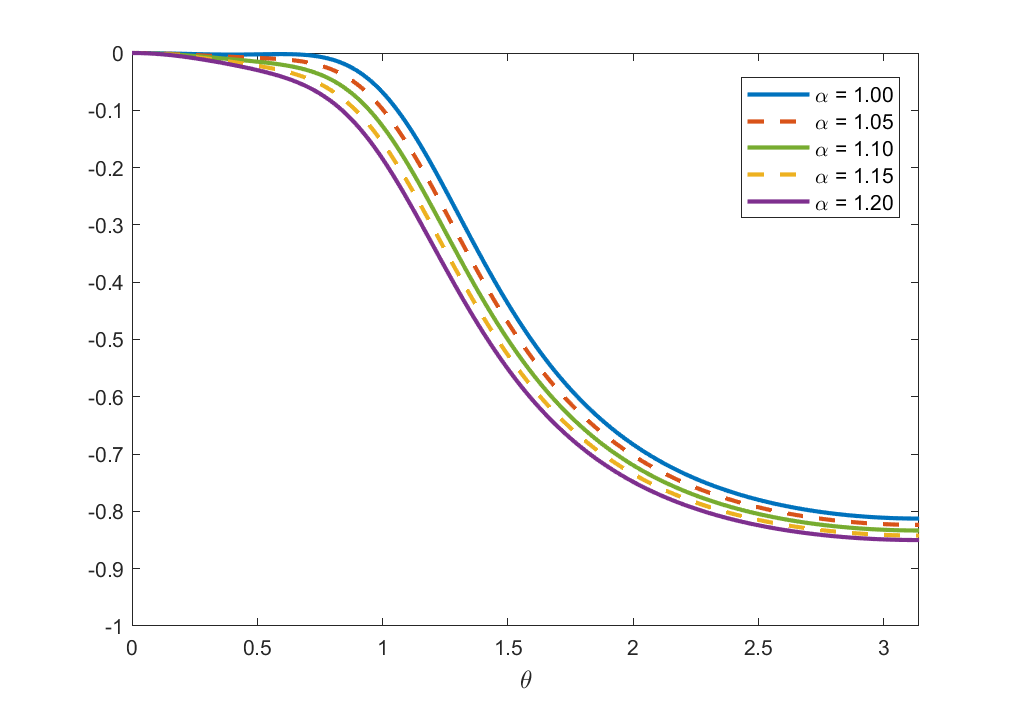}}
		\subfigure[ $\widetilde{\mathcal{E}}_{\mathrm{W4}}(\alpha,\theta)$]
		{\includegraphics[width=1.67in]{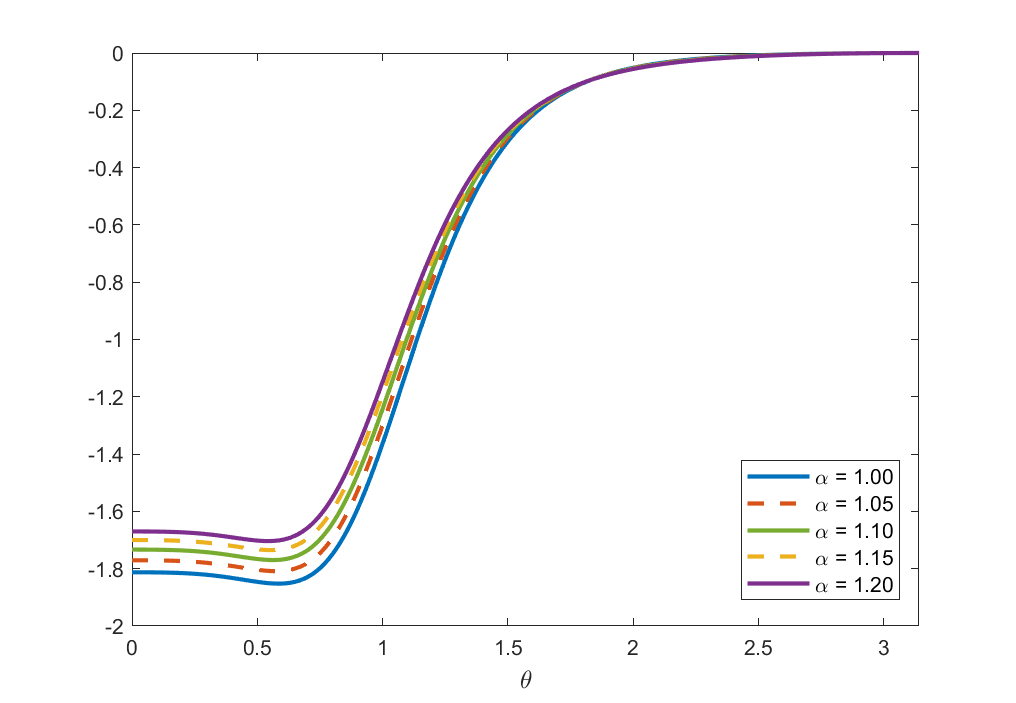}}
		\subfigure[ $\widetilde{\mathcal{I}}_{\mathrm{W4}}(\alpha,\theta)$]
		{\includegraphics[width=1.67in]{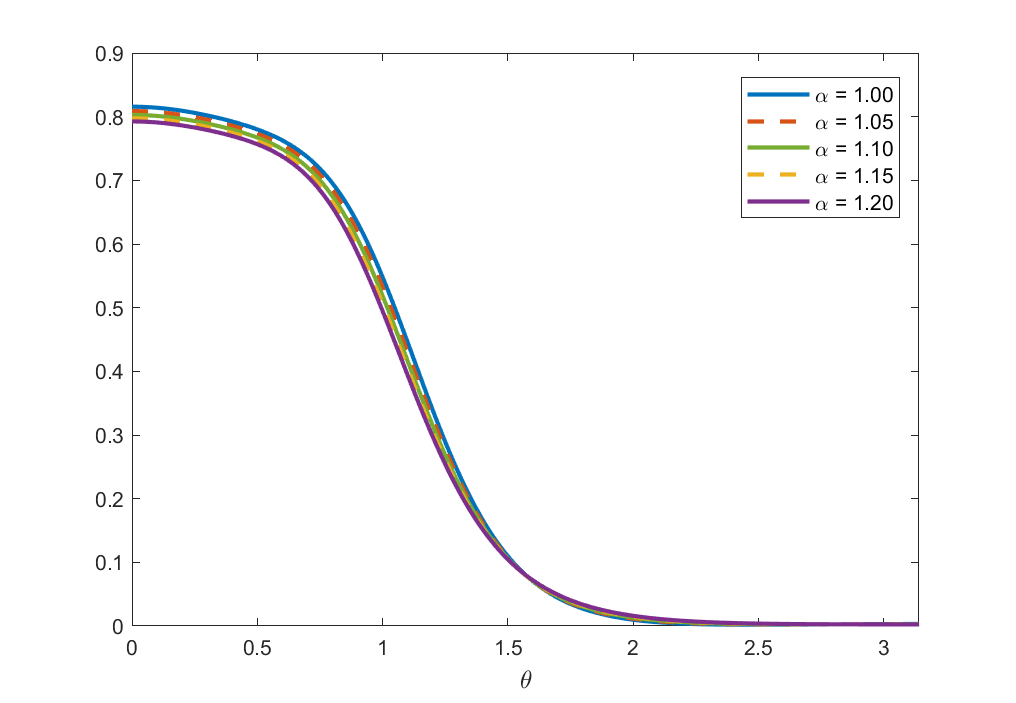}}
		\caption{Curves of $\widetilde{\mathcal{F}}_{\mathrm{W4}}(\alpha,\theta)$, $\widetilde{\mathcal{E}}_{\mathrm{W4}}(\alpha,\theta)$ and $\widetilde{\mathcal{I}}_{\mathrm{W4}}(\alpha,\theta)$ for $\theta\in[0,\pi]$.}
		\label{fig: lambda WBDF4}
	\end{figure}
	
	To get the upper bounds of $\sigma_{\mathrm{F},\mathrm{W}}^{(4,\alpha)}$ and $\sigma_{\mathrm{E},\mathrm{W}}^{(4,\alpha)}$ and the lower bound of $\lambda_{\mathrm{I},\mathrm{W}}^{(4,\alpha)}$, one can follow the proof of Proposition \ref{prop: WBDF3} to check the following inequalities for $\alpha\ge1$ and $\theta\in[0,2\pi)$,
	\begin{align*}
		\mathcal{F}_{\mathrm{W4}}:=&\,1-a_{\mathrm{W}}^{(4)}(\theta)a_{\mathrm{W}}^{(4)}(-\theta)\le0,\\
		\mathcal{E}_{\mathrm{W4}}:=&\,(20\alpha-4)^2c_{\mathrm{W}}^{(4)}(\theta)c_{\mathrm{W}}^{(4)}(-\theta)-(42 \alpha+3)^2a_{\mathrm{W}}^{(4)}(\theta)a_{\mathrm{W}}^{(4)}(-\theta)\le0,\\
		\mathcal{I}_{\mathrm{W4}}	:=&\,(400\alpha-80)\Re\kbrab{b_{\mathrm{W}}^{(4)}(\theta)a_{\mathrm{W}}^{(4)}(-\theta)}-(120\alpha-61)a_{\mathrm{W}}^{(4)}(\theta)a_{\mathrm{W}}^{(4)}(-\theta)\ge0,
	\end{align*}
	while for the parameter $\alpha\ge6/5$,
	\begin{align*}
		\mathcal{I}_{\mathrm{W4}}	:=&\,(20\alpha-4)\Re\kbrab{b_{\mathrm{W}}^{(4)}(\theta)a_{\mathrm{W}}^{(4)}(-\theta)}-(6 \alpha-3)a_{\mathrm{W}}^{(4)}(\theta)a_{\mathrm{W}}^{(4)}(-\theta)\ge0.
	\end{align*}
	The technical details (finding the extreme values of quadratic polynomials) are rather lengthy and omitted here.
	As numerical illustrations, Figure \ref{fig: lambda WBDF4} depicts the following three auxiliary functions $$\widetilde{\mathcal{F}}_{\mathrm{W4}}(\alpha,\theta):=\frac1{\abst{a_{\mathrm{W}}^{(4)}(\theta)}}- 1,\quad \widetilde{\mathcal{E}}_{\mathrm{W4}}(\alpha,\theta):=\frac{\abst{c_{\mathrm{W}}^{(4)}(\theta)}}{\abst{a_{\mathrm{W}}^{(4)}(\theta)}}-\frac{3(14\alpha+1)}{4(5\alpha-1)},$$ and $\widetilde{\mathcal{I}}_{\mathrm{W4}}(\alpha,\theta):=\Re\kbraB{\frac{b_{\mathrm{W}}^{(4)}(\theta)}{a_{\mathrm{W}}^{(4)}(\theta)}}-\frac{120\alpha-61}{80(5\alpha -1)}$ for $\theta\in[0,\pi]$ with the five fixed parameters $\alpha=1.00,1.05,1.10,1.15,1.20$.
	It completes the proof.
\end{proof}

\subsection{WBDF5 scheme} For $\rmk=5$, we take $b_{\mathrm{W},0}^{(5)}=\alpha$, and $b_{\mathrm{W},j}^{(5)}=0$ for $2\le j\le5$ in the system \eqref{scheme: multistep-linear order conditions} with $q=5$ and recover the WBDF5 scheme with
\begin{align}\label{scheme: WBDF5}
	&\vec{a}_{\mathrm{W}}^{(5)}=\brab{\tfrac{125 \alpha+12}{60}, \tfrac{77-240 \alpha}{60}, \tfrac{180 \alpha-43}{60},\tfrac{17-80\alpha}{60}, \tfrac{5 \alpha -1}{20}},\\
	&\vec{b}_{\mathrm{W}}^{(5)}=\brab{\alpha ,1-\alpha ,0,0,0,0},\quad
	\vec{c}_{\mathrm{W}}^{(5)}=\brab{4 \alpha +1,-10 \alpha ,10 \alpha ,-5 \alpha ,\alpha}.	\notag
\end{align}
The first characteristic polynomial ${\varrho}_{a,\mathrm{W}}^{(5)}(\zeta)$ is a Schur polynomial if $\alpha>\frac15$. The formula \eqref{scheme: general imex Truncation Error} gives the leading error 
\begin{align}\label{truncation: WBDF5}
	R_{\mathrm{W}}^{(5,\alpha)}=\frac{1-6 \alpha}{30}u_t^{(6)}(t_n)\tau^5+\alpha \mathcal{F}_t^{(5)}[u(t_n)]\tau^5.
\end{align}
We have the following proposition.
\begin{proposition}\label{prop: WBDF5}
	For the WBDF5 scheme \eqref{scheme: WBDF5} with the free parameter $\alpha\ge1$, it holds that
	\begin{align*}
		& 1\le \sigma_{\mathrm{F},\mathrm{W}}^{(5,\alpha)}\le \frac{47\alpha-2}{40\alpha},\quad
		\frac{15 (30 \alpha +1)}{32 (5 \alpha -1)}\le\sigma_{\mathrm{E},\mathrm{W}}^{(5,\alpha)}\le\frac{45 (10 \alpha +1)}{32 (5 \alpha -1)},\\
		&	\frac{5 (6 \alpha -5)}{32 (5 \alpha -1)}
		\le\lambda_{\mathrm{I},\mathrm{W}}^{(5,\alpha)}\le\frac{15(2\alpha-1) }{32 (5 \alpha -1) }
		\;\;\text{such that }\;\;
		\frac{6 \alpha -5}{90 \alpha +9}\le\mathfrak{I}_{\mathrm{IE},\mathrm{W}}^{(5,\alpha)}\le\frac{2 \alpha -1}{30 \alpha +1}.
	\end{align*}
\end{proposition}
\begin{proof}
	The semi-generating functions are
	$$a_{\mathrm{W}}^{(5)}(\theta)=\tfrac{125 \alpha+12}{60}+\tfrac{77-240 \alpha}{60}e^{\imath \theta}
	+ \tfrac{180 \alpha-43}{60}e^{2\imath \theta}+\tfrac{17-80\alpha}{60}e^{3\imath \theta}+\tfrac{5 \alpha -1}{20}e^{4\imath \theta},$$ $b_{\mathrm{W}}^{(5)}(\theta)=\alpha+(1-\alpha)e^{\imath \theta}$ and 
	$c_{\mathrm{W}}^{(5)}(\theta)=4 \alpha +1-10 \alpha e^{\imath \theta}+10 \alpha e^{2\imath \theta}-5 \alpha e^{3\imath \theta}+\alpha e^{4\imath \theta}.$
	The definitions in \eqref{def: lambda sigma} give $\sigma_{\mathrm{F},\mathrm{W}}^{(5,\alpha)}\ge \frac{1}{\abst{a_{\mathrm{W}}^{(5)}(0)}}=1,$ $\sigma_{\mathrm{E},\mathrm{W}}^{(5,\alpha)}\ge 
	\frac{\abst{c_{\mathrm{W}}^{(5)}(\pi)}}{\abst{a_{\mathrm{W}}^{(5)}(\pi)}}=\frac{15 (30 \alpha +1)}{32 (5 \alpha -1)}$ and	$\lambda_{\mathrm{I},\mathrm{W}}^{(5,\alpha)}\le 
	\frac{b_{\mathrm{W}}^{(5)}(\pi)}{a_{\mathrm{W}}^{(5)}(\pi)}
	=\frac{15-30 \alpha }{32-160 \alpha }$. 
	
	\begin{figure}[htb!]
		\centering
		\subfigure[ $\widetilde{\mathcal{F}}_{\mathrm{W5}}(\alpha,\theta)$]
		{\includegraphics[width=1.67in]{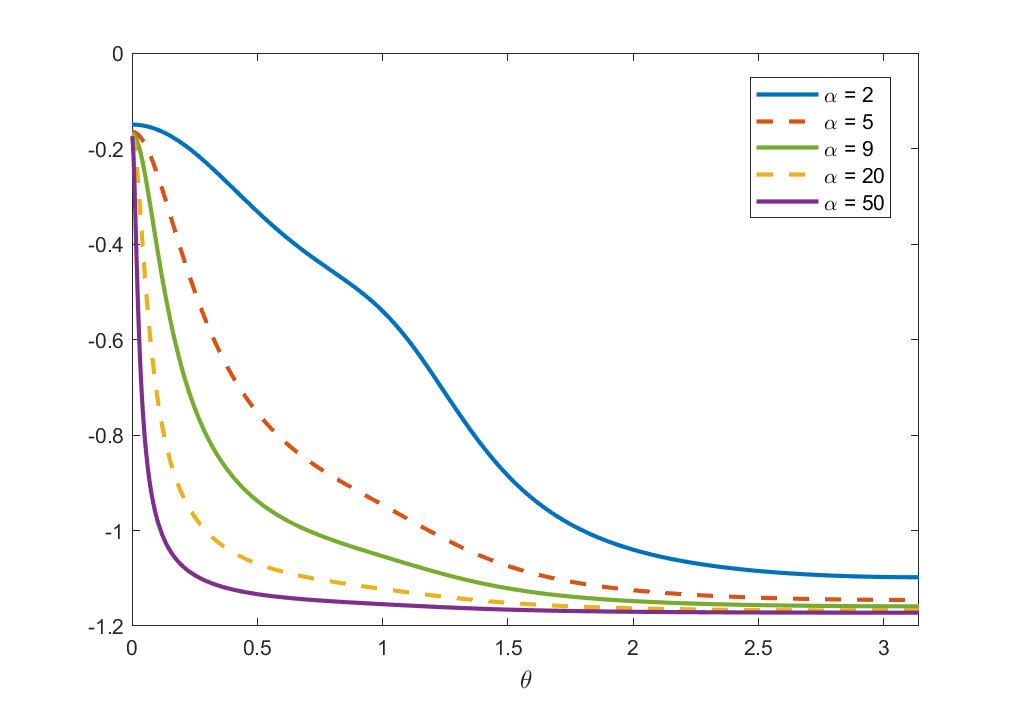}}
		\subfigure[ $\widetilde{\mathcal{E}}_{\mathrm{W5}}(\alpha,\theta)$]
		{\includegraphics[width=1.67in]{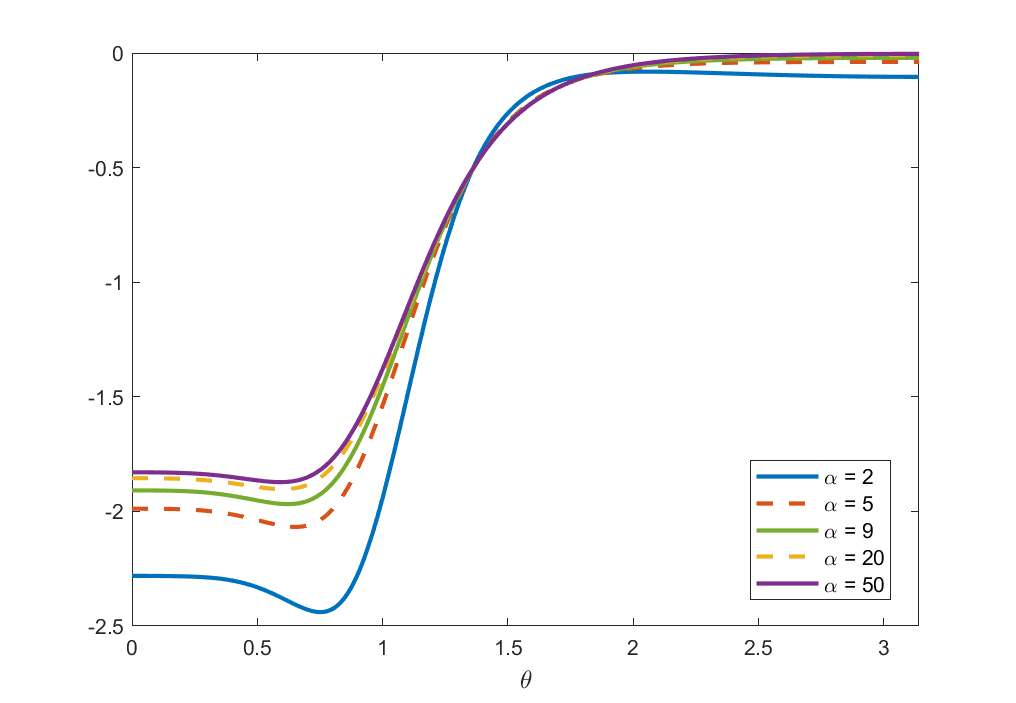}}
		\subfigure[ $\widetilde{\mathcal{I}}_{\mathrm{W5}}(\alpha,\theta)$]
		{\includegraphics[width=1.67in]{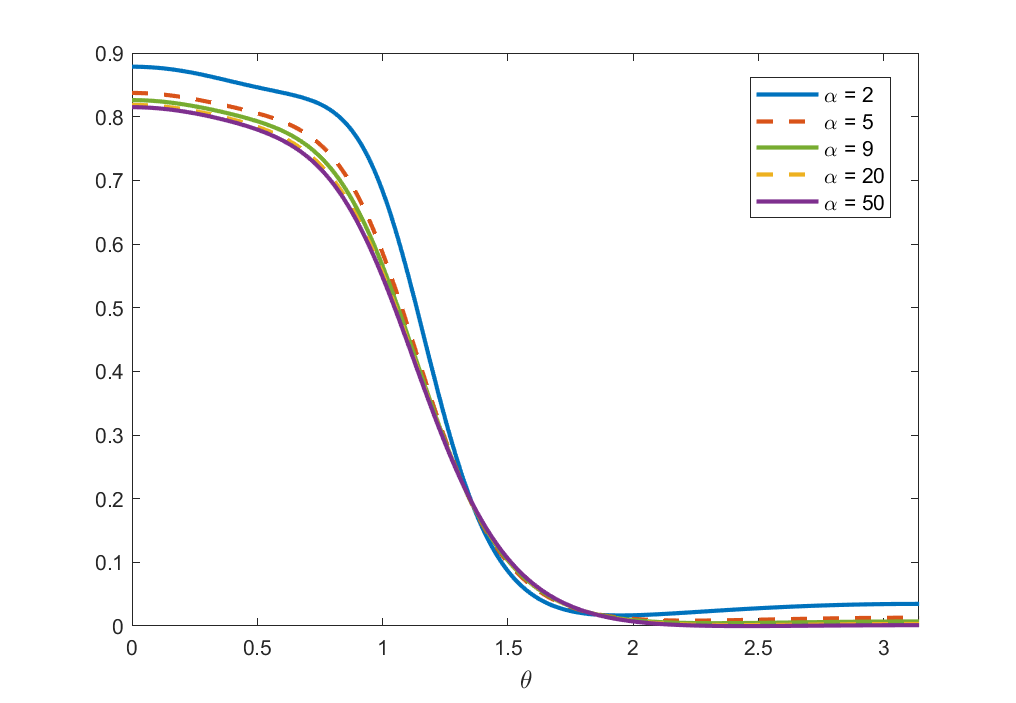}}
		\caption{Curves of $\widetilde{\mathcal{F}}_{\mathrm{W5}}(\alpha,\theta)$, $\widetilde{\mathcal{E}}_{\mathrm{W5}}(\alpha,\theta)$ and $\widetilde{\mathcal{I}}_{\mathrm{W5}}(\alpha,\theta)$ for $\theta\in[0,\pi]$.}
		\label{fig: lambda WBDF5}
	\end{figure}
	
	To get the upper bounds of $\sigma_{\mathrm{F},\mathrm{W}}^{(5,\alpha)}$ and $\sigma_{\mathrm{E},\mathrm{W}}^{(5,\alpha)}$ and the lower bound of $\lambda_{\mathrm{I},\mathrm{W}}^{(5,\alpha)}$, one can follow the proof of Proposition \ref{prop: WBDF3} to check the following inequalities for $\alpha\ge1$ and $\theta\in[0,2\pi)$,
	\begin{align*}
		\mathcal{F}_{\mathrm{W5}}:=&\,(40\alpha)^2-(47\alpha-2)^2a_{\mathrm{W}}^{(5)}(\theta)a_{\mathrm{W}}^{(5)}(-\theta)\le0,\\
		\mathcal{E}_{\mathrm{W5}}:=&\,(160 \alpha -32)^2c_{\mathrm{W}}^{(5)}(\theta)c_{\mathrm{W}}^{(5)}(-\theta)
		-(450 \alpha +45)^2a_{\mathrm{W}}^{(5)}(\theta)a_{\mathrm{W}}^{(5)}(-\theta)\le0,\\
		\mathcal{I}_{\mathrm{W5}}	:=&\,
		(160 \alpha-32)\Re\kbrab{b_{\mathrm{W}}^{(5)}(\theta)a_{\mathrm{W}}^{(5)}(-\theta)}
		-(30\alpha-25)a_{\mathrm{W}}^{(5)}(\theta)a_{\mathrm{W}}^{(5)}(-\theta)\ge0,
	\end{align*}
	while the technical details are rather lengthy and omitted here. As numerical illustrations, Figure \ref{fig: lambda WBDF5} depicts the following three auxiliary functions $$\widetilde{\mathcal{F}}_{\mathrm{W5}}(\alpha,\theta):=\frac1{\abst{a_{\mathrm{W}}^{(5)}(\theta)}}- \frac{47\alpha-2}{40\alpha},\quad \widetilde{\mathcal{E}}_{\mathrm{W5}}(\alpha,\theta):=\frac{\abst{c_{\mathrm{W}}^{(5)}(\theta)}}{\abst{a_{\mathrm{W}}^{(5)}(\theta)}}-\frac{45 (10 \alpha +1)}{32 (5 \alpha -1)},$$ and $\widetilde{\mathcal{I}}_{\mathrm{W5}}(\alpha,\theta):=\Re\kbraB{\frac{b_{\mathrm{W}}^{(5)}(\theta)}{a_{\mathrm{W}}^{(5)}(\theta)}}-	\frac{5 (6 \alpha -5)}{32 (5 \alpha -1)}$ for $\theta\in[0,\pi]$ with the five fixed parameters $\alpha=2,5,9,20,50$.
	It completes the proof.
\end{proof}


\section{$\beta$-parameterized GBDF methods}\label{sec: SM-GBDF}

Following the idea of \cite{HuangShen:2024-SM},  the $\beta$-parameterized GBDF methods are constructed by approximating each term of the differential equations at the off-set grid point $t_*:=t_{n-1+\beta}$  $(\beta\ge1)$, 
\begin{align*}
	\sum_{j=0}^{\rmk-1}a_{\mathrm{G},j}^{(\rmk)}\partial_{\tau}u^{n-j}\approx u'(t_*),\quad
	\sum_{j=0}^{\rmk-1}b_{\mathrm{G},j}^{(\rmk)}u^{n-j}\approx u(t_*),\quad
	\sum_{j=0}^{\rmk-1}c_{\mathrm{G},j}^{(\rmk)}u^{n-j-1}\approx u(t_*).
\end{align*}
That is, the discrete coefficients $a_{\mathrm{G},j}^{(\rmk)}$, $b_{\mathrm{G},j}^{(\rmk)}$ and $c_{\mathrm{G},j}^{(\rmk)}$  for $0\le j\le \rmk-1$ can be determined independently
by three linear algebraic systems of Vandermonde-type. 
Note that, the GBDF6 method constructed by following \cite{HuangShen:2024-SM} does not satisfy $\lambda_{\mathrm{I},\mathrm{G}}^{(6)}>0$ for $\beta>2$ since the first characteristic polynomial $\varrho_{a,\mathrm{G}}^{(6)}$ is not a Schur polynomial if $\beta>2$.  Here we consider the stability property of GBDF-$\rmk$ ($2\le \rmk\le 5$) methods for $\beta\ge1$.

\subsection{GBDF2 scheme} The GBDF2 scheme in \cite{HuangShen:2024-SM} is the same to the WBDF2 scheme \eqref{scheme: WBDF2} with the discrete coefficients
\begin{align}\label{scheme: GBDF2}
	&\vec{a}_{\mathrm{G}}^{(2)}=(\tfrac{1}{2}+\beta,\tfrac{1}{2}-\beta),\quad
	\vec{b}_{\mathrm{G}}^{(2)}=(\beta,1-\beta,0),\quad
	\vec{c}_{\mathrm{G}}^{(2)}=(\beta +1,-\beta).
\end{align}
The formula \eqref{scheme: general imex Truncation Error} gives the leading error 
\begin{align*}
	R_{\mathrm{G}}^{(2,\beta)}=\frac{1-3 \beta}{6}u_t^{(3)}(t_n)\tau^2+\beta \mathcal{F}_t^{(2)}[u(t_n)]\tau^2.
\end{align*}
Proposition \ref{prop: WBDF2} gives $\sigma_{\mathrm{F},\mathrm{G}}^{(2\beta)}=1,$ 
$\sigma_{\mathrm{E},\mathrm{G}}^{(2,\beta)}=\frac{2\beta+1}{2\beta}$ and $\lambda_{\mathrm{I},\mathrm{G}}^{(2,\beta)}=\frac{2\beta-1}{2\beta}$. 
The associated implicit-explicit controllability intensity $\mathfrak{I}_{\mathrm{IE},\mathrm{G}}^{(2,\beta)}
=\frac{2\beta-1}{2\beta+1}$ for $\beta\ge1$.

\subsection{GBDF3 scheme}   
The discrete coefficients of the GBDF3 scheme read
\begin{align}\label{scheme: GBDF3}
	&\vec{a}_{\mathrm{G}}^{(3)}
	=\brab{\tfrac{3 \beta ^2+6 \beta +2}{6},\tfrac{-6 \beta ^2-6 \beta +5}{6},\tfrac{3 \beta ^2-1}{6}},\quad
	\vec{b}_{\mathrm{G}}^{(3)}=\brab{\tfrac{\beta ^2+\beta}{2} ,1-\beta ^2,\tfrac{\beta ^2-\beta}{2},0},\\
	&\vec{c}_{\mathrm{G}}^{(3)}=(\tfrac{\beta ^2+3 \beta +2}{2},-2\beta-\beta ^2,\tfrac{\beta ^2+\beta}{2}).\notag
\end{align}
The first characteristic polynomial ${\varrho}_{a,\mathrm{G}}^{(3)}(\zeta)$ is a Schur polynomial if $\beta >\frac{\sqrt{21}-3}{6}$. 
The formula \eqref{scheme: general imex Truncation Error} gives the leading error 
\begin{align}\label{truncation: GBDF3}
	R_{\mathrm{G}}^{(3,\beta)}=\frac{1-\beta-3 \beta ^2}{12}u_t^{(4)}(t_n)\tau^3+\frac{\beta^2+\beta}{2} \mathcal{F}_t^{(3)}[u(t_n)]\tau^3.
\end{align}
We have the following proposition.
\begin{proposition}\label{prop: GBDF3}
	For the GBDF3 scheme \eqref{scheme: GBDF3} with the free parameter $\beta\ge1$, it holds that
	\begin{align*}
		&\sigma_{\mathrm{F},\mathrm{G}}^{(3,\beta)}=1,\quad	\sigma_{\mathrm{E},\mathrm{G}}^{(3,\beta)}=\frac{6 \beta ^2+12 \beta +3}{6 \beta ^2+6 \beta -2},\quad
		\frac{6 \beta ^2-4}{6 \beta ^2+6 \beta -2}\le\lambda_{\mathrm{I},\mathrm{G}}^{(3,\beta)}
		\le\frac{6 \beta ^2-3}{6 \beta ^2+6 \beta -2}\\
		&\text{such that}\quad
		\frac{6\beta ^2-4}{6 \beta ^2+12 \beta +3}\le\mathfrak{I}_{\mathrm{IE},\mathrm{G}}^{(3,\beta)}\le\frac{6\beta ^2-3}{6 \beta ^2+12 \beta +3}.
	\end{align*}
\end{proposition}
\begin{proof}
	The associated semi-generating functions are
	$$a_{\mathrm{G}}^{(3)}(\theta)=\tfrac{3 \beta ^2+6 \beta +2}{6}-\tfrac{6 \beta ^2+6 \beta -5}{6}e^{\imath \theta}+\tfrac{3 \beta ^2-1}{6}e^{2\imath \theta},$$
	$b_{\mathrm{G}}^{(3)}(\theta)=\tfrac{\beta ^2+\beta}{2}+(1-\beta ^2)e^{\imath \theta}+\tfrac{\beta ^2-\beta}{2}e^{2\imath \theta}$ and $c_{\mathrm{G}}^{(3)}(\theta)=\tfrac{\beta ^2+3 \beta +2}{2}-(2\beta+\beta ^2) e^{\imath \theta}+\tfrac{\beta ^2+\beta}{2}e^{2\imath \theta}$.
	According to the definitions in \eqref{def: lambda sigma}, it is not difficult to get $\sigma_{\mathrm{F},\mathrm{G}}^{(3,\beta)}\ge 
	\frac{1}{\abst{a_{\mathrm{G}}^{(3)}(0)}}=1$, $\sigma_{\mathrm{E},\mathrm{G}}^{(3,\beta)}\ge 
	\frac{\abst{c_{\mathrm{G}}^{(3)}(\pi)}}{\abst{a_{\mathrm{G}}^{(3)}(\pi)}}=\frac{6 \beta ^2+12 \beta +3}{6 \beta ^2+6 \beta -2}$ and $\lambda_{\mathrm{I},\mathrm{G}}^{(3,\beta)}\le 
	\frac{b_{\mathrm{G}}^{(3)}(\pi)}{a_{\mathrm{G}}^{(3)}(\pi)}
	=\frac{6 \beta ^2-3}{6 \beta ^2+6 \beta -2}$.
	
	Furthermore, direct but lengthy calculations give
	\begin{align*}
		a_{\mathrm{G}}^{(3)}(\theta)a_{\mathrm{G}}^{(3)}(-\theta)=&\,  \tfrac{9 \beta ^4+18 \beta ^3+3 \beta ^2-6 \beta +5}{6}-\tfrac{36 \beta ^4+72 \beta ^3+12 \beta ^2-24 \beta-5}{18}\cos\theta\\
		&\,+\tfrac{9 \beta ^4+18 \beta ^3+3 \beta ^2-6 \beta -2}{18}\cos2 \theta,
		\\
		\Re\kbrab{b_{\mathrm{G}}^{(3)}(\theta)a_{\mathrm{G}}^{(3)}(-\theta)}
		= &\,\tfrac{18 \beta ^4+18 \beta ^3-15 \beta ^2-9 \beta +10}{12}-\tfrac{24 \beta ^4+24 \beta ^3-20 \beta ^2-12 \beta -2}{12} \cos\theta\\
		&\,+\tfrac{6 \beta ^4+6 \beta ^3-5 \beta^2 -3\beta}{12} \cos2 \theta,\\
		c_{\mathrm{G}}^{(3)}(\theta)c_{\mathrm{G}}^{(3)}(-\theta)
		=&\, \tfrac{3 \beta ^4+12 \beta ^3+15 \beta ^2+6 \beta +2}{2}-2\beta  (\beta +1)^2 (\beta +2)\cos\theta\\
		&\,+\tfrac{\beta  (\beta +1)^2 (\beta +2)}{2}\cos2\theta.
	\end{align*}
	It is not difficult to check that $9 \beta ^4+18 \beta ^3-6 \beta ^2-6 \beta +1>0$ for $\beta\ge1$ and
	\begin{align*}
		{\mathcal{F}}_{\mathrm{G3}}:=&\,1-a_{\mathrm{G}}^{(3)}(\theta)a_{\mathrm{G}}^{(3)}(-\theta)
		=\kbraB{\frac{2}{9} (9 \beta ^4+18 \beta ^3+3 \beta ^2-6 \beta -2)(\cos\theta-1)
			-\frac{1}{3}}\sin^2\frac{\theta }{2}\\
		\le	&\,-\frac{1}{3}\sin^2\frac{\theta }{2}\le0.
	\end{align*}
	Thus one has $\frac{1}{\abst{a_{\mathrm{G}}^{(3)}(\theta)}}\le 1$ such that $\sigma_{\mathrm{F},\mathrm{G}}^{(3,\beta)}=1$. Similarly,  since
	\begin{align*}
		{\mathcal{E}}_{\mathrm{G3}}:=&\,(6 \beta ^2+6 \beta -2)^2c_{\mathrm{G}}^{(3)}(\theta)c_{\mathrm{G}}^{(3)}(-\theta)-(6 \beta ^2+12 \beta +3)^2a_{\mathrm{G}}^{(3)}(\theta)a_{\mathrm{G}}^{(3)}(-\theta)\\
		=&\,\cos^2\frac{\theta }{2}\Big[2(3 \beta ^4+30 \beta ^3+57 \beta ^2+30 \beta +2)\cos\theta\Big.\\
		&\,\qquad\Big.-3(2 \beta ^4+44 \beta ^3+94 \beta ^2+52 \beta +3)\Big]\\
		\le&\, -\bra{72 \beta ^3+168 \beta ^2+96 \beta +5}\cos^2\frac{\theta }{2}\le0,
	\end{align*}
	one gets $\frac{\abst{c_{\mathrm{G}}^{(3)}(\theta)}}{\abst{a_{\mathrm{G}}^{(3)}(\theta)}}\le 
	\frac{6 \beta ^2+12 \beta +3}{6 \beta ^2+6 \beta -2}$ for $\theta\in[0,2\pi)$ and thus  $\sigma_{\mathrm{E},\mathrm{G}}^{(3,\beta)}=\frac{6 \beta ^2+12 \beta +3}{6 \beta ^2+6 \beta -2}$.

	To get the lower bound of $\lambda_{\mathrm{I},\mathrm{G}}^{(3,\beta)}$,
	we consider the auxiliary function for $\theta\in[0,\pi]$,
	\begin{align*}
		\mathcal{I}_{\mathrm{G3}}(\beta,\theta):=&\,(6 \beta ^2+6 \beta -2)\Re\kbrab{b^{(3)}(\theta)a^{(3)}(-\theta)}
		-(6 \beta ^2-4)a^{(3)}(\theta)a^{(3)}(-\theta)\\
		=&\, \frac{27 \beta ^4+54 \beta ^3+45 \beta +30}{18}
		-\frac{18 \beta ^4+36 \beta ^3+6 \beta ^2-39 \beta -7}{9}\cos\theta\\
		&\,+\frac{9 \beta ^4+18 \beta ^3+12 \beta ^2-15 \beta -8}{18}\cos2 \theta
	\end{align*}
	with $\mathcal{I}_{\mathrm{G3}}(\beta,0)=2+6\beta$ and $\mathcal{I}_{\mathrm{G3}}(\beta,\pi)=\tfrac{4}{9} (3 \beta ^2+3 \beta -1)^2> \mathcal{I}_{\mathrm{G3}}(\beta,0)$ for $\beta\ge1$.
	Note that,
	\begin{align*}
		\frac{\partial \mathcal{I}_{\mathrm{G3}}(\beta,\theta)}{\partial{\theta}}
		=&\,\frac{\sin\theta}{9}\Big[
		\brat{18 \beta ^4+36 \beta ^3+6 \beta ^2-39 \beta -7}\Big.\\
		&\,\qquad\Big.	-(18 \beta ^4+36 \beta ^3+24 \beta ^2-30 \beta -16)\cos\theta\Big].
	\end{align*}
	Solving $\frac{\partial \mathcal{I}_{\mathrm{G3}}(\beta,\theta)}{\partial{\theta}}=0$ for $\theta\in(0,\pi)$ gives the single stationary point
	\begin{align*}
		\cos\theta_*=1-\frac{9 \left(2 \beta ^2+\beta -1\right)}{2 \left(9 \beta ^4+18 \beta ^3+12 \beta ^2-15 \beta -8\right)}\in(0,1).
	\end{align*}
	Thus, the minimum value of $\mathcal{I}_{\mathrm{G3}}(\beta,\theta)$ for $\theta\in[0,\pi]$ is
	\begin{align*}
		\min_{\theta\in[0,\pi]}\mathcal{I}_{\mathrm{G3}}(\beta,\theta)=&\,\min\big\{\mathcal{I}_{\mathrm{G3}}(\beta,0),\mathcal{I}_{\mathrm{G3}}(\beta,\theta_*)
		\big\}=\mathcal{I}_{\mathrm{G3}}(\beta,\theta_*)\\
		=&\,6 \beta +1+\frac{36 \beta ^3+75 \beta ^2-42 \beta -41}{4 \left(9 \beta ^4+18 \beta ^3+12 \beta ^2-15 \beta -8\right)}>0\quad\text{for $\beta\ge1$.}
	\end{align*}
	The symmetry of $\mathcal{I}_{\mathrm{G3}}(\beta,\theta)$ with respect to $\theta$ implies that
	\begin{align*}
		&	\Re\kbraB{\frac{b_{\mathrm{G}}^{(3)}(\theta)}{a_{\mathrm{G}}^{(3)}(\theta)}}\ge \frac{6 \beta ^2-4}{6 \beta ^2+6 \beta -2}
		\;\;\text{for $\theta\in[0,2\pi)$ and}\;\; \lambda_{\mathrm{I},\mathrm{G}}^{(3,\beta)}\ge\frac{6 \beta ^2-4}{6 \beta ^2+6 \beta -2}.
	\end{align*}
	It completes the proof.
\end{proof}

\subsection{GBDF4 scheme}
The discrete coefficients of the GBDF4 scheme read
\begin{align}\label{scheme: GBDF4}
	&\vec{a}_{\mathrm{G}}^{(4)}
	=\brab{\tfrac{2 \beta ^3+9 \beta ^2+11 \beta +3}{12},\tfrac{-6 \beta ^3-21\beta ^2-9 \beta +13}{12}, \tfrac{6 \beta^3+15 \beta^2-3\beta-5}{12}, \tfrac{-2 \beta^3-3 \beta^2+\beta+1}{12}},\\
	&\vec{b}_{\mathrm{G}}^{(4)}=\brab{\tfrac{\beta^3+3 \beta^2+2 \beta}{6},\tfrac{-\beta^3-2 \beta^2+\beta +2}{2},\tfrac{\beta^3+\beta^2 -2\beta}{2} ,\tfrac{\beta -\beta^3}{6},0},\notag\\
	&\vec{c}_{\mathrm{G}}^{(4)}=\brab{\tfrac{\beta^3+6 \beta^2+11 \beta +6}{6}, \tfrac{-\beta^3-5 \beta^2 -6\beta}{2}, \tfrac{\beta^3+4\beta^2 +3\beta}{2}, \tfrac{-\beta^3-3 \beta^2-2\beta}{6} }.\notag
\end{align}
The first characteristic polynomial ${\varrho}_{a,\mathrm{G}}^{(4)}(\zeta)$ is a Schur polynomial if $\beta >\sqrt{2}-1$. 
The formula \eqref{scheme: general imex Truncation Error} gives the leading error 
\begin{align}\label{truncation: GBDF4}
	R_{\mathrm{G}}^{(4,\beta)}=\frac{3-5 \beta ^3-10 \beta ^2}{60}	u_t^{(5)}(t_n)\tau^4
	+\frac{\beta ^3+3 \beta^2 +2\beta}{6}\mathcal{F}_t^{(4)}[u(t_n)]\tau^4.
\end{align}
One has the following proposition.
\begin{proposition}\label{prop: GBDF4}
	For the GBDF4 scheme \eqref{scheme: GBDF4} with the parameter $\beta\ge1$, it holds that
	\begin{align*}
		&\sigma_{\mathrm{F},\mathrm{G}}^{(4,\beta)}\ge 1,\quad
		\sigma_{\mathrm{E},\mathrm{G}}^{(4,\beta)}\ge\frac{4 \beta ^3+18\beta ^2+20 \beta +3}{4 (\beta ^3+3 \beta ^2+\beta -1)},\quad 
		\lambda_{\mathrm{I},\mathrm{G}}^{(4,\beta)}\le\frac{4 \beta ^3+6 \beta ^2-4 \beta -3}{4(\beta ^3+3 \beta ^2+\beta -1)}\\
		&\text{such that}\quad
		\mathfrak{I}_{\mathrm{IE},\mathrm{G}}^{(4,\beta)}\le
		\frac{4 \beta ^3+6 \beta ^2-4 \beta -3}{4 \beta ^3+18\beta ^2+20 \beta +3}.
	\end{align*}
	If the parameter $\beta=9$, then 
	\begin{align*}\sigma_{\mathrm{F},\mathrm{G}}^{(4,9)}\le \frac{49}{45},\;\;
		\sigma_{\mathrm{E},\mathrm{G}}^{(4,9)}\le\frac{2319}{1960},\;\;\lambda_{\mathrm{I},\mathrm{G}}^{(4,9)}\ge \frac{1641}{1960}\;\;
		\text{such that}\;\;
		\mathfrak{I}_{\mathrm{IE},\mathrm{G}}^{(4,9)}\ge \frac{547}{773}\approx 0.707633.
	\end{align*}
\end{proposition} 

\begin{proof}
	The associated semi-generating functions  read
	\begin{align*}
		a_{\mathrm{G}}^{(4)}(\theta)=&\,\tfrac{2 \beta ^3+9 \beta ^2+11 \beta +3}{12}+\tfrac{-6 \beta ^3-21\beta ^2-9 \beta +13}{12}e^{\imath \theta}\\&\,
		+\tfrac{6 \beta^3+15 \beta^2-3\beta-5}{12}e^{2\imath \theta}+\tfrac{-2 \beta^3-3 \beta^2+\beta+1}{12}e^{3\imath \theta},\\
		b_{\mathrm{G}}^{(4)}(\theta)=&\,\tfrac{\beta^3+3 \beta^2+2 \beta}{6}+\tfrac{-\beta^3-2 \beta^2+\beta +2}{2}e^{\imath \theta}
		+\tfrac{\beta^3+\beta^2 -2\beta}{2}e^{2\imath \theta}+\tfrac{\beta -\beta^3}{6}e^{3\imath \theta},\\
		c_{\mathrm{G}}^{(4)}(\theta)=&\,\tfrac{\beta^3+6 \beta^2+11 \beta +6}{6}+\tfrac{-\beta^3-5 \beta^2 -6\beta}{2}e^{\imath \theta}
		+\tfrac{\beta^3+4\beta^2 +3\beta}{2}e^{2\imath \theta}+\tfrac{-\beta^3-3 \beta^2-2\beta}{6}e^{3\imath \theta}.
	\end{align*}
	According to the definitions in \eqref{def: lambda sigma}, it is not difficult to get $\sigma_{\mathrm{F},\mathrm{G}}^{(4,\beta)}\ge\frac{1}{\abst{a_{\mathrm{G}}^{(4)}(0)}}=1$,
	$\sigma_{\mathrm{E},\mathrm{G}}^{(4,\beta)}\ge\frac{\abst{c_{\mathrm{G}}^{(4)}(\pi)}}{\abst{a_{\mathrm{G}}^{(4)}(\pi)}}
	=\frac{4 \beta ^3+18\beta ^2+20 \beta +3}{4 (\beta ^3+3 \beta ^2+\beta -1)}$ and
	$\lambda_{\mathrm{I},\mathrm{G}}^{(4,\beta)}\le\frac{b_{\mathrm{G}}^{(4)}(\pi)}{a_{\mathrm{G}}^{(4)}(\pi)}
	=\frac{4 \beta ^3+6 \beta ^2-4 \beta -3}{4(\beta ^3+3 \beta ^2+\beta -1)}$. 
	
	To find the upper bounds of $\sigma_{\mathrm{F},\mathrm{G}}^{(4,\beta)}$, $\sigma_{\mathrm{E},\mathrm{G}}^{(4,\beta)}$ and the lower bound of $\lambda_{\mathrm{I},\mathrm{G}}^{(4,\beta)}$, it is to consider  
	the following auxiliary functions
	\begin{align*}
		{\mathcal{F}}_{\mathrm{G4}}(\beta,\theta):=&\,100\beta^2-(11\beta-1)^2a_{\mathrm{G}}^{(4)}(\theta)a_{\mathrm{G}}^{(4)}(-\theta),\\
		{\mathcal{E}}_{\mathrm{G4}}(\beta,\theta):=&\,16(\beta ^3+3 \beta ^2+\beta -1)^2c_{\mathrm{G}}^{(4)}(\theta)c_{\mathrm{G}}^{(4)}(-\theta)\\
		&\,-(4 \beta ^3+19 \beta ^2+20 \beta +3)^2a_{\mathrm{G}}^{(4)}(\theta)a_{\mathrm{G}}^{(4)}(-\theta),\\
		{\mathcal{I}}_{\mathrm{G4}}(\beta,\theta):=&\,4(\beta ^3+3 \beta ^2+\beta -1)\Re\kbrab{b_{\mathrm{G}}^{(4)}(\theta)a_{\mathrm{G}}^{(4)}(-\theta)}\\
		&\,
		-(4 \beta ^3+5 \beta ^2-4 \beta -3)a_{\mathrm{G}}^{(4)}(\theta)a_{\mathrm{G}}^{(4)}(-\theta).
	\end{align*}
	For any given value of $\beta$ such as $\beta=9$, the functions ${\mathcal{F}}_{\mathrm{G4}}(9,\theta)$, ${\mathcal{E}}_{\mathrm{G4}}(9,\theta)$ and ${\mathcal{I}}_{\mathrm{G4}}(9,\theta)$ are cubic polynomials with respect to $y=\cos\theta\in[-1,1]$. 
	Actually, it is easy to verify  that
	${\mathcal{F}}_{\mathrm{G4}}(9,\theta)\le0$, ${\mathcal{E}}_{\mathrm{G4}}(9,\theta)\le0$ and ${\mathcal{I}}_{\mathrm{G4}}(9,\theta)\ge0$
	for $\theta\in[0,2\pi)$ and they lead to the claimed bounds, while the technical details are omitted. For the general case of $\beta\ge1$, it would be rather complex and lengthy to verify that ${\mathcal{F}}_{\mathrm{G4}}(\beta,\theta)\le0$, ${\mathcal{E}}_{\mathrm{G4}}(\beta,\theta)\le0$ and ${\mathcal{I}}_{\mathrm{G4}}(\beta,\theta)\ge0$
	for $\theta\in[0,2\pi)$ and we will check the results numerically in Remark \ref{remark: lambda GBDF4} for some fixed $\beta$. 
	It completes the proof.
\end{proof}

\begin{figure}[htb!]
	\centering
	\subfigure[ $\widetilde{\mathcal{F}}_{\mathrm{G4}}(\beta,\theta)$]
	{\includegraphics[width=1.67in]{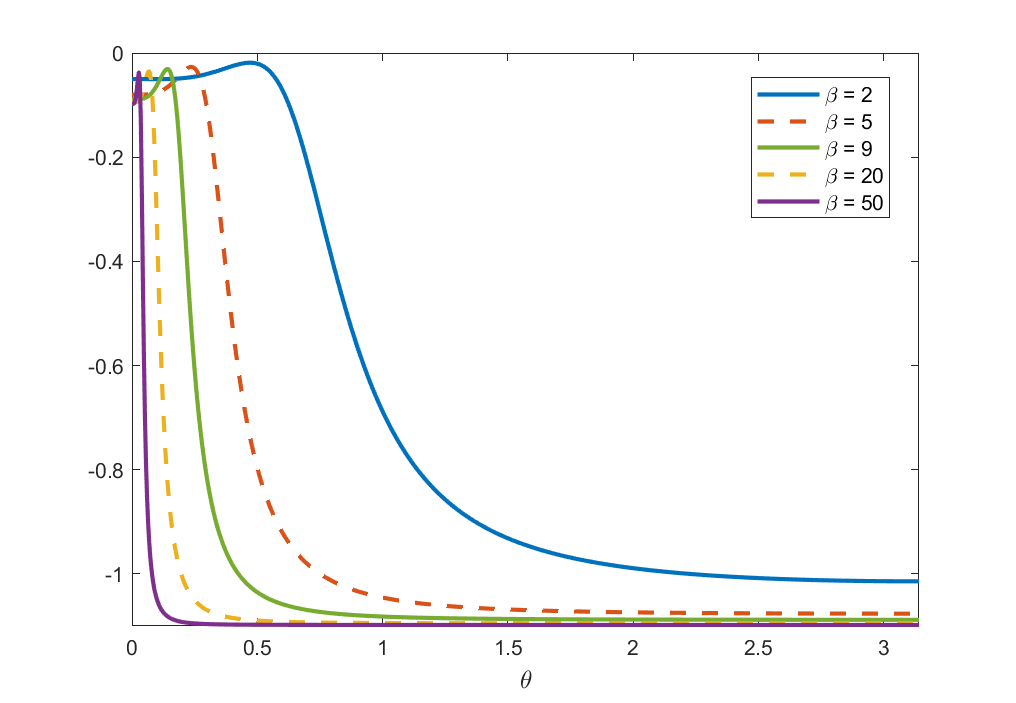}}
	\subfigure[ $\widetilde{\mathcal{E}}_{\mathrm{G4}}(\beta,\theta)$]
	{\includegraphics[width=1.67in]{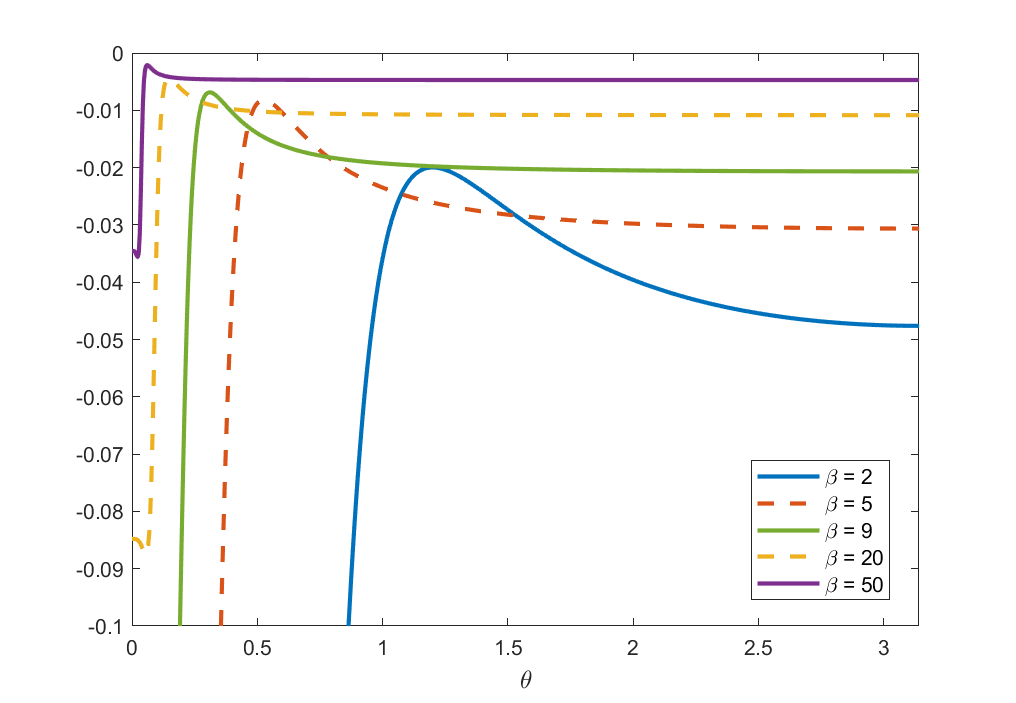}}
	\subfigure[ $\widetilde{\mathcal{I}}_{\mathrm{G4}}(\beta,\theta)$]
	{\includegraphics[width=1.67in]{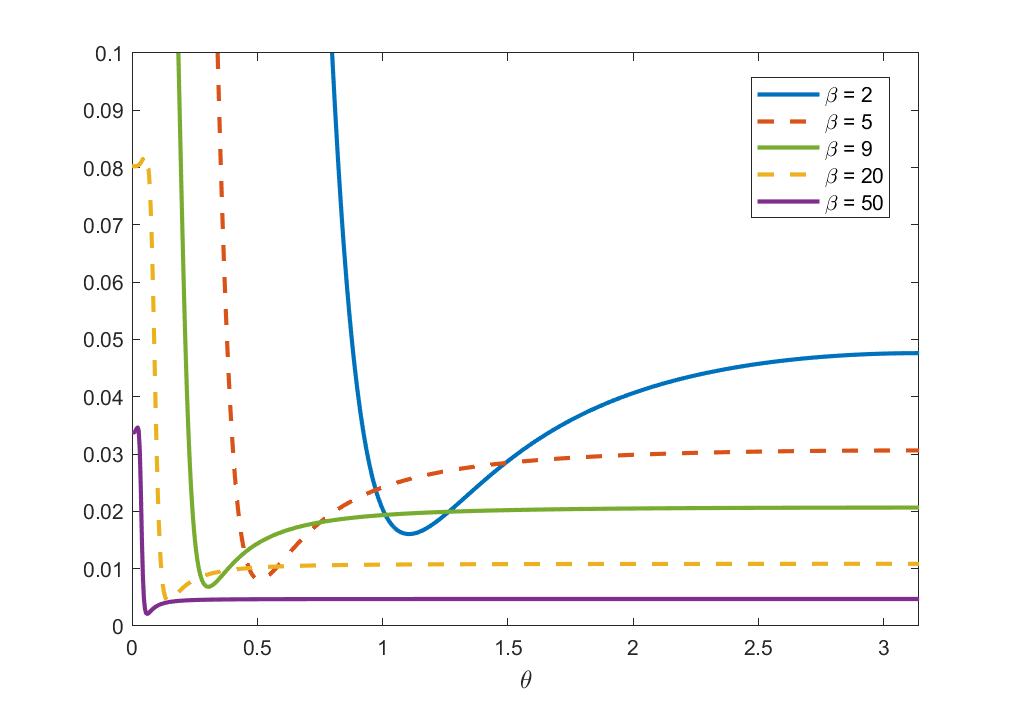}}
	\caption{Curves of $\widetilde{\mathcal{F}}_{\mathrm{G4}}(\beta,\theta)$, $\widetilde{\mathcal{E}}_{\mathrm{G4}}(\beta,\theta)$ and $\widetilde{\mathcal{I}}_{\mathrm{G4}}(\beta,\theta)$ for $\theta\in[0,\pi]$.}
	\label{fig: lambda GBDF4}
\end{figure}

\begin{remark}\label{remark: lambda GBDF4}
	For the GBDF4 scheme \eqref{scheme: GBDF4} with the parameter $\beta\ge1$, numerical tests suggest that
	\begin{align*}
		&	\widetilde{\mathcal{F}}_{\mathrm{G4}}(\beta,\theta):=\frac1{\abst{a_{\mathrm{G}}^{(4)}(\theta)}}- \frac{11\beta-1}{10\beta}\le0,\\
		&\widetilde{\mathcal{E}}_{\mathrm{G4}}(\beta,\theta):=\frac{\abst{c_{\mathrm{G}}^{(4)}(\theta)}}{\abst{a_{\mathrm{G}}^{(4)}(\theta)}}-\frac{4 \beta ^3+19\beta ^2+20 \beta +3}{4 (\beta ^3+3 \beta ^2+\beta -1)}\le0,\\
		&
		\widetilde{\mathcal{I}}_{\mathrm{G4}}(\beta,\theta):=\Re\kbra{\frac{b_{\mathrm{G}}^{(4)}(\theta)}{a_{\mathrm{G}}^{(4)}(\theta)}}-\frac{4 \beta ^3+5 \beta ^2-4 \beta -3}{4(\beta ^3+3 \beta ^2+\beta -1)}\ge0,
	\end{align*}	
	see Figure \ref{fig: lambda GBDF4}, 
	in which the auxiliary functions
	$\widetilde{\mathcal{F}}_{\mathrm{G4}}(\beta,\theta)$, $\widetilde{\mathcal{E}}_{\mathrm{G4}}(\beta,\theta)$ and $\widetilde{\mathcal{I}}_{\mathrm{G4}}(\beta,\theta)$ are depicted for $\theta\in[0,\pi]$ with five fixed parameters $\beta=2,5,9,20$ and 50.	
	One has
	\begin{align*}
		&\sigma_{\mathrm{F},\mathrm{G}}^{(4,\beta)}\le \frac{11\beta-1}{10\beta},\quad\sigma_{\mathrm{E},\mathrm{G}}^{(4,\beta)}\le\frac{4 \beta ^3+19\beta ^2+20 \beta +3}{4 (\beta ^3+3 \beta ^2+\beta -1)},\quad
		\lambda_{\mathrm{I},\mathrm{G}}^{(4,\beta)}\ge\frac{4 \beta ^3+5 \beta ^2-4 \beta -3}{4(\beta ^3+3 \beta ^2+\beta -1)},\\
		&\text{such that}\quad
		\mathfrak{I}_{\mathrm{IE},\mathrm{G}}^{(4,\beta)}\ge\frac{4 \beta ^3+5 \beta ^2-4 \beta -3}{4 \beta ^3+19\beta ^2+20 \beta +3}.
	\end{align*}	
	
\end{remark}

\subsection{GBDF5 scheme}

The discrete coefficients of the GBDF5 scheme read
\begin{align}\label{scheme: GBDF5}
	\vec{a}_{\mathrm{G}}^{(5)}
	=&\,\big(\tfrac{5\beta ^4+40\beta ^3+105 \beta ^2+100 \beta+24}{120},\tfrac{-10 \beta ^4-70 \beta ^3-135 \beta ^2-25 \beta +77}{60},\big.\\
	&\,\quad\big.\tfrac{15\beta ^4+90 \beta ^3+120 \beta ^2-45\beta-43}{60},
	\tfrac{-10 \beta ^4-50 \beta ^3-45 \beta ^2+25 \beta +17}{60},\big.\notag\\
	&\,\qquad\big.\tfrac{5 \beta ^4+20 \beta ^3+15 \beta ^2-10 \beta -6}{120}\big),\notag\\
	\vec{b}_{\mathrm{G}}^{(5)}=&\,\big(\tfrac{\beta(\beta ^3+6 \beta ^2+11 \beta +6)}{24},\tfrac{-\beta ^4-5 \beta ^3-5 \beta ^2+5 \beta +6}{6},\tfrac{\beta(\beta ^3+4 \beta ^2+\beta -6)}{4},\big.\notag\\
	&\,\qquad\big. \tfrac{\beta (-\beta ^3-3 \beta ^2+\beta +3)}{6}, \tfrac{\beta(\beta ^3+2 \beta ^2-\beta -2)}{24},0\big),\notag\\
	\vec{c}_{\mathrm{G}}^{(5)}=&\,\big(\tfrac{\beta ^4+10 \beta ^3+35 \beta ^2+50 \beta +24}{24},
	-\tfrac{\beta  (\beta ^3+9 \beta ^2+26 \beta +24)}{6},\big.\notag\\
	&\,\qquad\big.\tfrac{\beta(\beta ^3+8 \beta ^2+19 \beta +12)}{4},
	-\tfrac{\beta(\beta ^3+7 \beta ^2+14 \beta +8)}{6},
	\tfrac{\beta (\beta ^3+6 \beta ^2+11 \beta +6)}{24}\big).\notag
\end{align}
The first characteristic polynomial ${\varrho}_{a,\mathrm{G}}^{(5)}(\zeta)$ is a Schur polynomial if $\beta >0.50969$. 
The formula \eqref{scheme: general imex Truncation Error} gives the leading error 
\begin{align}\label{truncation: GBDF5}
	R_{\mathrm{G}}^{(5,\beta)}=\tfrac{24-15 \beta ^4-70 \beta ^3-75 \beta ^2+16 \beta}{720}u_t^{(6)}(t_n)\tau^5
	+\tfrac{\beta (\beta ^3+6 \beta ^2+11 \beta +6)}{24}\mathcal{F}_t^{(5)}[u(t_n)]\tau^5.
\end{align}
One has the following proposition.
\begin{proposition}\label{prop: GBDF5}
	For the GBDF5 scheme \eqref{scheme: GBDF5} with the parameter $\beta\ge1$, it holds that $\sigma_{\mathrm{F},\mathrm{G}}^{(5,\beta)}\ge 1,$
	\begin{align*}
		&\sigma_{\mathrm{E},\mathrm{G}}^{(5,\beta)}\ge\frac{5(2 \beta ^4+16 \beta ^3+40 \beta ^2+32 \beta +3)}{10 \beta ^4+60 \beta ^3+90 \beta ^2-32},\quad 
		\lambda_{\mathrm{I},\mathrm{G}}^{(5,\beta)}\le\frac{5(2 \beta ^4+8 \beta ^3+4 \beta ^2-8 \beta -3)}{10 \beta ^4+60 \beta ^3+90 \beta ^2-32}\\
		&\text{such that}\quad
		\mathfrak{I}_{\mathrm{IE},\mathrm{G}}^{(5,\beta)}\le
		\frac{2 \beta ^4+8 \beta ^3+4 \beta ^2-8 \beta -3}{2 \beta ^4+16 \beta ^3+40 \beta ^2+32 \beta +3}.
	\end{align*}
	If the parameter $\beta=20$, then $\sigma_{\mathrm{F},\mathrm{G}}^{(5,20)}\le \frac{399}{200},$
	\begin{align*}
		\sigma_{\mathrm{E},\mathrm{G}}^{(5,20)}\le\frac{2501825}{2115968},
		\quad\lambda_{\mathrm{I},\mathrm{G}}^{(5,20)}\ge \frac{1769085}{2115968}\quad
		\text{such that}\quad
		\mathfrak{I}_{\mathrm{IE},\mathrm{G}}^{(5,20)}\ge \frac{353817}{500365}\approx 0.70712.
	\end{align*}
\end{proposition} 

\begin{proof} With the discrete coefficients in \eqref{scheme: GBDF5}, one can write out 
	the associated semi-generating functions $a_{\mathrm{G}}^{(5)}(\theta)=\sum_{j=0}^{4}a_{\mathrm{G},j}^{(5)}e^{\imath j\theta}$,
	$b_{\mathrm{G}}^{(5)}(\theta)=\sum_{j=0}^{4}b_{\mathrm{G},j}^{(5)}e^{\imath j\theta}$
	and $c_{\mathrm{G}}^{(5)}(\theta)=\sum_{j=0}^{4}c_{\mathrm{G},j}^{(5)}e^{\imath j\theta}$.	
	According to the definitions in \eqref{def: lambda sigma},  $\sigma_{\mathrm{F},\mathrm{G}}^{(5,\beta)}\ge\frac{1}{\abst{a_{\mathrm{G}}^{(5)}(0)}}=1$,
	\begin{align*}
		&	\sigma_{\mathrm{E},\mathrm{G}}^{(5,\beta)}\ge
		\frac{\abst{c_{\mathrm{G}}^{(5)}(\pi)}}{\abst{a_{\mathrm{G}}^{(5)}(\pi)}}
		=\frac{5(2 \beta ^4+16 \beta ^3+40 \beta ^2+32 \beta +3)}{10 \beta ^4+60 \beta ^3+90 \beta ^2-32},\\
		&	\lambda_{\mathrm{I},\mathrm{G}}^{(5,\beta)}\le\frac{b_{\mathrm{G}}^{(5)}(\pi)}{a_{\mathrm{G}}^{(5)}(\pi)}
		=\frac{5(2 \beta ^4+8 \beta ^3+4 \beta ^2-8 \beta -3)}{10 \beta ^4+60 \beta ^3+90 \beta ^2-32}.
	\end{align*}

	To find the upper bounds of $\sigma_{\mathrm{F},\mathrm{G}}^{(5,\beta)}$, $\sigma_{\mathrm{E},\mathrm{G}}^{(5,\beta)}$ and the lower bound of $\lambda_{\mathrm{I},\mathrm{G}}^{(5,\beta)}$, it is to consider  
	the following auxiliary functions
	\begin{align*}
		{\mathcal{F}}_{\mathrm{G5}}(\beta,\theta):=&\,
		100\beta^2-(20\beta-1)^2a_{\mathrm{G}}^{(5)}(\theta)a_{\mathrm{G}}^{(5)}(-\theta),\\
		{\mathcal{E}}_{\mathrm{G5}}(\beta,\theta):=&\,(10 \beta ^4+60 \beta ^3+90 \beta ^2-32)^2c_{\mathrm{G}}^{(5)}(\theta)c_{\mathrm{G}}^{(5)}(-\theta)\\
		&\,-25(2 \beta ^4+21 \beta ^3+29 \beta ^2+38 \beta +5)^2a_{\mathrm{G}}^{(5)}(\theta)a_{\mathrm{G}}^{(5)}(-\theta),\\
		{\mathcal{I}}_{\mathrm{G5}}(\beta,\theta):=&\,(10 \beta ^4+60 \beta ^3+90 \beta ^2-32)\Re\kbrab{b_{\mathrm{G}}^{(5)}(\theta)a_{\mathrm{G}}^{(5)}(-\theta)}\\
		&\,-5 (2 \beta ^4+3 \beta ^3+25 \beta ^2-9 \beta -3)a_{\mathrm{G}}^{(5)}(\theta)a_{\mathrm{G}}^{(5)}(-\theta).
	\end{align*}
	For any $\beta\ge18$ such as $\beta=20$, the functions ${\mathcal{F}}_{\mathrm{G5}}(20,\theta)$, ${\mathcal{E}}_{\mathrm{G5}}(20,\theta)$ and ${\mathcal{I}}_{\mathrm{G5}}(20,\theta)$ are quartic polynomials with respect to $y=\cos\theta\in[-1,1]$. 
	Actually, it is easy to verify  that
	${\mathcal{F}}_{\mathrm{G5}}(20,\theta)\le0$, ${\mathcal{E}}_{\mathrm{G5}}(20,\theta)\le0$ and ${\mathcal{I}}_{\mathrm{G5}}(20,\theta)\ge0$
	for $\theta\in[0,2\pi)$ and they lead to the claimed bounds, while the technical details are omitted. For the general case of $\beta\ge1$, we will check the results numerically in Remark \ref{remark: lambda GBDF5} for some fixed $\beta$. 
	It completes the proof.
\end{proof}

\begin{figure}[htb!]
	\centering
	\subfigure[ $\widetilde{\mathcal{F}}_{\mathrm{G5}}(\beta,\theta)$]
	{\includegraphics[width=1.67in]{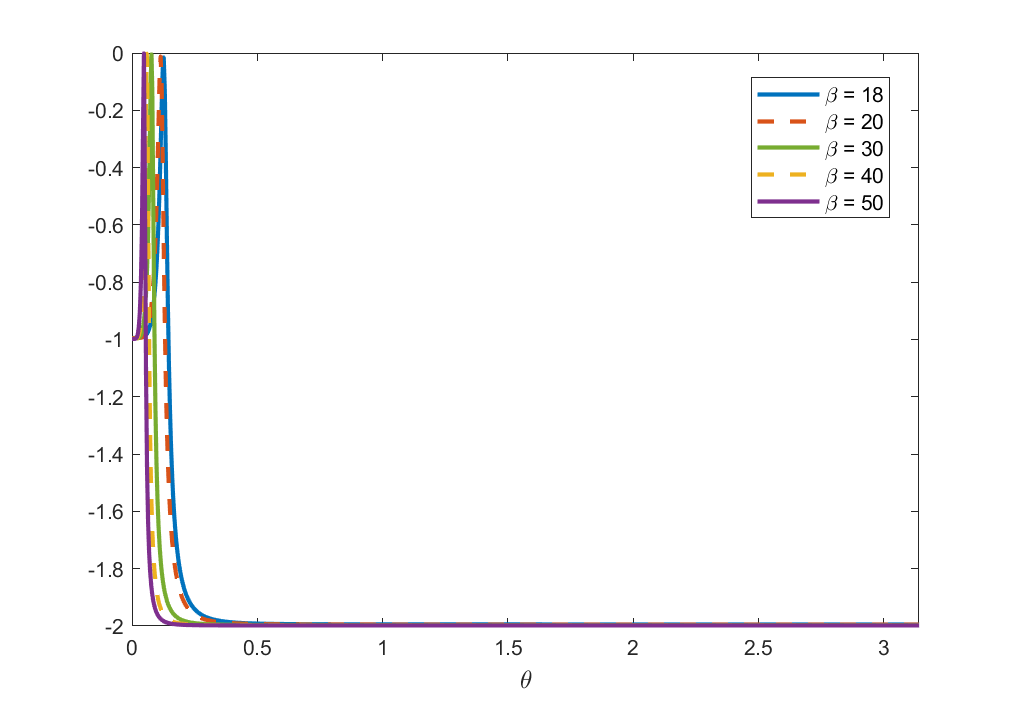}}
	\subfigure[ $\widetilde{\mathcal{E}}_{\mathrm{G5}}(\beta,\theta)$]
	{\includegraphics[width=1.67in]{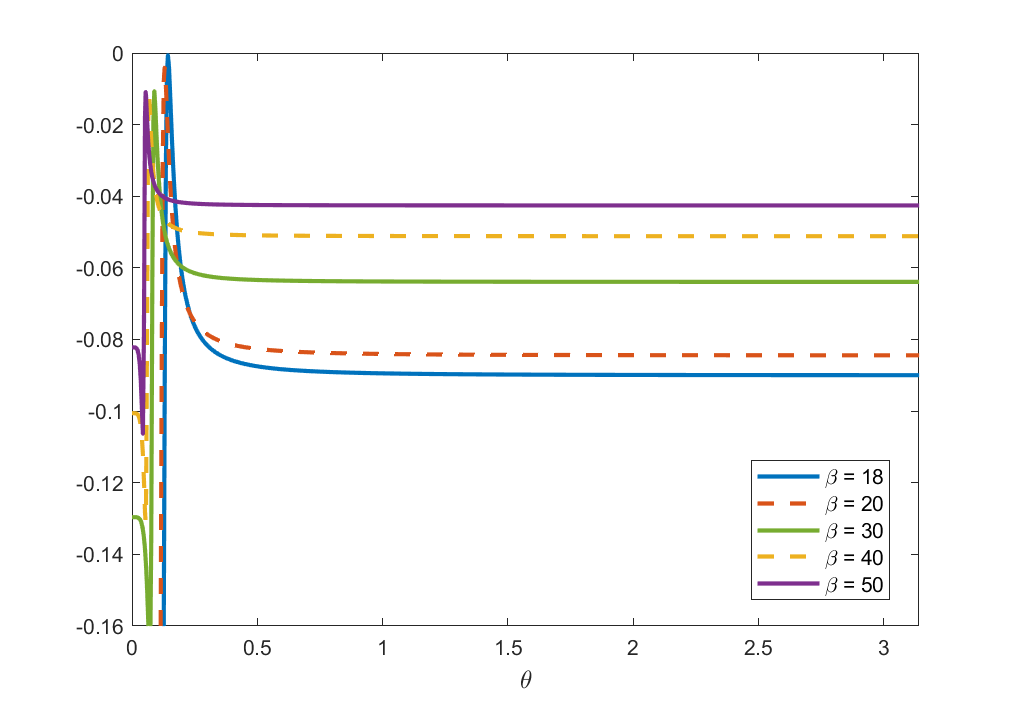}}
	\subfigure[ $\widetilde{\mathcal{I}}_{\mathrm{G5}}(\beta,\theta)$]
	{\includegraphics[width=1.67in]{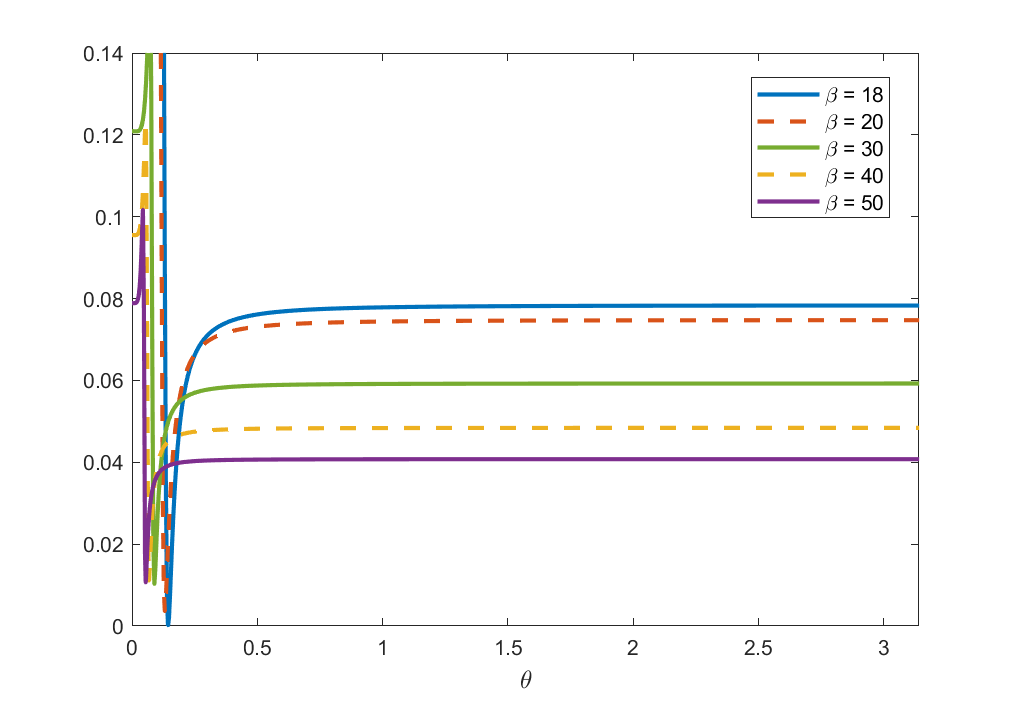}}
	\caption{Curves of $\widetilde{\mathcal{F}}_{\mathrm{G5}}(\beta,\theta)$, $\widetilde{\mathcal{E}}_{\mathrm{G5}}(\beta,\theta)$ and $\widetilde{\mathcal{I}}_{\mathrm{G5}}(\beta,\theta)$ for $\theta\in[0,\pi]$.}
	\label{fig: lambda GBDF5}
\end{figure}

\begin{figure}[htb!]
	\centering
	\subfigure[ $\widehat{\mathcal{F}}_{\mathrm{G5}}(\beta,\theta)$]
	{\includegraphics[width=1.67in]{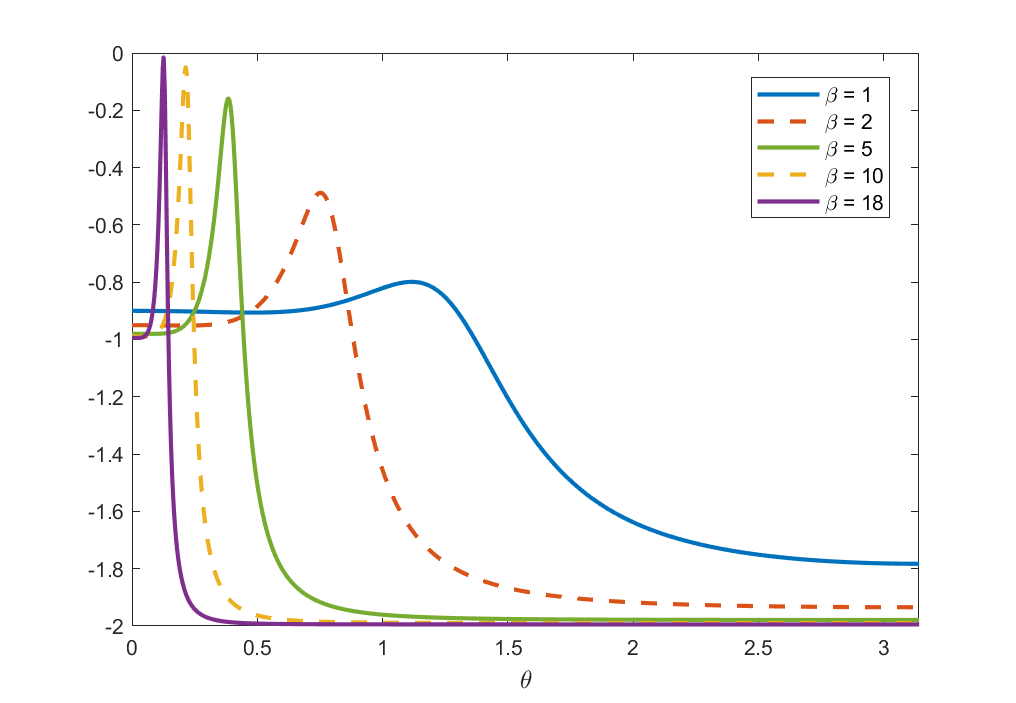}}
	\subfigure[ $\widehat{\mathcal{E}}_{\mathrm{G5}}(\beta,\theta)$]
	{\includegraphics[width=1.67in]{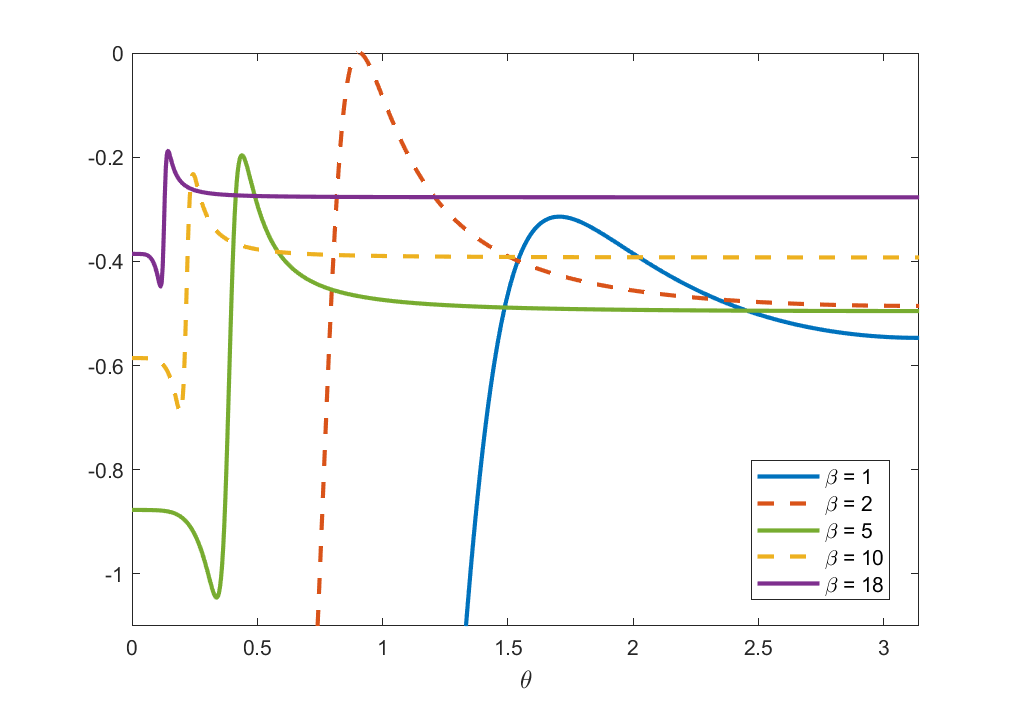}}
	\subfigure[ $\widehat{\mathcal{I}}_{\mathrm{G5}}(\beta,\theta)$]
	{\includegraphics[width=1.67in]{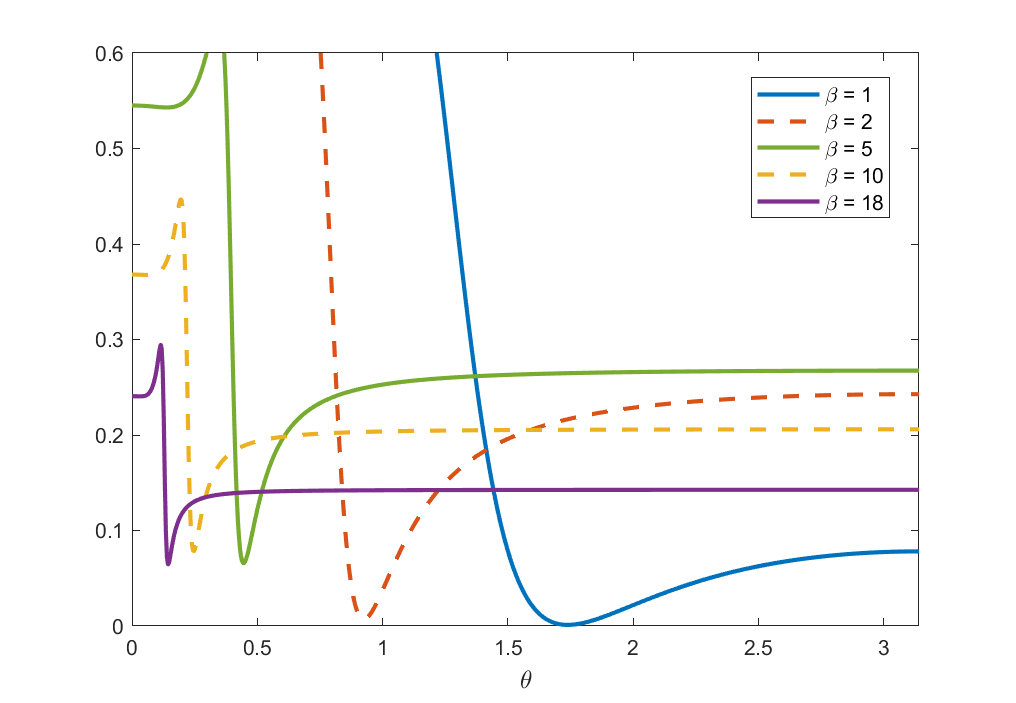}}
	\caption{Curves of $\widehat{\mathcal{F}}_{\mathrm{G5}}(\beta,\theta)$, $\widehat{\mathcal{E}}_{\mathrm{G5}}(\beta,\theta)$ and $\widehat{\mathcal{I}}_{\mathrm{G5}}(\beta,\theta)$ for $\theta\in[0,\pi]$.}
	\label{fig: lambda GBDF5small}
\end{figure}

\begin{remark}\label{remark: lambda GBDF5}
	For the GBDF5 scheme \eqref{scheme: GBDF5} with the parameter $\beta\ge18$, numerical tests suggest that
	\begin{align*}
		&	\widetilde{\mathcal{F}}_{\mathrm{G5}}(\beta,\theta):=\frac1{\abst{a_{\mathrm{G}}^{(5)}(\theta)}}- \frac{20\beta-1}{10\beta}\le0,\\
		&	\widetilde{\mathcal{E}}_{\mathrm{G5}}(\beta,\theta):=\frac{\abst{c_{\mathrm{G}}^{(5)}(\theta)}}{\abst{a_{\mathrm{G}}^{(5)}(\theta)}}-\frac{5 (2 \beta ^4+21 \beta ^3+29 \beta ^2+38 \beta +5)}{10 \beta ^4+60 \beta ^3+90 \beta ^2-32}\le0,\\
		&
		\widetilde{\mathcal{I}}_{\mathrm{G5}}(\beta,\theta):=\Re\kbraB{\frac{b_{\mathrm{G}}^{(5)}(\theta)}{a_{\mathrm{G}}^{(5)}(\theta)}}-\frac{5 (2 \beta ^4+3 \beta ^3+25 \beta ^2-9 \beta -3)}{10 \beta ^4+60 \beta ^3+90 \beta ^2-32}\ge0,
	\end{align*}	
	see Figure \ref{fig: lambda GBDF5}, 
	in which the auxiliary functions
	$\widetilde{\mathcal{F}}_{\mathrm{G5}}(\beta,\theta)$, $\widetilde{\mathcal{E}}_{\mathrm{G5}}(\beta,\theta)$ and $\widetilde{\mathcal{I}}_{\mathrm{G5}}(\beta,\theta)$ are depicted for $\theta\in[0,\pi]$ with the fixed parameters $\beta=18,20,30,40$ and 50.	They suggest that for $\beta\ge18$, $\sigma_{\mathrm{F},\mathrm{G}}^{(5,\beta)}\le \frac{20\beta-1}{10\beta},$
	\begin{align*}
		&\sigma_{\mathrm{E},\mathrm{G}}^{(5,\beta)}\le\frac{5 (2 \beta ^4+21 \beta ^3+29 \beta ^2+38 \beta +5)}{10 \beta ^4+60 \beta ^3+90 \beta ^2-32},\quad
		\lambda_{\mathrm{I},\mathrm{G}}^{(5,\beta)}\ge\frac{5 (2 \beta ^4+3 \beta ^3+25 \beta ^2-9 \beta -3)}{10 \beta ^4+60 \beta ^3+90 \beta ^2-32},\\
		&
		\text{such that}\quad
		\mathfrak{I}_{\mathrm{IE},\mathrm{G}}^{(5,\beta)}\ge\frac{2 \beta ^4+3 \beta ^3+25 \beta ^2-9 \beta -3}{2 \beta ^4+21 \beta ^3+29 \beta ^2+38 \beta +5}.
	\end{align*}	
	Also, numerical tests suggest the rough bounds for $1\le \beta<18$, $\sigma_{\mathrm{F},\mathrm{G}}^{(5,\beta)}\le \frac{20\beta-1}{10\beta},$
	\begin{align*}
		&\sigma_{\mathrm{E},\mathrm{G}}^{(5,\beta)}\le\frac{5(2 \beta ^4+30 \beta ^3+32 \beta ^2+38 \beta +5)}{10 \beta ^4+60 \beta ^3+90 \beta ^2-32},\quad
		\lambda_{\mathrm{I},\mathrm{G}}^{(5,\beta)}\ge\frac{5(2 \beta ^4+\beta ^3+4 \beta ^2-4 \beta -2)}{10 \beta ^4+60 \beta ^3+90 \beta ^2-32}\\
		&
		\text{such that}\quad
		\mathfrak{I}_{\mathrm{IE},\mathrm{G}}^{(5,\beta)}\ge\frac{2 \beta ^4+\beta ^3+4 \beta ^2-4 \beta -2}{2 \beta ^4+30 \beta ^3+32 \beta ^2+38 \beta +5},
	\end{align*}	
	see Figure \ref{fig: lambda GBDF5small}, 
	where the auxiliary functions
	$\widehat{\mathcal{F}}_{\mathrm{G5}}(\beta,\theta)$, $\widehat{\mathcal{E}}_{\mathrm{G5}}(\beta,\theta)$ and $\widehat{\mathcal{I}}_{\mathrm{G5}}(\beta,\theta)$ are defined in similar to $\widetilde{\mathcal{F}}_{\mathrm{G5}}(\beta,\theta)$, $\widetilde{\mathcal{E}}_{\mathrm{G5}}(\beta,\theta)$ and $\widetilde{\mathcal{I}}_{\mathrm{G5}}(\beta,\theta)$, respectively.
\end{remark}

\subsection{$\delta$-parameterized NIMEX methods}

This subsection revisits the $\delta$-parameterized NIMEX  schemes \cite{RosalesSeiboldShirokoffZhou:2017-SM} with some comments for the value of implicit-explicit controllability intensity. The $\rmk$-step NIMEX  schemes \eqref{scheme: NIMEX multistep} are determined by the three characteristic polynomials (for the sake of consistency in the present context, the range of free parameter $\delta$ is modified from   $0<\delta\le1$ in the original papers \cite{RosalesSeiboldShirokoffZhou:2017-SM} to $\delta\ge1$ with the transform $\delta\leftarrow1/\delta$)
$$\varrho_{a,\mathrm{N}}^{(\rmk)}(\zeta):=\sum_{j=1}^{\rmk}\frac{f_{\mathrm{N}}^{(j)}(1)}{j!}(\zeta-1)^{j-1}\quad \text{with}\quad f_{\mathrm{N}}(z)=(\delta z-\delta+1)^{\rmk}\ln z,$$ 
$\varrho_{b,\mathrm{N}}^{(\rmk)}(\zeta):=(\delta\zeta-\delta+1)^{\rmk}$ and $\varrho_{c,\mathrm{N}}^{(\rmk)}(\zeta):=(\delta\zeta-\delta+1)^{\rmk}-\delta^{\rmk}(\zeta-1)^{\rmk}$,
respectively.

We take $\rmk=2$ with $\delta\ge6/5$. 
The formula \eqref{scheme: general imex Truncation Error} gives the leading error 
\begin{align*}
	R_{\mathrm{N}}^{(2,\delta)}=-\frac{3 \delta ^2-3 \delta +1}{3}u_t^{(3)}(t_n)\tau^2+\delta ^2 \mathcal{F}_t^{(2)}[u(t_n)]\tau^2.
\end{align*}
The corresponding semi-generating functions are $a_{\mathrm{N}}^{(2)}(\theta)=2 \delta -\frac{1}{2}+(\frac{3}{2}-2 \delta)e^{\imath\theta}$,
$b_{\mathrm{N}}^{(2)}(\theta)=\delta ^2+2 \delta(1 -\delta)e^{\imath\theta}+(1-\delta)^2e^{2\imath\theta}$ and $c_{\mathrm{N}}^{(2)}(\theta)=2 \delta+(1-2 \delta)e^{\imath\theta}$. By following the proof of Proposition \ref{prop: GBDF3}, one can find that $\sigma_{\mathrm{F},\mathrm{N}}^{(2,\delta)}=1,$ 
$\sigma_{\mathrm{E},\mathrm{N}}^{(2,\delta)}=\frac{4 \delta-1}{4 \delta-2}$ and
$$\lambda_{\mathrm{I},\mathrm{N}}^{(2,\delta)}=\frac{4 \sqrt{2} \sqrt{(2 \delta -1)^3 \left(4 \delta ^2-5 \delta +1\right)^2}+192 \delta ^4-416 \delta ^3+316 \delta ^2-100 \delta +11}{\left(16 \delta ^2-16 \delta +3\right)^2}$$
such that
$$\mathfrak{I}_{\mathrm{IE},\mathrm{N}}^{(2,\delta)}=\frac{2 \braB{192 \delta ^4-416 \delta ^3+316 \delta ^2-100 \delta +11+4 \sqrt{2} \sqrt{(2 \delta -1)^3 \left(4 \delta ^2-5 \delta +1\right)^2}}}{(2 \delta -1)^{-1}(3-4 \delta )^2 (4 \delta -1)^3}.$$
One can find that $\max_{\delta\ge6/5}\lambda_{\mathrm{I},\mathrm{N}}^{(2,\delta)}=\lambda_{\mathrm{I},\mathrm{N}}^{(2,\frac{11}4)}
=\frac{27}{32}$ and $\max_{\delta\ge6/5}\mathfrak{I}_{\mathrm{IE},\mathrm{N}}^{(2,\delta)}\approx 0.795354$ as $\delta\approx8.5176$. However, the controllability intensity $\mathfrak{I}_{\mathrm{IE},\mathrm{G}}^{(2,\beta)}=\tfrac{2\beta-1}{2\beta +1}$ of GBDF2 scheme for the parameter  $\beta\ge \frac{9}{2}$ is always larger than the maximum value of $\mathfrak{I}_{\mathrm{IE},\mathrm{N}}^{(2,\delta)}$.

Consider  $\rmk=3$ with  $\delta\ge2$. 
The formula \eqref{scheme: general imex Truncation Error} gives the leading error 
\begin{align*}
	R_{\mathrm{N}}^{(3,\delta)}=\frac{1}{4} \left(-4 \delta ^3+6 \delta ^2-4 \delta +1\right)u_t^{(4)}(t_n)\tau^3+\delta ^3 \mathcal{F}_t^{(3)}[u(t_n)]\tau^3.
\end{align*}
The semi-generating functions are
\begin{align*}
	&a_{\mathrm{N}}^{(3)}(\theta)=3 \delta ^2-\frac{3 \delta }{2}+\frac{1}{3}-(6 \delta ^2-6 \delta +\frac{7}{6})e^{\imath\theta}+(3 \delta ^2-\frac{9 \delta }{2}+\frac{11}{6})e^{2\imath\theta},\\
	&b_{\mathrm{N}}^{(3)}(\theta)=\delta ^3+3\delta^2(1 -\delta)e^{\imath\theta}+3\delta(1 -\delta)^2e^{2\imath\theta}
	+(1-\delta)^3e^{3\imath\theta},
\end{align*}
and $c_{\mathrm{N}}^{(3)}(\theta)=3 \delta ^2+3 \delta(1-2 \delta)e^{\imath\theta}
+(3 \delta ^2-3 \delta +1)e^{2\imath\theta}$. By following the proof of Proposition \ref{prop: GBDF3}, we obtain that $\sigma_{\mathrm{F},\mathrm{N}}^{(3,\delta)}=1,$ 
$\sigma_{\mathrm{E},\mathrm{N}}^{(3,\delta)}=\frac{36 \delta ^2-18 \delta +3}{36 \delta ^2-36 \delta +10}$ and
$$\lambda_{\mathrm{I},\mathrm{N}}^{(3,\delta)}\ge \frac{24 \delta ^2-21 \delta +3}{36 \delta ^2-36 \delta +10}
\quad\text{ such that}\quad
\mathfrak{I}_{\mathrm{IE},\mathrm{N}}^{(3,\delta)}\ge\frac{24 \delta ^2-21 \delta +3}{36 \delta ^2-18 \delta +3}.
$$
One can find that $\frac{19}{37}\le\mathfrak{I}_{\mathrm{IE},\mathrm{N}}^{(3,\delta)}<\frac{88}{125}$ for  $\delta\ge2$. 
However, if  $\beta\ge 5.39469$, 
the controllability intensity $\mathfrak{I}_{\mathrm{IE},\mathrm{G}}^{(3,\beta)}$ of GBDF3 scheme is always larger than $\mathfrak{I}_{\mathrm{IE},\mathrm{N}}^{(3,\delta)}$.

Although the controllability intensities $\mathfrak{I}_{\mathrm{IE},\mathrm{N}}^{(\rmk,\delta)}$ of NIMEX-$\rmk$ methods \eqref{scheme: NIMEX multistep}	are not comparable to those of the GBDF-$\rmk$ schemes, we find that 
the NIMEX-$\rmk$ methods \eqref{scheme: NIMEX multistep} for $2\le\rmk\le9$ fulfill the assumptions of Lemma \ref{lemma: bound quadratic form} for proper parameter $\delta$. 
Actually, the first characteristic polynomial $\varrho_{a,\mathrm{N}}^{(\rmk)}(\zeta)$  is a Schur polynomial if the parameter $\delta>\frac{1}{2}$, $\delta>\frac{1}{2}$, $\delta>\frac{1}{2}$, $\delta >0.705247$, $\delta>0.856554$, $\delta>1.0241$, $\delta>1.20559$ and $\delta>1.39984$ corresponding to the order index $\rmk=2,3,\cdots,8$ and 9, respectively.

\section{$\gamma$-parameterized SIELM methods}\label{sec: SM-SIELM}

This section discusses a new class of $\gamma$-parameterized SIELM methods for which the associated implicit-explicit controllability intensity $\mathfrak{I}_{\mathrm{IE}}^{(\rmk)}$ can approach the optimal value 1 as the parameter $\gamma$ is properly large, especially for $2\le\rmk\le5$. For the nonlinear parabolic model \eqref{nonlinearModel}, they can be formulated as follows
\begin{align}\label{scheme: our IELM multistep-SM}	
	\sum_{j=0}^{\rmk-1}a_{\mathrm{S},j}^{(\rmk)}\partial_{\tau}u^{n-j}
	+\varpi\sum_{j=0}^{\rmk}b_{\mathrm{S},j}^{(\rmk)}\mathcal{L}u^{n-j}
	=\sum_{j=0}^{\rmk-1}c_{\mathrm{S},j}^{(\rmk)}\mathcal{F}(u^{n-j-1})\quad\text{for $n\ge\rmk$,}
\end{align}
where  $a_{\mathrm{S},j}^{(\rmk)}$, $b_{\mathrm{S},j}^{(\rmk)}$ and $c_{\mathrm{S},j}^{(\rmk)}$ are determined by the three characteristic polynomials 
$${\varrho}_{a,\mathrm{S}}^{(\rmk)}(\zeta):=\sum_{j=1}^{\rmk}\frac{f_{\mathrm{S}}^{(j)}(1)}{j!}(\zeta-1)^{j-1}\quad \text{with}\quad f_{\mathrm{S}}(z)=(\gamma z-\gamma+1)^{\rmk-1}z\ln z,$$ 
$\varrho_{b,\mathrm{S}}^{(\rmk)}(\zeta):=\zeta(\gamma\zeta-\gamma+1)^{\rmk-1}$ and $\varrho_{c,\mathrm{S}}^{(\rmk)}(\zeta):=\zeta(\gamma\zeta-\gamma+1)^{\rmk-1}-\gamma^{\rmk-1}(\zeta-1)^{\rmk}$,
respectively. By the Routh-Hurwitz criterion,  one can check that  all roots of $\varrho_{a,\mathrm{S}}^{(\rmk)}(\zeta)$ satisfy $\abs{\zeta}<1$ if  $\gamma>0$, $\gamma>\frac{1}{\sqrt{6}}$, $\gamma>\frac{1}{2}$, $\gamma >0.64752$, $\gamma >0.83037$, $\gamma >1.02816$, $\gamma>1.23687$ and $\gamma>1.45372$ corresponding to the order index $\rmk=2,3,\cdots,8$ and 9, respectively.


\subsection{SIELM2 scheme}
As noted, the SIELM2 scheme is the same to the WBDF2 or GBDF2 scheme. The associated discrete  coefficients
\begin{align}\label{scheme: SIELM2}
	&\vec{a}_{\mathrm{S}}^{(2)}=(\tfrac{1}{2}+\gamma,\tfrac{1}{2}-\gamma),\quad
	\vec{b}_{\mathrm{S}}^{(2)}=(\gamma,1-\gamma,0),\quad
	\vec{c}_{\mathrm{S}}^{(2)}=(\gamma +1,-\gamma ).
\end{align}
The formula \eqref{scheme: general imex Truncation Error} gives the leading error 
\begin{align*}
	R_{\mathrm{S}}^{(2,\gamma)}=-\frac{3 \gamma-1}{6}u_t^{(3)}(t_n)\tau^2+\gamma \mathcal{F}_t^{(2)}[u(t_n)]\tau^2.
\end{align*}
The results of Proposition \ref{prop: WBDF2} gives $\sigma_{\mathrm{F},\mathrm{S}}^{(2,\gamma)}=1,$ 
$\sigma_{\mathrm{E},\mathrm{S}}^{(2,\gamma)}=\frac{2\gamma+1}{2\gamma}$, $\lambda_{\mathrm{I},\mathrm{S}}^{(2,\gamma)}=\frac{2\gamma-1}{2\gamma}$ and
then $\mathfrak{I}_{\mathrm{IE},\mathrm{S}}^{(2,\gamma)}
=\frac{2\gamma-1}{2\gamma+1}$ for the free parameter $\gamma\ge1$.

\subsection{SIELM3 scheme}   
The SIELM3 scheme has the discrete coefficients
\begin{align}\label{scheme: SIELM3}
	&\vec{a}_{\mathrm{S}}^{(3)}
	=\brab{\gamma ^2+\gamma -\tfrac{1}{6},\tfrac{5}{6}-2 \gamma ^2,\gamma ^2-\gamma +\tfrac{1}{3}},\\
	&\vec{b}_{\mathrm{S}}^{(3)}=\brab{\gamma ^2,2 \gamma -2 \gamma ^2,\gamma ^2-2 \gamma +1,0},
	\vec{c}_{\mathrm{S}}^{(3)}=(\gamma ^2+2 \gamma ,-2 \gamma ^2-2 \gamma +1,\gamma ^2).\notag
\end{align}
The formula \eqref{scheme: general imex Truncation Error} gives the leading error 
\begin{align}\label{truncation: SIELM3}
	R_{\mathrm{S}}^{(3,\gamma)}=-\frac{6 \gamma ^2-4 \gamma +1}{12}u_t^{(4)}(t_n)\tau^3+\gamma ^2  \mathcal{F}_t^{(3)}[u(t_n)]\tau^3.
\end{align}
By following the proof of Proposition \ref{prop: WBDF3}, it is easy to verify the following result.
\begin{proposition}\label{prop: SIELM3}
	For the SIELM3 scheme \eqref{scheme: SIELM3} with the free parameter $ \gamma\ge1$, it holds that
	$\sigma_{\mathrm{F},\mathrm{S}}^{(3, \gamma)}=1,$ 
	\begin{align*}
		\sigma_{\mathrm{E},\mathrm{S}}^{(3, \gamma)}=\frac{3 (4 \gamma ^2+4 \gamma -1)}{12 \gamma ^2-2},\;\;
		\lambda_{\mathrm{I},\mathrm{S}}^{(3, \gamma)}=\frac{3 (2\gamma-1)^2}{12 \gamma ^2-2}
		\;\;\text{such that}\;\;
		\mathfrak{I}_{\mathrm{IE},\mathrm{S}}^{(3, \gamma)}=\frac{(2\gamma-1)^2}{4 \gamma ^2+4 \gamma -1}.
	\end{align*}
\end{proposition}

\subsection{SIELM4 scheme}   
The SIELM4 scheme has the discrete coefficients
\begin{align}\label{scheme: SIELM4}
	&\vec{a}_{\mathrm{S}}^{(4)}
	=\big(\gamma ^3+\tfrac{3 \gamma ^2}{2}-\tfrac{\gamma }{2}+\tfrac{1}{12},-3 \gamma ^3-\tfrac{3 \gamma ^2}{2}+3 \gamma -\tfrac{5}{12},\big.\\
	&\big.\hspace{3cm}3 \gamma ^3-\tfrac{3 \gamma ^2}{2}-\tfrac{3 \gamma }{2}+\tfrac{13}{12},-\gamma ^3+\tfrac{3 \gamma ^2}{2}-\gamma +\tfrac{1}{4}\big),\notag\\
	&\vec{b}_{\mathrm{S}}^{(4)}=\brab{\gamma ^3,3 \gamma ^2-3 \gamma ^3,3 \gamma ^3-6 \gamma ^2+3 \gamma ,-\gamma ^3+3 \gamma ^2-3 \gamma +1,0},\notag\\
	&\vec{c}_{\mathrm{S}}^{(4)}=(\gamma ^3+3 \gamma ^2,-3 \gamma ^3-6 \gamma ^2+3 \gamma ,3 \gamma ^3+3 \gamma ^2-3 \gamma +1,-\gamma ^3).\notag
\end{align}
The formula \eqref{scheme: general imex Truncation Error} gives the leading error 
\begin{align}\label{truncation: SIELM4}
	R_{\mathrm{S}}^{(4,\gamma)}=-\frac{10 \gamma ^3-10 \gamma ^2+5 \gamma -1}{20} 
	u_t^{(5)}(t_n)\tau^4+\gamma ^3  \mathcal{F}_t^{(4)}[u(t_n)]\tau^4.
\end{align}
By following the proof of Proposition \ref{prop: WBDF3}, one can verify the following result.
\begin{proposition}\label{prop: SIELM4}
	For the SIELM4 scheme \eqref{scheme: SIELM4} with the free parameter $ \gamma\ge6/5$, it holds that
	\begin{align*}
		&\sigma_{\mathrm{F},\mathrm{S}}^{(4, \gamma)}=1,\quad	
		\sigma_{\mathrm{E},\mathrm{S}}^{(4, \gamma)}=\frac{3 (8 \gamma ^3+12 \gamma ^2-6 \gamma +1)}{4(6 \gamma ^3-3 \gamma +1)},\quad
		\lambda_{\mathrm{I},\mathrm{S}}^{(4, \gamma)}=\frac{3 (2 \gamma -1)^3}{4 (6 \gamma ^3-3 \gamma +1)}
		\\
		&\text{such that}\quad
		\mathfrak{I}_{\mathrm{IE},\mathrm{S}}^{(4, \gamma)}=\frac{(2 \gamma -1)^3}{8 \gamma ^3+12 \gamma ^2-6 \gamma +1}.
	\end{align*}
\end{proposition}
As numerical illustrations for the claimed results in Proposition \ref{prop: SIELM4}, Figure \ref{fig: lambda SIELM4} depicts the following three auxiliary functions
\begin{align*}
	&\widetilde{\mathcal{F}}_{\mathrm{S4}}(\gamma,\theta):=\frac1{\abst{a_{\mathrm{S}}^{(4)}(\theta)}}- 1,\quad
	\widetilde{\mathcal{E}}_{\mathrm{S4}}(\gamma,\theta):=\frac{\abst{c_{\mathrm{S}}^{(4)}(\theta)}}{\abst{a_{\mathrm{S}}^{(4)}(\theta)}}-\frac{3 (8 \gamma ^3+12 \gamma ^2-6 \gamma +1)}{4(6 \gamma ^3-3 \gamma +1)},
	\\
	&\widetilde{\mathcal{I}}_{\mathrm{S4}}(\gamma,\theta):=\Re\kbraB{\frac{b_{\mathrm{S}}^{(4)}(\theta)}{a_{\mathrm{S}}^{(4)}(\theta)}}-\frac{3 (2 \gamma -1)^3}{4 (6 \gamma ^3-3 \gamma +1)}
\end{align*}
for $\theta\in[0,\pi]$ with the fixed parameters $\gamma=6/5,3,7,10,30$.

\begin{figure}[htb!]
	\centering
	\subfigure[ $\widetilde{\mathcal{F}}_{\mathrm{S4}}(\gamma,\theta)$]
	{\includegraphics[width=1.67in]{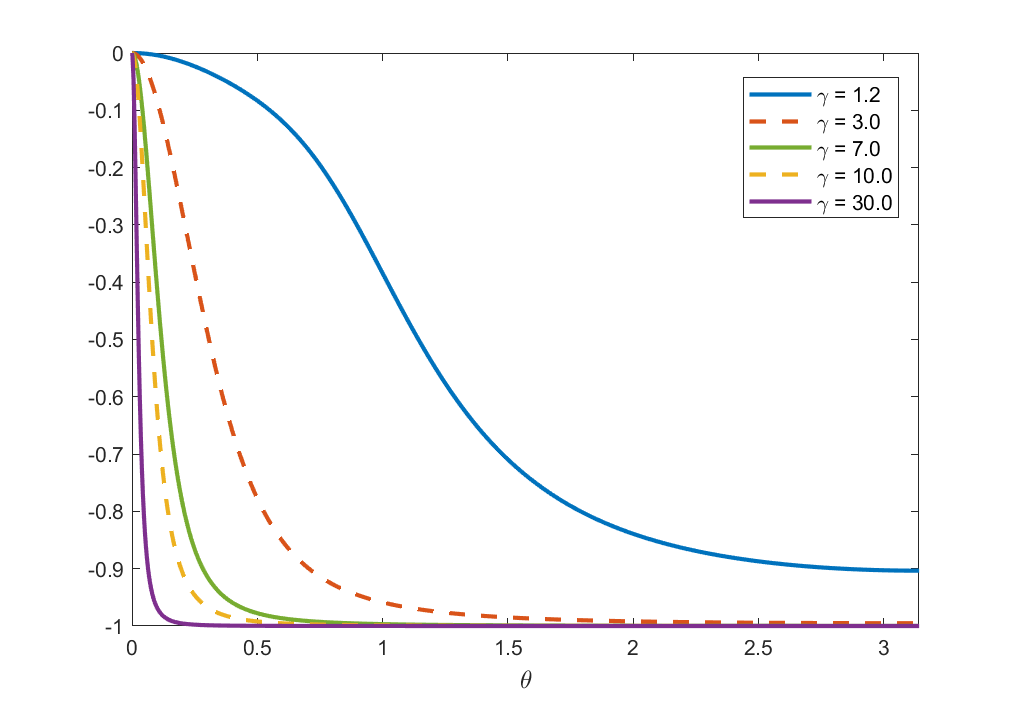}}
	\subfigure[ $\widetilde{\mathcal{E}}_{\mathrm{S4}}(\gamma,\theta)$]
	{\includegraphics[width=1.67in]{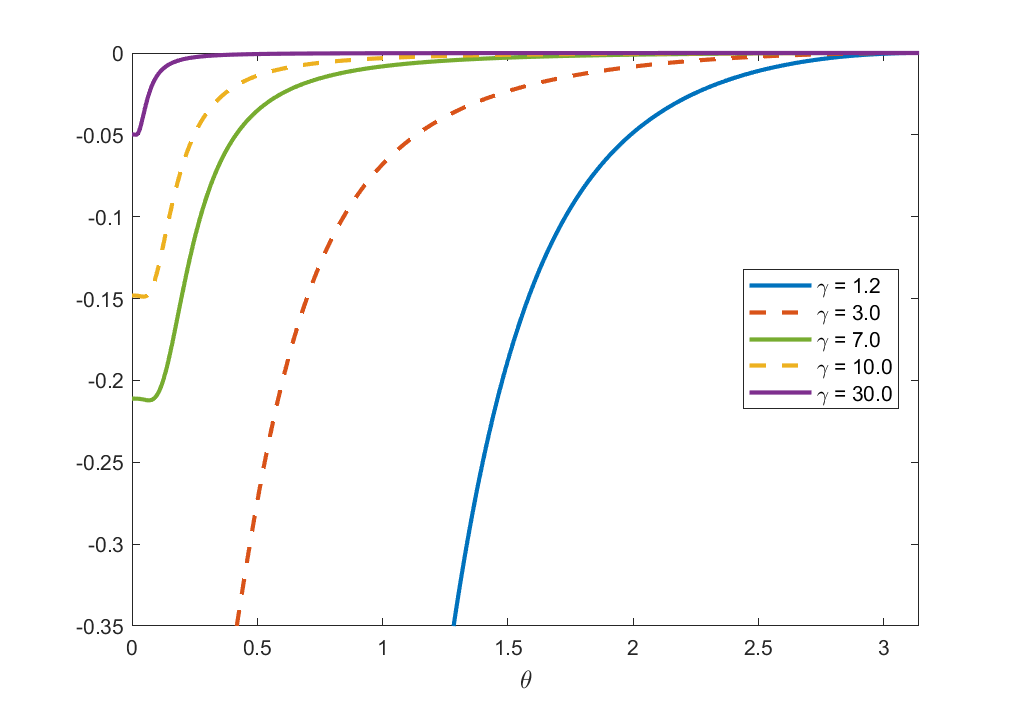}}
	\subfigure[ $\widetilde{\mathcal{I}}_{\mathrm{S4}}(\gamma,\theta)$]
	{\includegraphics[width=1.67in]{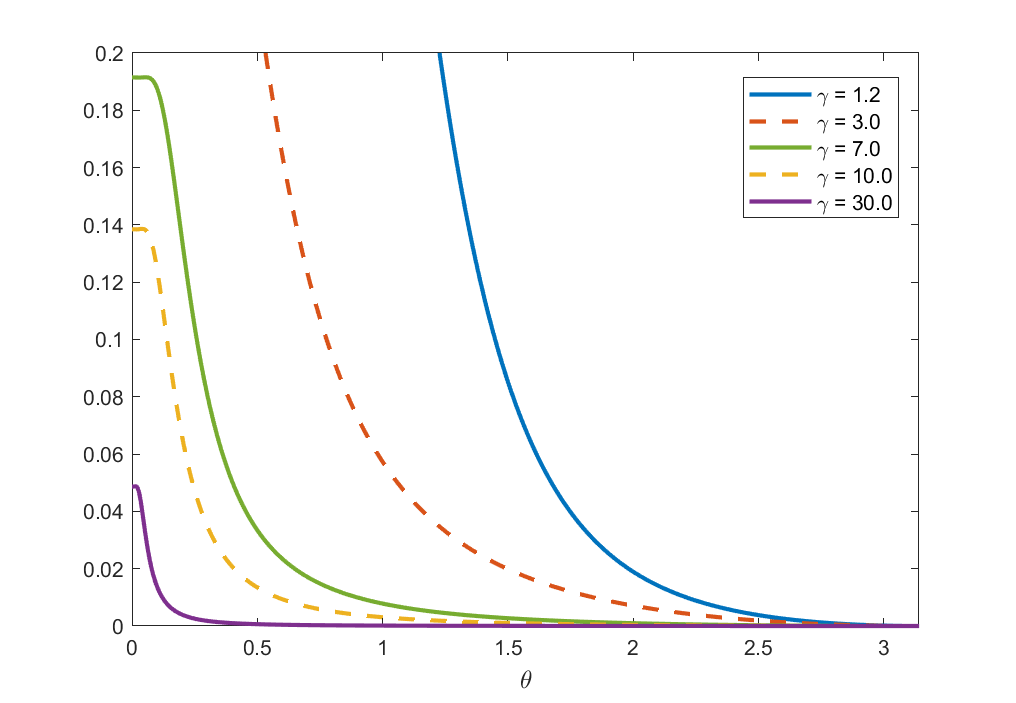}}
	\caption{Curves of $\widetilde{\mathcal{F}}_{\mathrm{S4}}(\gamma,\theta)$, $\widetilde{\mathcal{E}}_{\mathrm{S4}}(\gamma,\theta)$ and $\widetilde{\mathcal{I}}_{\mathrm{S4}}(\gamma,\theta)$ for $\theta\in[0,\pi]$.}
	\label{fig: lambda SIELM4}
\end{figure}

\subsection{SIELM5 scheme}   
The SIELM5 scheme has the discrete coefficients
\begin{align}\label{scheme: SIELM5}
	&\vec{a}_{\mathrm{S}}^{(5)}
	=\big(\gamma ^4+2 \gamma ^3-\gamma ^2+\tfrac{\gamma }{3}-\tfrac{1}{20},-4 \gamma ^4-4 \gamma ^3+7 \gamma ^2-2 \gamma +\tfrac{17}{60},\\
	&\qquad\qquad\big.6 \gamma ^4-9 \gamma ^2+6 \gamma -\tfrac{43}{60},
	-4 \gamma ^4+4 \gamma ^3+\gamma ^2-\tfrac{10 \gamma }{3}+\tfrac{77}{60},\notag\\
	&\qquad\qquad\gamma ^4-2 \gamma ^3+2 \gamma ^2-\gamma +\tfrac{1}{5}\big),\notag
\end{align}
while $\vec{b}_{\mathrm{S}}^{(5)}$ and $\vec{c}_{\mathrm{S}}^{(5)}$ can be generated by the 
second and third characteristic polynomials $\varrho_{b,\mathrm{S}}^{(5)}(\zeta):=\zeta(\gamma\zeta-\gamma+1)^{4}$ and $\varrho_{c,\mathrm{S}}^{(5)}(\zeta):=\zeta(\gamma\zeta-\gamma+1)^{4}-\gamma^{4}(\zeta-1)^{5}$,
respectively.
The formula \eqref{scheme: general imex Truncation Error} gives the leading error 
\begin{align}\label{truncation: SIELM5}
	R_{\mathrm{S}}^{(5,\gamma)}=-\frac{15 \gamma ^4-20 \gamma ^3+15 \gamma ^2-6 \gamma +1}{30}
	u_t^{(6)}(t_n)\tau^5+\gamma ^4  \mathcal{F}_t^{(5)}[u(t_n)]\tau^5.
\end{align}
By following the proof of Proposition \ref{prop: WBDF3}, one can verify the following result.
\begin{proposition}\label{prop: SIELM5}
	For the SIMLM5 scheme \eqref{scheme: SIELM5} with the free parameter $ \gamma\ge7/5$, it holds that $\sigma_{\mathrm{F},\mathrm{S}}^{(5, \gamma)}=1,$
	\begin{align*}
		&\sigma_{\mathrm{E},\mathrm{S}}^{(5, \gamma)}=
		\frac{15(16 \gamma ^4+32 \gamma ^3-24 \gamma ^2+8 \gamma -1)}{16 (15 \gamma ^4-15 \gamma ^2+10 \gamma -2)},\;\;
		\lambda_{\mathrm{I},\mathrm{S}}^{(5, \gamma)}=\frac{15 (2\gamma-1)^4}{16(15 \gamma ^4-15 \gamma ^2+10 \gamma -2)}
		\\
		&\text{such that}\quad
		\mathfrak{I}_{\mathrm{IE},\mathrm{S}}^{(5, \gamma)}=\frac{(2\gamma-1)^4}{16 \gamma ^4+32 \gamma ^3-24 \gamma ^2+8 \gamma -1}.
	\end{align*}
\end{proposition}
As numerical illustrations for the claimed results in Proposition \ref{prop: SIELM5}, Figure \ref{fig: lambda SIELM5} depicts the following three auxiliary functions $\widetilde{\mathcal{F}}_{\mathrm{S5}}(\gamma,\theta):=\frac1{\abst{a_{\mathrm{S}}^{(5)}(\theta)}}- 1$,
\begin{align*}
	&\widetilde{\mathcal{E}}_{\mathrm{S5}}(\gamma,\theta):=\frac{\abst{c_{\mathrm{S}}^{(5)}(\theta)}}{\abst{a_{\mathrm{S}}^{(5)}(\theta)}}-	\frac{15(16 \gamma ^4+32 \gamma ^3-24 \gamma ^2+8 \gamma -1)}{16 (15 \gamma ^4-15 \gamma ^2+10 \gamma -2)},
	\\
	&\widetilde{\mathcal{I}}_{\mathrm{S5}}(\gamma,\theta):=\Re\kbraB{\frac{b_{\mathrm{S}}^{(5)}(\theta)}{a_{\mathrm{S}}^{(5)}(\theta)}}-\frac{15 (2\gamma-1)^4}{16(15 \gamma ^4-15 \gamma ^2+10 \gamma -2)}
\end{align*}
for $\theta\in[0,\pi]$ with the parameter $\gamma=7/5,3,7,10$ and 30.

\begin{figure}[htb!]
	\centering
	\subfigure[ $\widetilde{\mathcal{F}}_{\mathrm{S5}}(\gamma,\theta)$]
	{\includegraphics[width=1.67in]{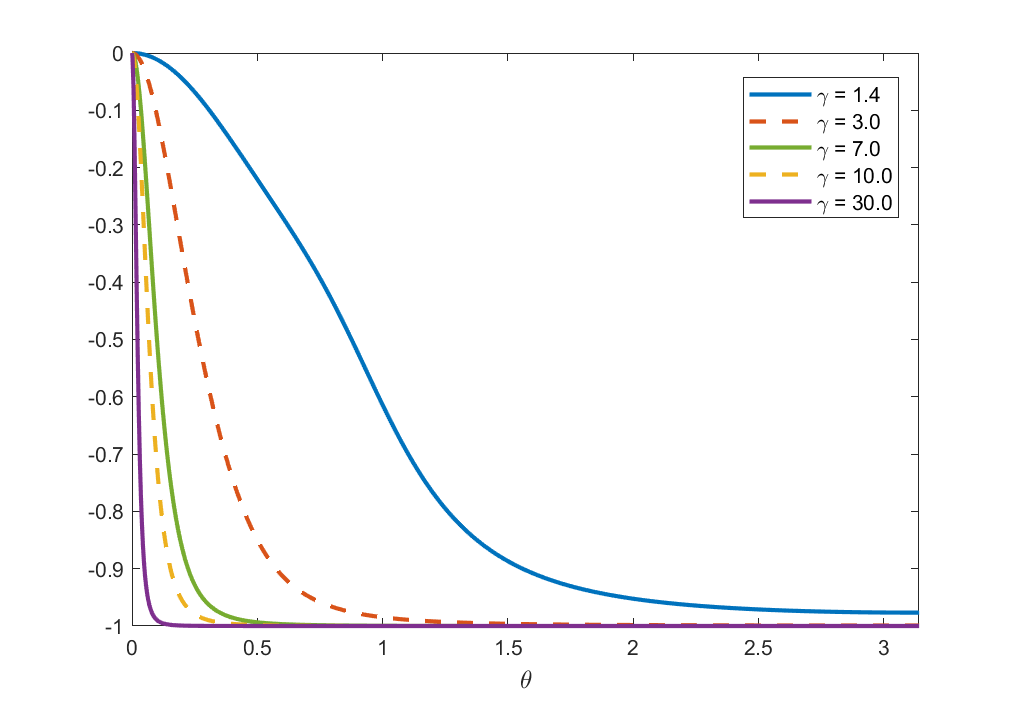}}
	\subfigure[ $\widetilde{\mathcal{E}}_{\mathrm{S5}}(\gamma,\theta)$]
	{\includegraphics[width=1.67in]{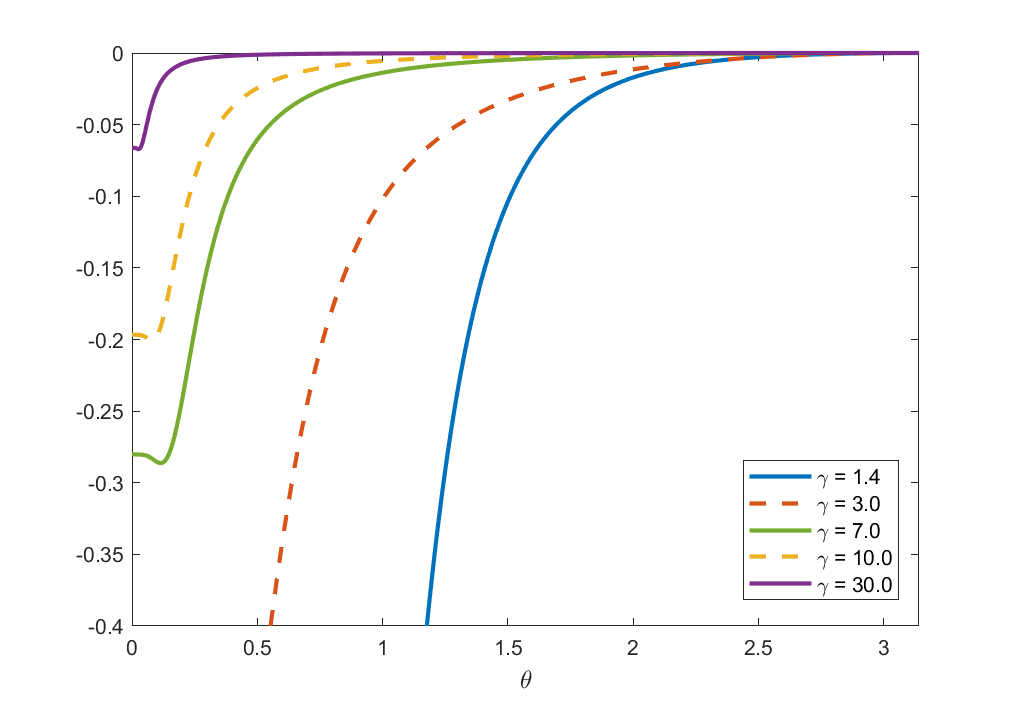}}
	\subfigure[ $\widetilde{\mathcal{I}}_{\mathrm{S5}}(\gamma,\theta)$]
	{\includegraphics[width=1.67in]{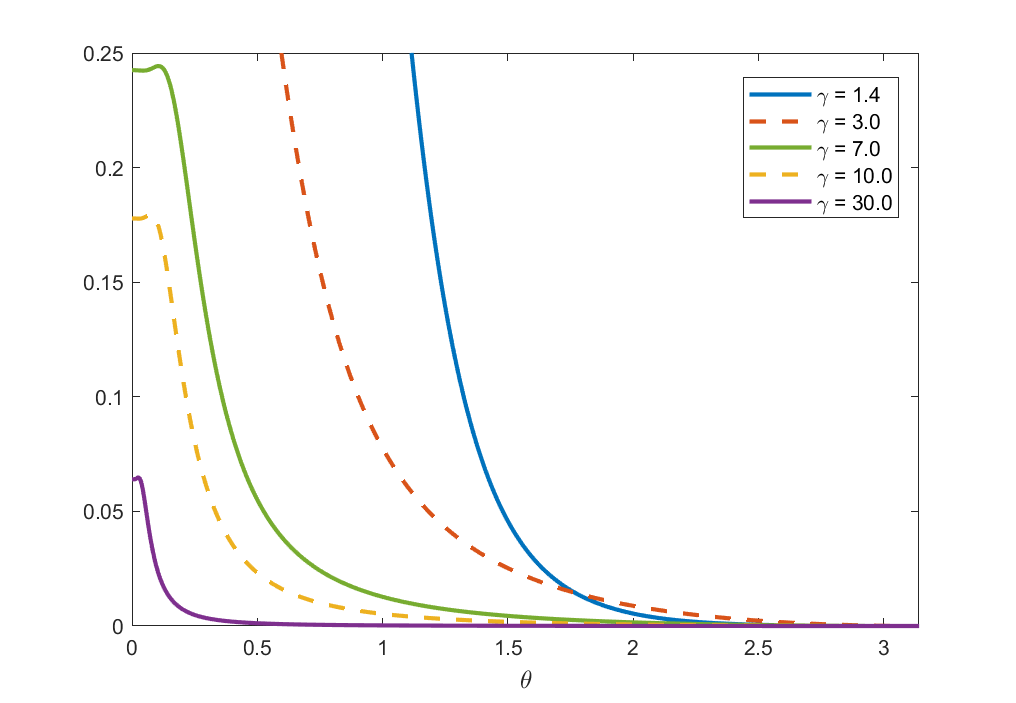}}
	\caption{Curves of $\widetilde{\mathcal{F}}_{\mathrm{S5}}(\gamma,\theta)$, $\widetilde{\mathcal{E}}_{\mathrm{S5}}(\gamma,\theta)$ and $\widetilde{\mathcal{I}}_{\mathrm{S5}}(\gamma,\theta)$ for $\theta\in[0,\pi]$.}
	\label{fig: lambda SIELM5}
\end{figure}

\subsection{SIELM6 scheme} 
The SIELM6 scheme has the discrete coefficient
\begin{align}\label{scheme: SIELM6}
	&\vec{a}_{\mathrm{S}}^{(6)}
	=\big(\gamma ^5+\tfrac{5 \gamma ^4}{2}-\tfrac{5 \gamma ^3}{3}+\tfrac{5 \gamma ^2}{6}-\tfrac{\gamma }{4}+\tfrac{1}{30},\\
	&\qquad
	-5 \gamma ^5-\tfrac{15 \gamma ^4}{2}+\tfrac{40 \gamma ^3}{3}-\tfrac{35 \gamma ^2}{6}+\tfrac{5 \gamma }{3}-\tfrac{13}{60},\notag\\
	&\qquad10 \gamma ^5+5 \gamma ^4-\tfrac{80 \gamma ^3}{3}+20 \gamma ^2-5 \gamma +\tfrac{37}{60},\notag\\
	&\qquad-10 \gamma ^5+5 \gamma ^4+\tfrac{50 \gamma ^3}{3}-\tfrac{70 \gamma ^2}{3}+10 \gamma -\tfrac{21}{20},\notag\\
	&\qquad5 \gamma ^5-\tfrac{15 \gamma ^4}{2}+\tfrac{5 \gamma ^3}{3}+\tfrac{35 \gamma ^2}{6}-\tfrac{65 \gamma }{12}+\tfrac{29}{20},\notag\\
	&\qquad-\gamma ^5+\tfrac{5 \gamma ^4}{2}-\tfrac{10 \gamma ^3}{3}+\tfrac{5 \gamma ^2}{2}-\gamma +\tfrac{1}{6}\big),\notag
\end{align}
while $\vec{b}_{\mathrm{S}}^{(6)}$ and $\vec{c}_{\mathrm{S}}^{(6)}$ can be generated by the 
second and third characteristic polynomials $\varrho_{b,\mathrm{S}}^{(6)}(\zeta):=\zeta(\gamma\zeta-\gamma+1)^{5}$ and $\varrho_{c,\mathrm{S}}^{(6)}(\zeta):=\zeta(\gamma\zeta-\gamma+1)^{5}-\gamma^{5}(\zeta-1)^{6}$,
respectively.
The formula \eqref{scheme: general imex Truncation Error} gives the leading error 
\begin{align}\label{truncation: SIELM6}
	R_{\mathrm{S}}^{(6,\gamma)}=\tfrac{-21 \gamma ^5+35 \gamma ^4-35 \gamma ^3+21 \gamma ^2-7 \gamma +1}{42} 	u_t^{(7)}(t_n)\tau^6+\gamma ^5  \mathcal{F}_t^{(6)}[u(t_n)]\tau^6.
\end{align}
By following the proof of Proposition \ref{prop: WBDF3}, one has the following result.
\begin{proposition}\label{prop: SIELM6}
	For the SIMLM6 scheme \eqref{scheme: SIELM6} with  $ 2\le\gamma\le80$, it holds that $\sigma_{\mathrm{F},\mathrm{S}}^{(6, \gamma)}=1,$
	\begin{align*}
		&\sigma_{\mathrm{E},\mathrm{S}}^{(6, \gamma)}=
		\frac{15(32 \gamma ^5+80 \gamma ^4-80 \gamma ^3+40 \gamma ^2-10 \gamma +1)}{16(30 \gamma ^5-50 \gamma ^3+50 \gamma ^2-20 \gamma +3)},\\
		&	\lambda_{\mathrm{I},\mathrm{S}}^{(6, \gamma)}=\frac{15 (2 \gamma -1)^5}{16 \left(30 \gamma ^5-50 \gamma ^3+50 \gamma ^2-20 \gamma +3\right)}\\
		&\text{such that}\quad\mathfrak{I}_{\mathrm{IE},\mathrm{S}}^{(6, \gamma)}=\frac{(2 \gamma -1)^5}{32 \gamma ^5+80 \gamma ^4-80 \gamma ^3+40 \gamma ^2-10 \gamma +1}.
	\end{align*}
\end{proposition}

As numerical illustrations for the claimed results in Proposition \ref{prop: SIELM6}, Figure \ref{fig: lambda SIELM6} depicts the following three auxiliary functions $\widetilde{\mathcal{F}}_{\mathrm{S6}}(\gamma,\theta)$, $\widetilde{\mathcal{E}}_{\mathrm{S6}}(\gamma,\theta)$ and $\widetilde{\mathcal{I}}_{\mathrm{S6}}(\gamma,\theta)$, defined similar to the functions $\widetilde{\mathcal{F}}_{\mathrm{S5}}(\gamma,\theta)$, $\widetilde{\mathcal{E}}_{\mathrm{S5}}(\gamma,\theta)$ and $\widetilde{\mathcal{I}}_{\mathrm{S5}}(\gamma,\theta)$, respectively, for $\theta\in[0,\pi]$ with  $\gamma=2,4,10,30$ and 80. 

\begin{figure}[htb!]
	\centering
	\subfigure[ $\widetilde{\mathcal{F}}_{\mathrm{S6}}(\gamma,\theta)$]
	{\includegraphics[width=1.67in]{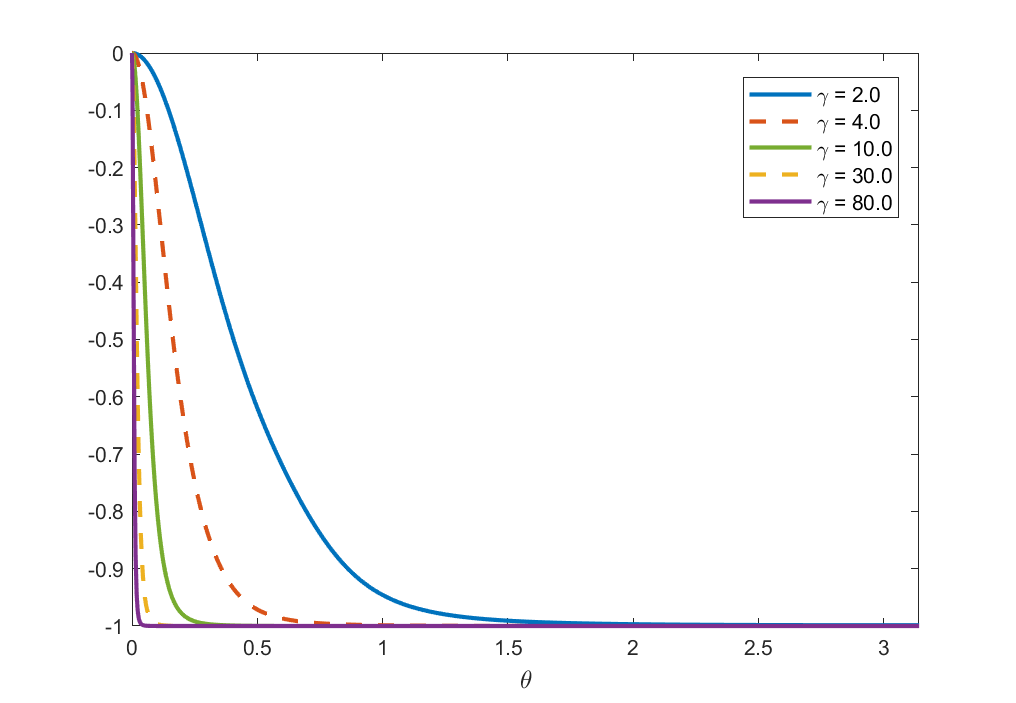}}
	\subfigure[ $\widetilde{\mathcal{E}}_{\mathrm{S6}}(\gamma,\theta)$]
	{\includegraphics[width=1.67in]{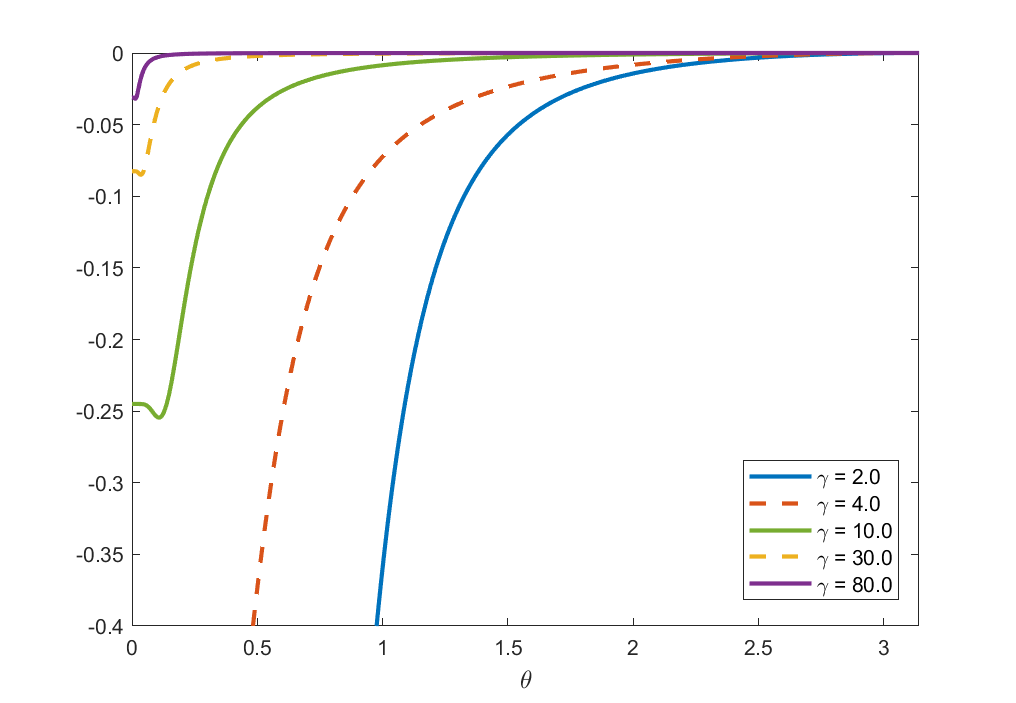}}
	\subfigure[ $\widetilde{\mathcal{I}}_{\mathrm{S6}}(\gamma,\theta)$]
	{\includegraphics[width=1.67in]{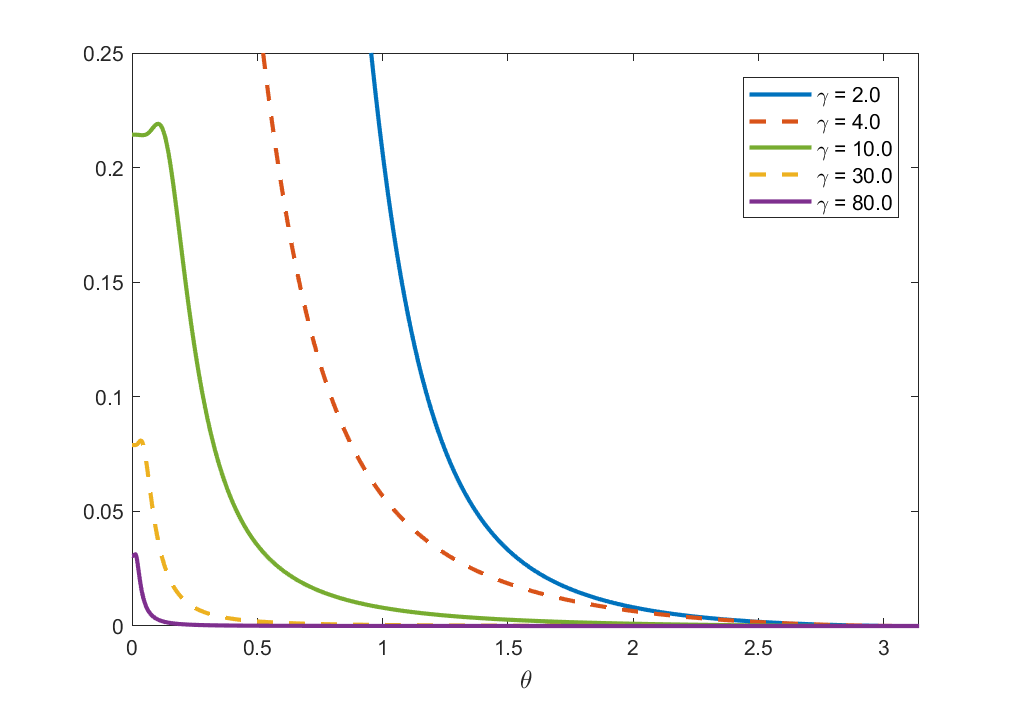}}
	\caption{Curves of $\widetilde{\mathcal{F}}_{\mathrm{S6}}(\gamma,\theta)$, $\widetilde{\mathcal{E}}_{\mathrm{S6}}(\gamma,\theta)$ and $\widetilde{\mathcal{I}}_{\mathrm{S6}}(\gamma,\theta)$ for $\theta\in[0,\pi]$.}
	\label{fig: lambda SIELM6}
\end{figure}

\subsection{SIELM7 scheme}   
The SIELM7 scheme has the discrete coefficient
\begin{align}\label{scheme: SIELM7}
	&\vec{a}_{\mathrm{S}}^{(7)}
	=\big(\gamma ^6+3 \gamma ^5-\tfrac{5 \gamma ^4}{2}+\tfrac{5 \gamma ^3}{3}-\tfrac{3 \gamma ^2}{4}+\tfrac{\gamma }{5}-\tfrac{1}{42},\\
	&\qquad-6 \gamma ^6-12 \gamma ^5+\tfrac{45 \gamma ^4}{2}-\tfrac{40 \gamma ^3}{3}+\tfrac{23 \gamma ^2}{4}-\tfrac{3 \gamma }{2}+\tfrac{37}{210},\notag\\
	&\qquad15 \gamma ^6+15 \gamma ^5-60 \gamma ^4+\tfrac{155 \gamma ^3}{3}-20 \gamma ^2+5 \gamma -\tfrac{241}{420},\notag\\
	&\qquad-20 \gamma ^6+65 \gamma ^4-\tfrac{260 \gamma ^3}{3}+45 \gamma ^2-10 \gamma +\tfrac{153}{140},\notag\\
	&\qquad 15 \gamma ^6-15 \gamma ^5-\tfrac{45 \gamma ^4}{2}+\tfrac{175 \gamma ^3}{3}-\tfrac{185 \gamma ^2}{4}+15 \gamma -\tfrac{197}{140},\notag\\
	&\qquad1-6 \gamma ^6+12 \gamma ^5-\tfrac{15 \gamma ^4}{2}-\tfrac{20 \gamma ^3}{3}+\tfrac{53 \gamma ^2}{4}-\tfrac{77 \gamma }{10}+\tfrac{223}{140},\notag\\
	&\qquad\gamma ^6-3 \gamma ^5+5 \gamma ^4-5 \gamma ^3+3 \gamma ^2-\gamma +\tfrac{1}{7}\big),\notag
\end{align}
while $\vec{b}_{\mathrm{S}}^{(7)}$ and $\vec{c}_{\mathrm{S}}^{(7)}$ can be generated by the 
second and third characteristic polynomials $\varrho_{b,\mathrm{S}}^{(7)}(\zeta):=\zeta(\gamma\zeta-\gamma+1)^{6}$ and $\varrho_{c,\mathrm{S}}^{(7)}(\zeta):=\zeta(\gamma\zeta-\gamma+1)^{6}-\gamma^{6}(\zeta-1)^{7}$,
respectively.
The formula \eqref{scheme: general imex Truncation Error} gives the leading error 
\begin{align}\label{truncation: SIELM7}
	R_{\mathrm{S}}^{(7,\gamma)}=-\tfrac{28 \gamma ^6-56 \gamma ^5+70 \gamma ^4-56 \gamma ^3+28 \gamma ^2-8 \gamma +1}{56}u_t^{(8)}(t_n)\tau^7+\gamma ^6  \mathcal{F}_t^{(7)}[u(t_n)]\tau^7.
\end{align}
By following the proof of Proposition \ref{prop: WBDF3}, one has the following result.
\begin{proposition}\label{prop: SIELM7}
	For the SIMLM7 scheme \eqref{scheme: SIELM7} with the free parameter $ 11/5\le\gamma\le60$, it holds that  
	$\sigma_{\mathrm{F},\mathrm{S}}^{(7, \gamma)}=1,$ $\sigma_{\mathrm{E},\mathrm{S}}^{(7,\gamma)}=\frac{\abst{c_{\mathrm{S}}^{(7)}(\pi)}}{\abst{a_{\mathrm{S}}^{(7)}(\pi)}},$
	\begin{align*}
		&\lambda_{\mathrm{I},\mathrm{S}}^{(7, \gamma)}=\frac{105 (1-2 \gamma )^6}{16 \left(420 \gamma ^6-1050 \gamma ^4+1400 \gamma ^3-840 \gamma ^2+252 \gamma -31\right)},
		\\
		&\text{such that}\quad
		\mathfrak{I}_{\mathrm{IE},\mathrm{S}}^{(7, \gamma)}=\frac{(1-2 \gamma )^6}{64 \gamma ^6+192 \gamma ^5-240 \gamma ^4+160 \gamma ^3-60 \gamma ^2+12 \gamma -1}.
	\end{align*}
\end{proposition}

As numerical illustrations for the claimed results in Proposition \ref{prop: SIELM7}, Figure \ref{fig: lambda SIELM7} depicts the following three auxiliary functions $\widetilde{\mathcal{F}}_{\mathrm{S7}}(\gamma,\theta)$, $\widetilde{\mathcal{E}}_{\mathrm{S7}}(\gamma,\theta)$ and $\widetilde{\mathcal{I}}_{\mathrm{S7}}(\gamma,\theta)$, defined similar to the functions $\widetilde{\mathcal{F}}_{\mathrm{S5}}(\gamma,\theta)$, $\widetilde{\mathcal{E}}_{\mathrm{S5}}(\gamma,\theta)$ and $\widetilde{\mathcal{I}}_{\mathrm{S5}}(\gamma,\theta)$, respectively, for $\theta\in[0,\pi]$ with  $\gamma=11/5,4,10,30$ and 60. 

\begin{figure}[htb!]
	\centering
	\subfigure[ $\widetilde{\mathcal{F}}_{\mathrm{S7}}(\gamma,\theta)$]
	{\includegraphics[width=1.67in]{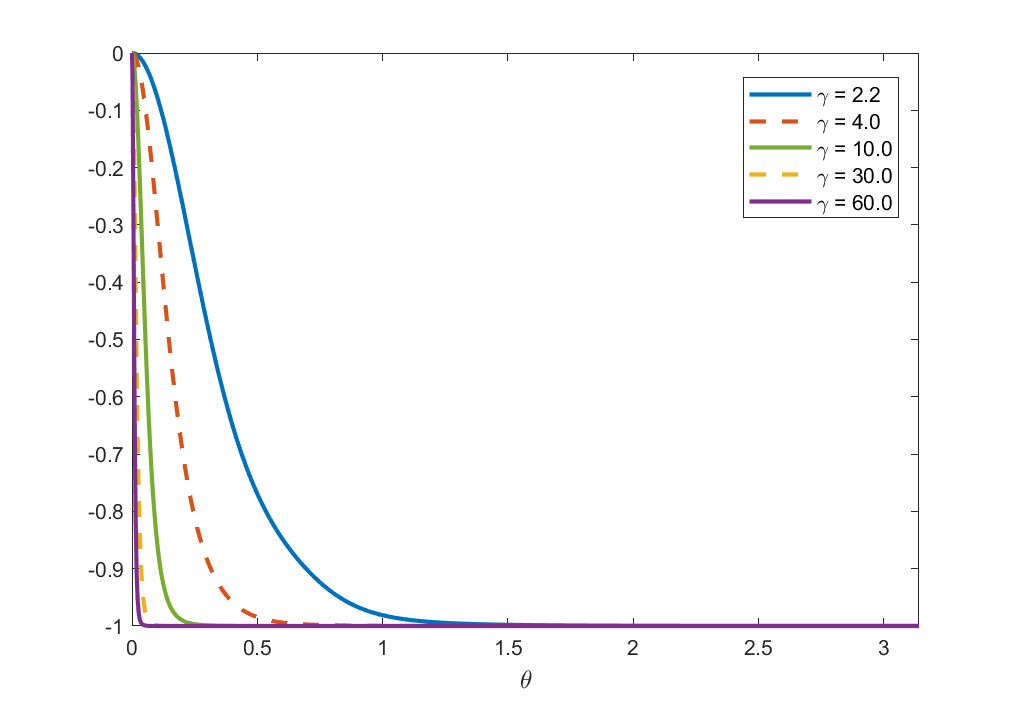}}
	\subfigure[ $\widetilde{\mathcal{E}}_{\mathrm{S7}}(\gamma,\theta)$]
	{\includegraphics[width=1.67in]{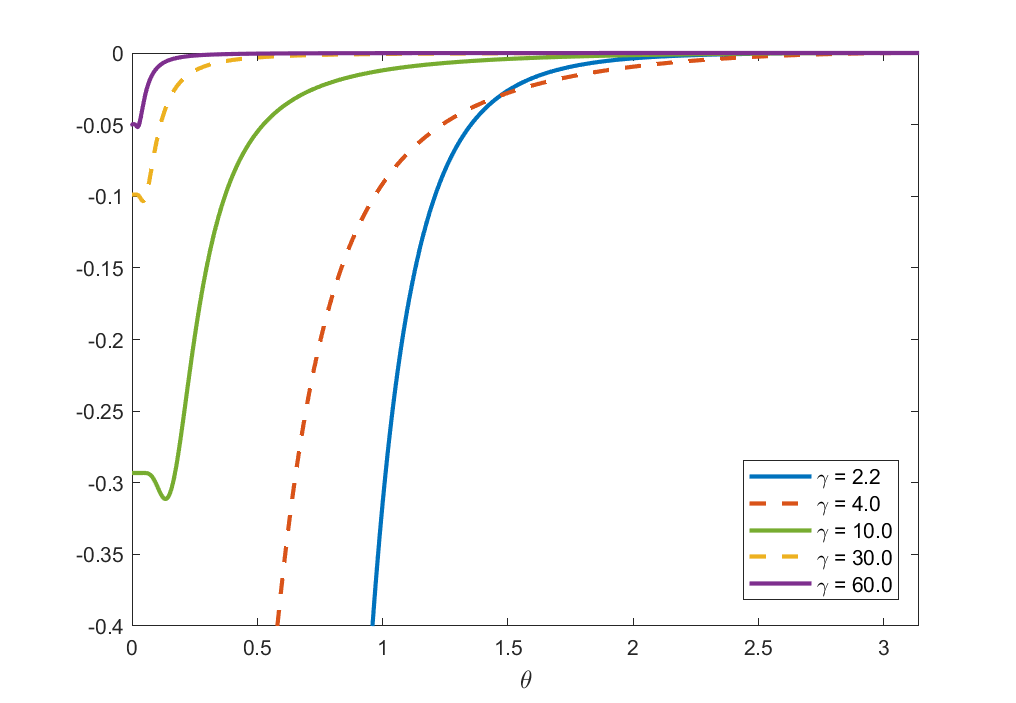}}
	\subfigure[ $\widetilde{\mathcal{I}}_{\mathrm{S7}}(\gamma,\theta)$]
	{\includegraphics[width=1.67in]{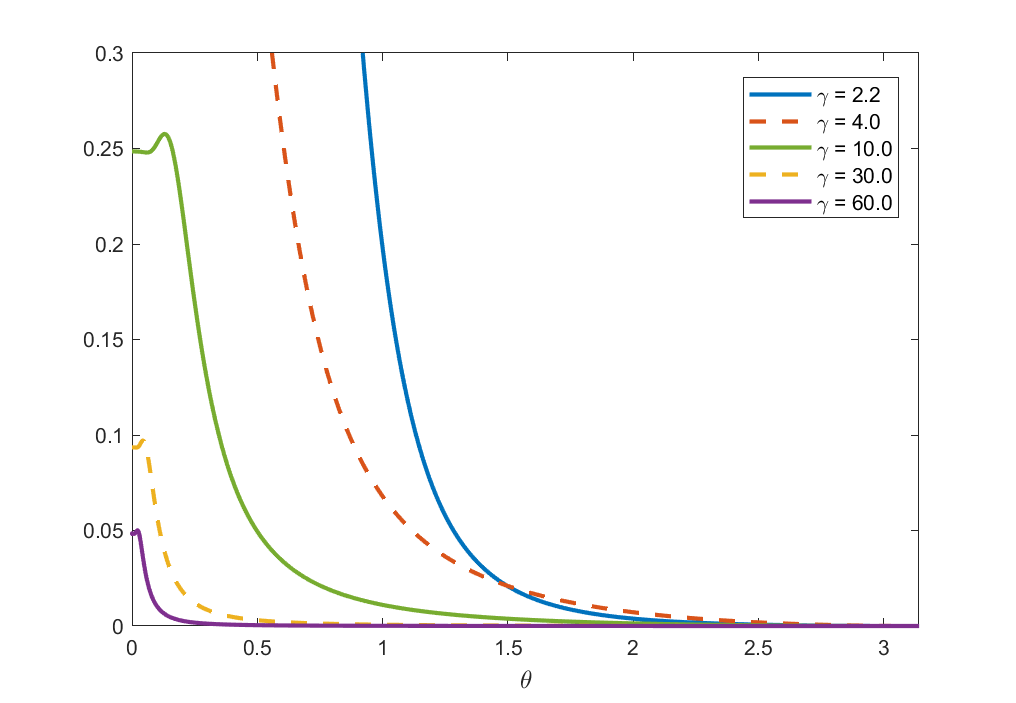}}
	\caption{Curves of $\widetilde{\mathcal{F}}_{\mathrm{S7}}(\gamma,\theta)$, $\widetilde{\mathcal{E}}_{\mathrm{S7}}(\gamma,\theta)$ and $\widetilde{\mathcal{I}}_{\mathrm{S7}}(\gamma,\theta)$ for $\theta\in[0,\pi]$.}
	\label{fig: lambda SIELM7}
\end{figure}

\subsection{SIELM8 scheme}   
The discrete coefficients  $\vec{a}_{\mathrm{S}}^{(8)}$, $\vec{b}_{\mathrm{S}}^{(8)}$ and $\vec{c}_{\mathrm{S}}^{(8)}$  of the SIELM8 scheme
can be generated by the three
characteristic polynomials 
$${\varrho}_{a,\mathrm{S}}^{(8)}(\zeta):=\sum_{j=1}^{9}\frac{f_{\mathrm{S}}^{(j)}(1)}{j!}(\zeta-1)^{j-1}\quad \text{with}\quad f_{\mathrm{S}}(z)=(\gamma z-\gamma+1)^{7}z\ln z,$$ 
$\varrho_{b,\mathrm{S}}^{(8)}(\zeta):=\zeta(\gamma\zeta-\gamma+1)^{7}$ and $\varrho_{c,\mathrm{S}}^{(8)}(\zeta):=\zeta(\gamma\zeta-\gamma+1)^{7}-\gamma^{7}(\zeta-1)^{8}$,
respectively.
The formula \eqref{scheme: general imex Truncation Error} gives the leading error 
\begin{align}\label{truncation: SIELM8}
	R_{\mathrm{S}}^{(8,\gamma)}=\tfrac{-36 \gamma ^7\!+84 \gamma ^6\!-126 \gamma ^5\!+126 \gamma ^4\!-84 \gamma ^3\!+36 \gamma ^2\!-9 \gamma +1}{72}u_t^{(9)}\!(t_n)\tau^8\!+\!\gamma ^7 \!\mathcal{F}_t^{(8)}\![u(t_n)]\tau^8.
\end{align}

\begin{figure}[htb!]
	\centering
	\subfigure[ $\widetilde{\mathcal{F}}_{\mathrm{S8}}(\gamma,\theta)$]
	{\includegraphics[width=1.67in]{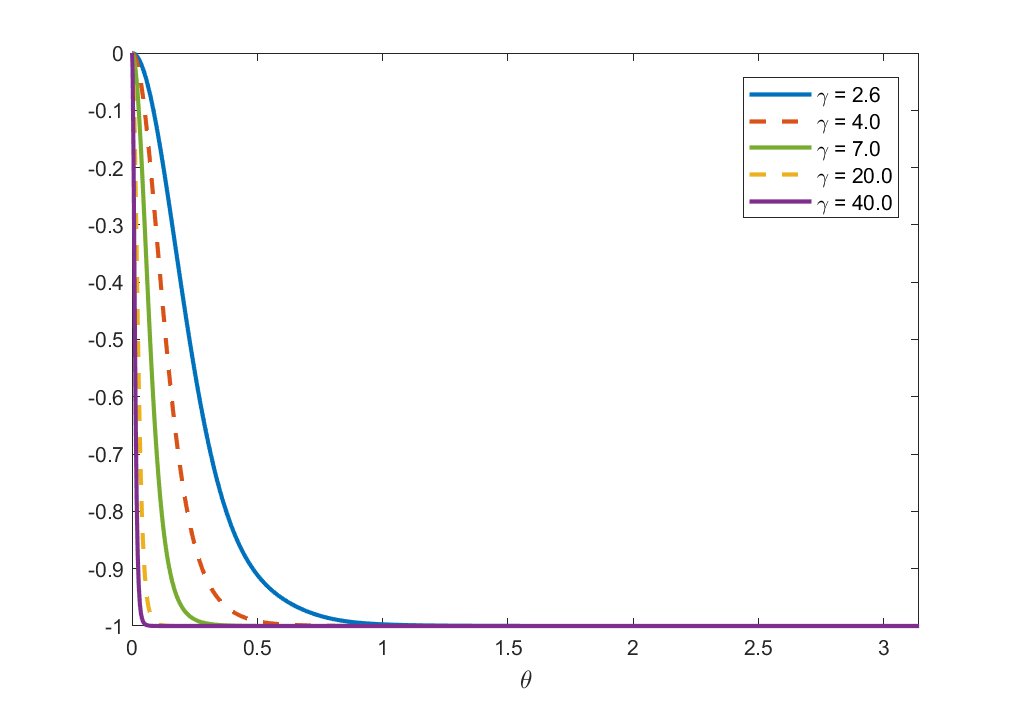}}
	\subfigure[ $\widetilde{\mathcal{E}}_{\mathrm{S8}}(\gamma,\theta)$]
	{\includegraphics[width=1.67in]{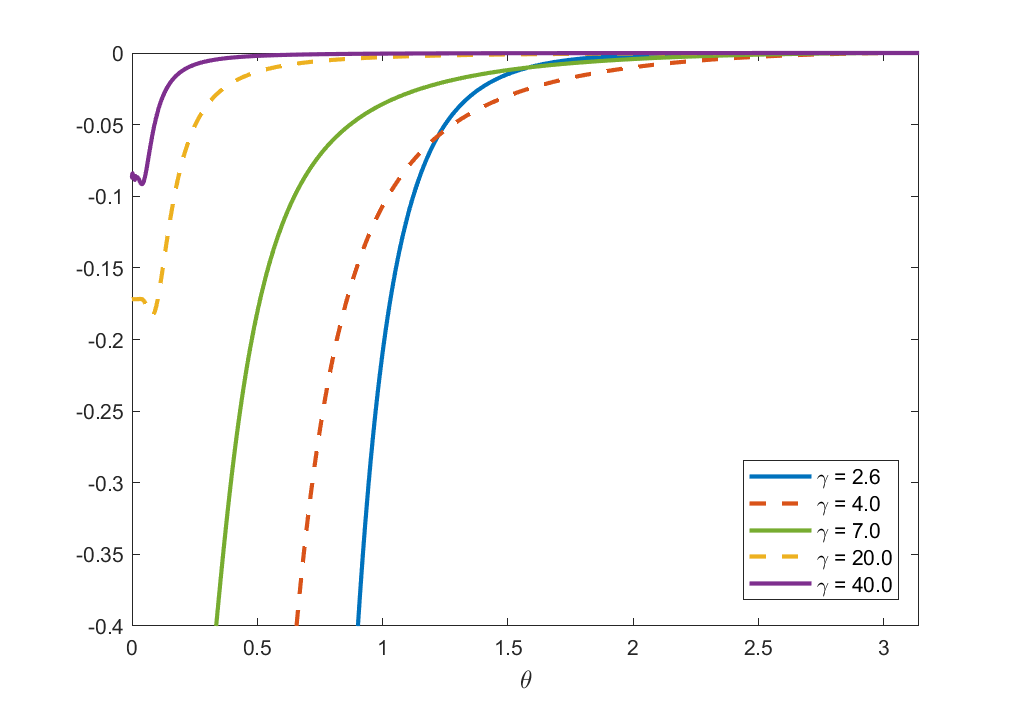}}
	\subfigure[ $\widetilde{\mathcal{I}}_{\mathrm{S8}}(\gamma,\theta)$]
	{\includegraphics[width=1.67in]{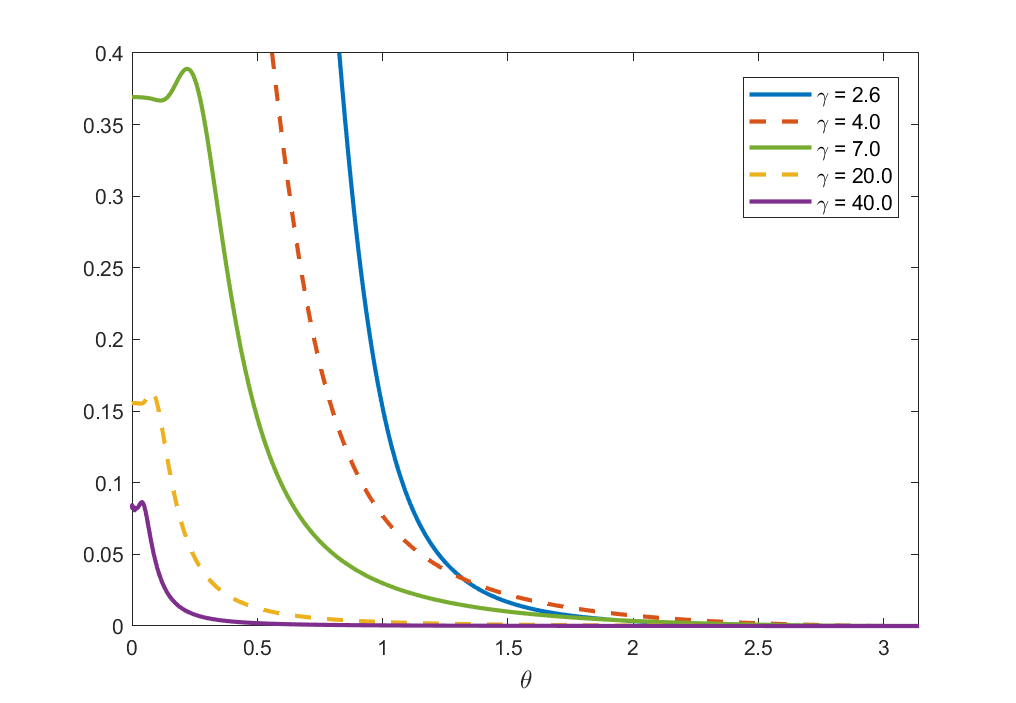}}
	\caption{Curves of $\widetilde{\mathcal{F}}_{\mathrm{S8}}(\gamma,\theta)$, $\widetilde{\mathcal{E}}_{\mathrm{S8}}(\gamma,\theta)$ and $\widetilde{\mathcal{I}}_{\mathrm{S8}}(\gamma,\theta)$ for $\theta\in[0,\pi]$.}
	\label{fig: lambda SIELM8}
\end{figure} 

By following the proof of Proposition \ref{prop: WBDF3}, one has the following result.
\begin{proposition}\label{prop: SIELM8}
	For the SIMLM8 scheme with  $ 13/5\le\gamma\le40$, it holds that $\sigma_{\mathrm{F},\mathrm{S}}^{(8, \gamma)}=1,$ $\sigma_{\mathrm{E},\mathrm{S}}^{(8,\gamma)}=\frac{\abst{c_{\mathrm{S}}^{(8)}(\pi)}}{\abst{a_{\mathrm{S}}^{(8)}(\pi)}},$
	\begin{align*}
		&	\lambda_{\mathrm{I},\mathrm{S}}^{(8, \gamma)}=\frac{105 (2 \gamma -1)^7}{32 \left(420 \gamma ^7-1470 \gamma ^5+2450 \gamma ^4-1960 \gamma ^3+882 \gamma ^2-217 \gamma +23\right)},
		\\
		&\text{such that}\quad
		\mathfrak{I}_{\mathrm{IE},\mathrm{S}}^{(8, \gamma)}=\frac{(2 \gamma -1)^7}{128 \gamma ^7+448 \gamma ^6-672 \gamma ^5+560 \gamma ^4-280 \gamma ^3+84 \gamma ^2-14 \gamma +1}.
	\end{align*}
\end{proposition}

As numerical illustrations for the claimed results in Proposition \ref{prop: SIELM8}, Figure \ref{fig: lambda SIELM8} depicts the following three auxiliary functions $\widetilde{\mathcal{F}}_{\mathrm{S8}}(\gamma,\theta)$, $\widetilde{\mathcal{E}}_{\mathrm{S8}}(\gamma,\theta)$ and $\widetilde{\mathcal{I}}_{\mathrm{S8}}(\gamma,\theta)$, defined similar to the functions $\widetilde{\mathcal{F}}_{\mathrm{S5}}(\gamma,\theta)$, $\widetilde{\mathcal{E}}_{\mathrm{S5}}(\gamma,\theta)$ and $\widetilde{\mathcal{I}}_{\mathrm{S5}}(\gamma,\theta)$, respectively, for $\theta\in[0,\pi]$ with  $\gamma=13/5,4,7,20$ and 40.

\subsection{SIELM9 scheme}   
The discrete coefficients  $\vec{a}_{\mathrm{S}}^{(9)}$, $\vec{b}_{\mathrm{S}}^{(9)}$ and $\vec{c}_{\mathrm{S}}^{(9)}$  of the SIELM9 scheme
can be generated by the three
characteristic polynomials 
$${\varrho}_{a,\mathrm{S}}^{(9)}(\zeta):=\sum_{j=1}^{9}\frac{f_{\mathrm{S}}^{(j)}(1)}{j!}(\zeta-1)^{j-1}\quad \text{with}\quad f_{\mathrm{S}}(z)=(\gamma z-\gamma+1)^{8}z\ln z,$$ 
$\varrho_{b,\mathrm{S}}^{(9)}(\zeta):=\zeta(\gamma\zeta-\gamma+1)^{8}$ and $\varrho_{c,\mathrm{S}}^{(9)}(\zeta):=\zeta(\gamma\zeta-\gamma+1)^{8}-\gamma^{8}(\zeta-1)^{9}$,
respectively.
The formula \eqref{scheme: general imex Truncation Error} gives the leading error 
\begin{align}\label{truncation: SIELM9}
	R_{\mathrm{S}}^{(9,\gamma)}=&\,\tfrac{-36 \gamma ^7+84 \gamma ^6-126 \gamma ^5+126 \gamma ^4-84 \gamma ^3+36 \gamma ^2-9 \gamma +1}{72}u_t^{(10)}(t_n)\tau^9\\
	&\,+\gamma ^8 \mathcal{F}_t^{(9)}[u(t_n)]\tau^9.\notag
\end{align}

\begin{figure}[htb!]
	\centering
	\subfigure[ $\widetilde{\mathcal{F}}_{\mathrm{S9}}(\gamma,\theta)$]
	{\includegraphics[width=1.67in]{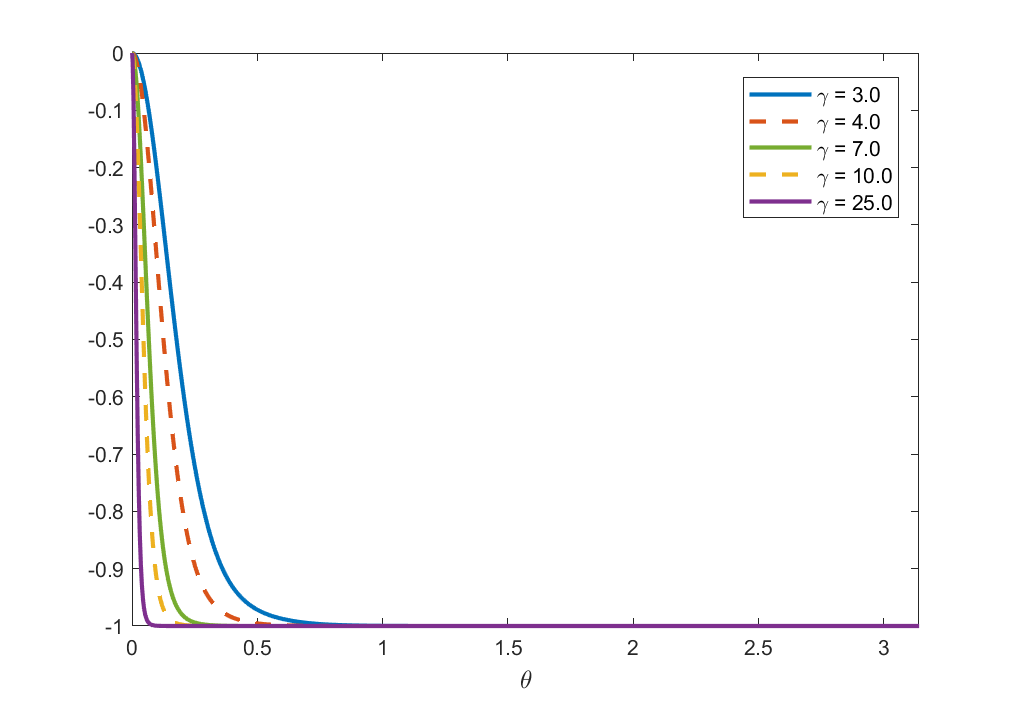}}
	\subfigure[ $\widetilde{\mathcal{E}}_{\mathrm{S9}}(\gamma,\theta)$]
	{\includegraphics[width=1.67in]{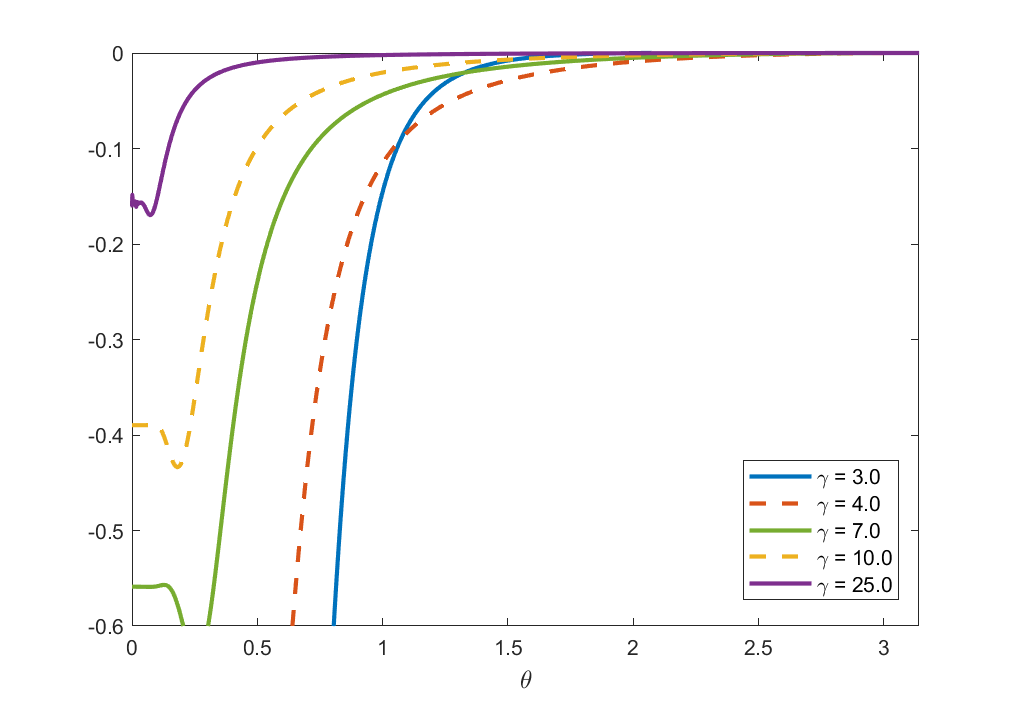}}
	\subfigure[ $\widetilde{\mathcal{I}}_{\mathrm{S9}}(\gamma,\theta)$]
	{\includegraphics[width=1.67in]{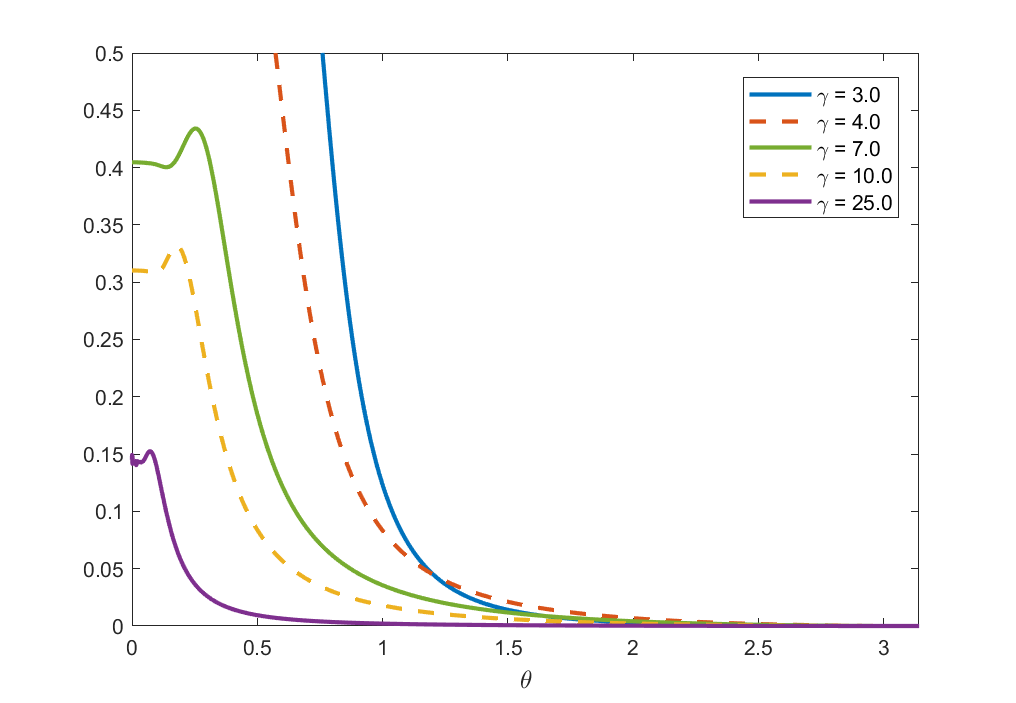}}
	\caption{Curves of $\widetilde{\mathcal{F}}_{\mathrm{S9}}(\gamma,\theta)$, $\widetilde{\mathcal{E}}_{\mathrm{S9}}(\gamma,\theta)$ and $\widetilde{\mathcal{I}}_{\mathrm{S9}}(\gamma,\theta)$ for $\theta\in[0,\pi]$.}
	\label{fig: lambda SIELM9}
\end{figure} 

By following the proof of Proposition \ref{prop: WBDF3}, one has the following result.
\begin{proposition}\label{prop: SIELM9}
	For the SIMLM9 scheme with the parameter $ 3\le\gamma\le25$, it holds that $\sigma_{\mathrm{F},\mathrm{S}}^{(9, \gamma)}=1,$ $\sigma_{\mathrm{E},\mathrm{S}}^{(9,\gamma)}
	=\frac{\abst{c_{\mathrm{S}}^{(9)}(\pi)}}{\abst{a_{\mathrm{S}}^{(9)}(\pi)}},$
	\begin{align*}
		&	\lambda_{\mathrm{I},\mathrm{S}}^{(9, \gamma)}=\frac{315 (1-2 \gamma )^8}{256 \left(315 \gamma ^8-1470 \gamma ^6+2940 \gamma ^5-2940 \gamma ^4+1764 \gamma ^3-651 \gamma ^2+138 \gamma -13\right)},
		\\
		&
		\mathfrak{I}_{\mathrm{IE},\mathrm{S}}^{(9, \gamma)}=\frac{(1-2\gamma)^8}{256\gamma^8 + 1024\gamma^7 - 1792\gamma^6 + 1792\gamma^5 - 1120\gamma^4 + 448\gamma^3 - 112\gamma^2 + 16\gamma - 1}.
	\end{align*}
\end{proposition}

As numerical illustrations for the claimed results in Proposition \ref{prop: SIELM9}, Figure \ref{fig: lambda SIELM9} depicts the following three auxiliary functions $\widetilde{\mathcal{F}}_{\mathrm{S9}}(\gamma,\theta)$, $\widetilde{\mathcal{E}}_{\mathrm{S9}}(\gamma,\theta)$ and $\widetilde{\mathcal{I}}_{\mathrm{S9}}(\gamma,\theta)$, defined similar to the functions $\widetilde{\mathcal{F}}_{\mathrm{S5}}(\gamma,\theta)$, $\widetilde{\mathcal{E}}_{\mathrm{S5}}(\gamma,\theta)$ and $\widetilde{\mathcal{I}}_{\mathrm{S5}}(\gamma,\theta)$, respectively, for $\theta\in[0,\pi]$ with  $\gamma=3,4,7,10$ and 25.

\end{document}